\documentclass[accepted]{uai2025} 

\usepackage[american]{babel}

\usepackage{amsmath,amsthm} 
\newtheorem{definition}{Definition}
\newtheorem{assumption}{Assumption}
\newtheorem{observation}{Observation}
\newtheorem{corollary}{Corollary}
\newtheorem{proposition}{Proposition}
\newtheorem{theorem}{Theorem}
\newtheorem{remark}{Remark}
\newtheorem{lemma}{Lemma}

\def\mb{\mathbb}

\usepackage{xcolor} 
\usepackage{hyperref} 
\hypersetup{
    colorlinks,
    linkcolor={red!50!black},
    citecolor={blue!50!black},
    urlcolor={blue!80!black}
}
\makeatletter
\renewcommand{\tagform@}[1]{\textup{(\textcolor{red!50!black}{#1})}}
\makeatother
\newcommand{\indicate}[1]{\mathds{1}\small{[#1]}}
\newcommand*\diff{\mathop{}\!\mathrm{d}}

\usepackage{algorithm}
\usepackage{algorithmic}
\usepackage{booktabs}       
\usepackage{multirow} 

\usepackage{tikz}
\usetikzlibrary{arrows}
\usetikzlibrary{bending}
\usetikzlibrary{decorations.pathreplacing}
\usetikzlibrary{shapes}
\usetikzlibrary{patterns}
\usepackage{amsfonts,amsmath,mathtools}       
\usepackage{bm}
\usepackage{dsfont}
\usepackage{enumitem}
\usepackage[labelfont={color=red!50!black}]{caption}

\usepackage[round]{natbib}

\usepackage[normalem]{ulem} 
\newcommand\redout{\bgroup\markoverwith{\textcolor{red}{\rule[.5ex]{2pt}{0.4pt}}}\ULon}

\newcommand{\cmnt}[1]{\ignorespaces}

\usepackage{mathtools} 

\title{Distributionally and Adversarially Robust Logistic Regression via Intersecting Wasserstein Balls}

\author[1]{\href{mailto:<a.selvi19@imperial.ac.uk>?Subject=Your UAI 2025 paper}{Aras Selvi}{}}
\author[2]{\href{mailto:<eleonora.kreacic@jpmorgan.com>?Subject=Your UAI 2025 paper}{Eleonora Krea{\v{c}}i{\'c}{}}{}}
\author[2]{Mohsen Ghassemi}
\author[2]{Vamsi K. Potluru}
\author[2]{Tucker Balch}
\author[2]{Manuela Veloso}
\affil[1]{%
    Imperial Business School, Imperial College London
}
\affil[2]{%
    JP Morgan AI Research
}

\begin{document}
\maketitle

\begin{abstract}
  Adversarially robust optimization (ARO) has emerged as the \emph{de facto} standard for training models that hedge against adversarial attacks in the test stage. While these models are robust against adversarial attacks, they tend to suffer severely from overfitting. To address this issue, some successful methods replace the empirical distribution in the training stage with alternatives including \emph{(i)} a worst-case distribution residing in an ambiguity set, resulting in a distributionally robust (DR) counterpart of ARO; \emph{(ii)} a mixture of the empirical distribution with a distribution induced by an auxiliary (\emph{e.g.}, synthetic, external, out-of-domain) dataset. Inspired by the former, we study the Wasserstein DR counterpart of ARO for logistic regression and show it admits a tractable convex optimization reformulation. Adopting the latter setting, we revise the DR approach by intersecting its ambiguity set with another ambiguity set built using the auxiliary dataset, which offers a significant improvement whenever the Wasserstein distance between the data generating and auxiliary distributions can be estimated. We study the underlying optimization problem, develop efficient solution algorithms, and demonstrate that the proposed method outperforms benchmark approaches on standard datasets. 
\end{abstract}
\section{INTRODUCTION}
Supervised learning traditionally involves access to a training dataset whose instances are assumed to be independently sampled from a true data-generating distribution~\citep{bishop2006pattern, hastie2009elements}. Optimizing an expected loss for the empirical distribution constructed from such a training set, also known as \textit{empirical risk minimization} (ERM), enjoys several desirable properties in relatively generic settings, including convergence to the true risk minimization problem as the number of training samples increases~\citep[Chapter 2]{vapnik1999nature}. In real-world applications, however, various challenges, such as data scarcity and the existence of adversarial attacks, lead to deteriorated out-of-sample performance for models trained via ERM.

One of the key limitations of ERM, particularly as it is designed to minimize an expected loss for the empirical distribution, emerges from the finite nature of data in practice. This leads ERM to suffer from the `optimism bias', also known as overfitting~\citep{murphy2022}, or the optimizer's curse~\citep{demiguel2009portfolio,smith2006optimizer}, causing deteriorated out-of-sample performance. A popular approach to prevent this phenomenon, \textit{distributionally robust optimization} (DRO; \citealt{DY10:distr_rob_opt}), optimizes the expected loss for the worst-case distribution residing within a pre-specified ambiguity set.

Another key challenge faced by ERM in practice is adversarial attacks, where an adversary perturbs the observed features during the testing or deployment phase~\citep{advattackpaper,goodfellow2014explaining}, also known as evasion corruption at test time~\citep{biggio2013evasion}. For neural networks, the paradigm of \textit{adversarial training} (AT; \citealt{madry2017towards}) is thus designed to provide adversarial robustness by simulating such attacks in the training stage. Several successful variants of AT, specialized to different losses and attacks, have been proposed in the literature to achieve adversarial robustness without significantly reducing performance on training sets~\citep{shafahi2019adversarial,zhang2019theoretically,gao2019convergence,PangLYZY22}. Some studies \citep{uesato2018adversarial,carlini2019evaluating,wu2020adversarial} investigate the adversarial robustness guarantees of various training algorithms, leading to a research direction focused on heuristic improvements to such models (\textit{e.g.}, \citealt{rade2022reducing}). Our work aligns with another recent line of research \citep{xing2022artificially,bennouna2023certified} on \textit{adversarially robust optimization} (ARO), which constrains ERM to guarantee an \emph{exact}, pre-specified level of adversarial robustness while maximizing training accuracy.

Recently, it has been observed that the two aforementioned notions of robustness can be at odds, as adversarially robust (AR) models suffer from severe overfitting (\textit{robust overfitting}; \citealt{raghunathan2019adversarial,yu2022understanding,li2023understanding}). Indeed, it is observed that robust overfitting is even more severe than traditional overfitting~\citep{rice2020overfitting}. To this end, some works address robust overfitting by revisiting AT algorithms and adding adjustments for better generalization~\citep{chen2020robust,li2023clean}. In a recent work,~\citet[Thm 3.2]{bennouna2023certified} decompose the error gap of robust overfitting into the statistical error of estimating the true data-generating distribution via the empirical distribution and an adversarial error resulting from the adversarial attacks, hence proposing the simultaneous adoption of DRO and ARO.

In this work, we study logistic regression (LR) for binary classification that is adversarially robust against $\ell_p$-attacks~\citep{croce2020robustbench}. To address robust overfitting faced by the adversarially robust LR model, we employ a DRO approach where distributional ambiguity is modeled with the type-1 Wasserstein metric. We base our work on an observation that the worst-case logistic loss under adversarial attacks can be represented as a Lipschitz continuous and convex loss function. This allows us to use existing Wasserstein DRO machinery for Lipschitz losses, and derive an \textit{exact} reformulation of the Wasserstein DR counterpart of adversarially robust LR as a tractable convex problem. 

Our main contribution lies in reducing the size of the Wasserstein ambiguity set in the DRO problem mentioned above, in order to create a less conservative problem while preserving the same distributional robustness guarantees. To accomplish this, we draw inspiration from recent work on ARO that leverages auxiliary datasets~(\textit{e.g.}, \citealt{gowal2021improving,xing2022artificially}) and revise our DRO problem by intersecting its ambiguity set with another ambiguity set constructed using an auxiliary dataset. Examples of auxiliary data include synthetic data generated from a generative model (\textit{e.g.}, privacy-preserving data release), data in the presence of distributional shifts (\textit{e.g.}, different time periods/regions), noisy data (\textit{e.g.}, measurement errors), or out-of-domain data (\textit{e.g.}, different source); any auxiliary dataset is viable as long as its instances are sampled independently from an underlying data-generating distribution whose Wasserstein distance to the true data-generating distribution is known or can be estimated. Figure~\ref{figure:balls} illustrates our framework.

\begin{figure}[!t]
    \centering
    \resizebox{1\linewidth}{!}{
    \begin{tikzpicture}
    \draw[dashed, line width=1.5pt] (0,0) circle (2);
    \draw[decorate, decoration={brace, amplitude=10pt, mirror}, line width = 0.8pt] (-2, 2) -- (-2,-2) node[midway, left=12pt] {\LARGE $\mathfrak{B}_{\varepsilon}(\mathbb{P}_N)$};

    \draw[dashed, line width=1.5pt, color = red] (3,-1) circle (2.5);
    \draw[decorate, decoration={brace, amplitude=10pt}, line width = 0.8pt, color = red] (5.5, 1.5) -- (5.5,-3.5) node[midway, right=12pt] {{\LARGE $\mathfrak{B}_{\widehat{\varepsilon}}(\widehat{\mathbb{P}}_{\widehat{N}})$}};    
    \begin{scope}
      \clip (0,0) circle (2);
      \fill[blue!30] (3,-1) circle (2.5);
    \end{scope}

    \fill (0,0) circle (2pt) node (pointA) {};
    
    \node at (0,0.3) {\large $\mathbb{P}_N$};

    \fill (1,-1.2) circle (2pt) node (pointB) {};
    \node at (0.9,-1.5) {\large $\mathbb{P}^0$};

    \fill (3,-1) circle (2pt) node (pointC) {};
    \node at (3.4,-0.8) {\large $\widehat{\mathbb{P}}_{\widehat{N}}$};

    \fill (2.2,-0.9) circle (2pt) node (pointD) {};
    \node at (2.2,-0.55) {\large $\widehat{\mathbb{P}}$};

    \fill (1.2,-0.2) circle (2pt) node (pointE) {};
    \node at (1.4,0.1) {\large $\mathbb{Q}_{\mathrm{mix}}$};
    
    
    \draw[line width = 1pt,color=darkgray] (pointA) -- node[midway, sloped, below, color = darkgray] { \textbf{(1)}} (pointB);
    \draw[line width = 1pt,color=darkgray] (pointC) -- node[midway, sloped, below, color = darkgray] { \textbf{(2)}} (pointD);
    \draw[line width = 1pt,color=darkgray] (pointD) -- node[midway, sloped, below, color = darkgray] { \textbf{(3)}} (pointB);

    \draw[<->, line width =0.8pt, color = darkgray] (pointA) -- (-1.9,0) node[midway, above] {\large $\varepsilon$};
    \draw[<->, line width =0.8pt, color = darkgray] (pointC) -- (3,-3.5) node[midway, right] {\large $\widehat{\varepsilon}$};
\end{tikzpicture}}
\caption{\textit{Traditional ARO optimizes the expected adversarial loss over the empirical distribution $\mathbb{P}_N$ constructed from $N$ i.i.d.\@ samples of the (unknown) true data-generating distribution $\mathbb{P}^0$. Replacing $\mathbb{P}_N$ with a worst-case distribution in the ball $\mathfrak{B}_{\varepsilon}(\mathbb{P}_N)$ gives us its DR counterpart. To reduce the size of this ball, we intersect it with another ball $\mathfrak{B}_{\widehat{\varepsilon}}(\widehat{P}_{\widehat{N}})$ while ensuring $\mathbb{P}^0$ is still included with high confidence. The latter ball is centered at an empirical distribution $\widehat{\mathbb{P}}_{N}$ constructed from $\widehat{N}$ i.i.d.\@ samples of some auxiliary distribution $\widehat{\mathbb{P}}$. Recent works using auxiliary data in ARO propose optimizing the expected adversarial loss over a mixture $\mathbb{Q}_\mathrm{mix}$ of $\mathbb{P}_N$ and $\widehat{P}_{\widehat{N}}$; we show that this distribution resides in $\mathfrak{B}_{\varepsilon}({\mathbb{P}}_N) \cap \mathfrak{B}_{\widehat \varepsilon}({\widehat{\mathbb{P}}}_{\widehat{N}})$ under some conditions.}}
    \label{figure:balls}
\end{figure}

The paper unfolds as follows. In Section~\ref{section:related}, we review related literature on DRO and ARO, with a focus on their interactions. We examine the use of auxiliary data in ARO and the intersection of Wasserstein balls in DRO. We discuss open questions for LR to motivate our loss function choice in this work. Section~\ref{section:prelim} gives preliminaries on ERM, ARO, and type-1 Wasserstein DRO. In Section~\ref{sectionDRO}, we discuss that the adversarial logistic loss can be reformulated as a Lipschitz convex function, enabling the use of Wasserstein DRO machinery for Lipschitz losses. Our main contribution (\textit{cf.}~Figure~\ref{figure:balls}) is in Section~\ref{sect_reducing_conserv}, where we provide an explicit reformulation of the distributionally and adversarially robust LR problem over the intersection of two Wasserstein balls, prove the NP-hardness of this problem, and derive a convex relaxation of it. Our work is mainly on \textit{optimization} where we focus on how to solve the underlying problems upon cross-validating Wasserstein ball radii, however, in Section~\ref{sec:radii} we discuss some preliminary statistical approaches to set such radii. We close the paper with numerical experiments on standard benchmark datasets in Section~\ref{section:experiments}. We borrow the standard notation in DR machine learning, which is elaborated on in our Appendices. 

\section{RELATED WORK} \label{section:related}
\textbf{Auxiliary data in ARO.} The use of auxiliary data appears in the ARO literature. In particular, it is shown that additional unlabeled data sampled from the same~\citep{carmon2022unlabeled,xing2022unlabeled} or different~\citep{deng2021improving} data-generating distributions could provide adversarial robustness. \cite{sehwag2022robust} show that adversarial robustness can be certified even when it is provided for a synthetic dataset as long as the distance between its generator and the true data-generating distribution can be quantified. \cite{gowal2021improving} and \cite{xing2022artificially} propose optimizing a weighted combination of ARO over empirical and synthetic datasets. We show that the latter approach can be recovered by our model. 

\textbf{DRO-ARO interactions.} In this work, we optimize ARO against worst-case data-generating distributions residing in an ambiguity set, where the \mbox{type-1} Wasserstein metric is used for distances since it is arguably the most common choice in machine learning (ML) with Lipschitz losses~\citep{shafieezadeh2019regularization,gao2023finite}. In the literature, it is shown that standard ARO is equivalent to the DRO of the original loss function with a type-$\infty$ Wasserstein metric~\citep{staib2017distributionally,khim2018adversarial,pydi2021many,regniez2022a,frank2024existence}. In other words, in the absence of adversarial attacks, training models adversarially with artificial attacks provide some distributional robustness. Hence, our DR ARO approach can be interpreted as optimizing the logistic loss over the worst-case distribution whose $1$-Wasserstein distance is bounded by a pre-specified radius from at least one distribution residing in an $\infty$-Wasserstein ball around the empirical distribution. Conversely, \cite{sinha2017certifying} discuss that while DRO over Wasserstein balls is intractable for generic losses (\textit{e.g.}, neural networks), its Lagrange relaxation resembles ARO and thus ARO yields a certain degree of (relaxed) distributional robustness~\citep{wu2020adversarial,bui2022unified,phan2023global}. Such literature suggests that, when there is no concern about the statistical errors caused by using empirical distributions (\textit{e.g.}, in very high-data regimes), one can train DR models to obtain adversarial robustness guarantees. However, as discussed by~\cite{bennouna2022holistic}, when statistical errors exist, then we need to be simultaneously robust against adversarial attacks and statistical errors. To the best of our knowledge, there have not been works optimizing a pre-specified level of type-$1$ Wasserstein distributional robustness (that hedges against overfitting, \citealt{kuhn2019wasserstein}) and adversarial robustness (that hedges against adversarial attacks, \citealt{goodfellow2014explaining}) \textit{simultaneously}. To our knowledge, the only approach that considers the exact DR counterpart of ARO is proposed by~\cite{bennouna2023certified}, who model distributional ambiguity with $\varphi$-divergences for neural networks.

\textbf{Intersecting ambiguity sets in DRO.} Recent work started to explore the intersection of ambiguity sets for different contexts~\citep{datajoin,wang2024contextual} or different metrics \citep{intersectballs}. Our idea of intersecting Wasserstein balls is originated from the ``Surround, then Intersect'' strategy~\citep[\S 5.2]{taskesen2021sequential} to train linear regression under sequential domain adaptation in a non-adversarial setting (see the work of~\citealt{shafahi2020adversarially} and~\citealt{song2018improving} for robustness in domain adaptation/transfer learning). The aforementioned work focuses on the squared loss function with an ambiguity set using the Wasserstein metric developed for the first and second distributional moments. In a recent study,~\cite{rychener2024wasserstein} generalize most of the previous results and prove that DRO problems over the intersection of two Wasserstein balls admit tractable convex reformulations whenever the loss function is the maximum of concave functions. They also discuss why distributions lying in the intersection of two Wasserstein balls are more natural candidates for the unknown true distribution than those that are Wasserstein barycenters or mixture distributions of the empirical and auxiliary distributions (referred to as heterogeneous data sources; see Example 1, Proposition 2, and Corollary 1).

\textbf{Logistic loss in DRO and ARO.} Our choice of LR aligns with the current directions and open questions in the related literature. In the DRO literature, even in the absence of adversarial attacks, the aforementioned work of~\cite{taskesen2021sequential} on the intersection of Wasserstein ambiguity sets is restricted to linear regression. The authors show that this problem admits a tractable convex optimization reformulation, and their proof relies on the properties of the squared loss. Similarly,~\cite{rychener2024wasserstein} discuss that the logistic loss fails to satisfy the piece-wise concavity assumption and is inherently difficult to optimize over the intersection of Wasserstein balls. We contribute to the DRO literature for adversarial and non-adversarial settings because we show that such a problem would be NP-hard for the logistic loss even without adversarial attacks, and develop specialized approximation techniques. Our problem recovers DR LR~\citep{NIPS2015,selvi2022wasserstein} as a special case in the absence of adversarial attacks and auxiliary data. Answering theoretical challenges posed by logistic regression has been useful in answering more general questions in the DRO literature, such as DR LR~\citep{NIPS2015} leading to DR ML~\citep{shafieezadeh2019regularization} and mixed-feature DR LR~\citep{selvi2022wasserstein} leading to mixed-feature DR Lipschitz ML~\citep{belbasi2023s}. Finally, in the (non-DR) ARO literature, there are recent theory developments on understanding the effect of auxiliary data (\textit{e.g.},~\citealt{xing2022artificially}) specifically for squared and logistic loss functions. 

\paragraph{Single-step adversarial training and single-shot ARO.}
Our work proposes a single-shot convex optimization procedure to train logistic models that are both adversarially and distributionally robust.
Although the terminology may resemble the recent work on \emph{single-step adversarial training} for neural networks \citep{wong2020fast, lin2023eliminating, lin2024layer}, the two approaches operate differently. Single-step adversarial training generates adversarial perturbations using a single gradient computation at each iteration of iterative model training and improves robustness through updates, with performance typically assessed at intermediate checkpoints. In contrast, our method solves a convex optimization problem once to obtain a model that satisfies both forms of robustness by design. To enable tractability, this approach leverages the convexity and Lipschitz continuity of the loss function under adversarial attacks, which hold for the logistic loss function. While single-step adversarial training applies broadly to general classes of models such as neural networks, our framework offers a complementary, optimization-based perspective in the logistic regression model, where structural properties can be fully exploited.

\section{PRELIMINARIES} \label{section:prelim}
We consider a binary classification problem where an instance is modeled as $(\bm{x}, y) \in \Xi := \mathbb{R}^n \times \{-1, +1\}$
and the labels depend on the features via $\mathrm{Prob}[y \mid \bm{x}] \; = \; [1 + \exp(-y \cdot \bm{\beta}^\top \bm{x})]^{-1}$ for some $\bm{\beta} \in \mathbb{R}^n$; its associated loss is the \textit{logloss} $\ell_{\bm{\beta}}(\bm{x}, y) := \log( 1 + \exp{(-y \cdot \bm{\beta}^\top \bm{x})})$. 

\textbf{Empirical risk minimization.} Let $\mathcal{P}(\Xi)$ denote the set of distributions supported on $\Xi$ and $\mathbb{P}^0 \in \mathcal{P}(\Xi)$ denote the true data-generating distribution. One wants to minimize the expected logloss over $\mathbb{P}^0$, that is
\begin{align}\label{ideal}\tag{RM}
    \begin{array}{cl}
        \displaystyle \underset{\bm{\beta}\in \mathbb{R}^n}{\inf} \; \displaystyle \mathbb{E}_{\mathbb{P}^0} [\ell_{\bm{\beta}}(\bm{x}, y)]. 
    \end{array}
\end{align}
In practice, $\mathbb{P}^0$ is hardly ever known, and one resorts to the empirical distribution ${\mathbb{P}}_N = \frac{1}{N} \sum_{i \in [N]} \delta_{\bm{\xi}^i}$ where $\bm{\xi}^i = (\bm{x}^i, y^i), \  i \in [N]$, are i.i.d.\@ samples from $\mathbb{P}^0$ and $\delta_{\bm{\xi}}$
denotes the Dirac distribution supported on $\bm{\xi}$.
The empirical risk minimization (ERM) problem is thus
\begin{align}\label{erm}\tag{ERM}
    \begin{array}{cl}
    \displaystyle \underset{\bm{\beta}\in \mathbb{R}^n}{\inf} \;\mathbb{E}_{{\mathbb{P}}_{N}} [\ell_{\bm{\beta}}(\bm{x}, y)].
    \end{array}
\end{align}

\textbf{Distributionally robust optimization.} To be able to define a distance between distributions, we first define the following feature-label metric on $\Xi$.
\begin{definition}\label{def:feature-label}
    The distance between instances $\bm{\xi} = (\bm{x}, y) \in \Xi$ and $\bm{\xi'} = (\bm{x'}, y') \in \Xi$ for $\kappa \geq 0$ and $q \geq 1$ is
    \begin{align*}
        d(\bm{\xi}, \bm{\xi'}) = \lVert \bm{x} - \bm{x'} \rVert_{q} + \kappa \cdot \indicate{y \neq y'}.
    \end{align*}
\end{definition}
Using this metric, we define the Wasserstein distance.
\begin{definition}\label{def:Wasser}
    The type-1 Wasserstein distance between distributions $\mathbb{Q}, \mathbb{Q}' \in \mathcal{P}(\Xi)$ is defined as
    \begin{align*}
        \mathrm{W}(\mathbb{Q}, \mathbb{Q}') =  \underset{\Pi \in \mathcal{C}(\mathbb{Q},\mathbb{Q}')}{\inf} \left\{ \int_{\Xi \times \Xi} 
        d(\bm{\xi}, \bm{\xi'}) \Pi(\diff \bm{\xi}, \diff \bm{\xi'}) \right\},
    \end{align*}
    where $\mathcal{C}(\mathbb{Q},\mathbb{Q}')$ is the set of couplings of $\mathbb{Q}$ and $\mathbb{Q}'$.
\end{definition}
In finite-data settings, the distance between the true data-generating distribution and the empirical distribution is upper-bounded by some $\epsilon>0$. The Wasserstein DRO problem is thus defined as
\begin{align}\label{dro}\tag{DRO}
    \begin{array}{cl}
        \displaystyle \underset{\bm{\beta}\in \mathbb{R}^n}{\inf} \; \displaystyle \underset{\mathbb{Q} \in \mathfrak{B}_{\varepsilon}({\mathbb{P}}_N)}{\sup} \; \mathbb{E}_{\mathbb{Q}} [\ell_{\bm{\beta}}(\bm{x}, y)], 
    \end{array}
\end{align}
where $\mathfrak{B}_{\varepsilon}(\mathbb{P}) := \{\mathbb{Q} \in \mathcal{P}(\Xi) : \mathrm{W}(\mathbb{Q}, \mathbb{P}) \leq \varepsilon \}$ denotes the Wasserstein ball centered at $\mathbb{P} \in \mathcal{P}(\Xi)$ with radius $\varepsilon$. We refer to \cite{med17} and \cite{kuhn2019wasserstein} for the properties of~\ref{dro} and estimating $\varepsilon$.

\paragraph{Adversarially robust optimization.}
    The goal of adversarial robustness is to provide robustness against adversarial attacks~\citep{goodfellow2014explaining}. An adversarial attack, in the widely studied $\ell_p$-noise setting \citep{croce2020robustbench}, perturbs the features of the test instances $(\bm{x}, y)$ by adding additive noise $\bm{z}$ to $\bm{x}$. The adversary chooses the noise vector $\bm{z}$, subject to $\lVert \bm{z} \rVert_{p} \leq \alpha$, so as to maximize the loss $\ell_{\bm{\beta}}(\bm{x}+ \bm{z}, y)$ associated with this perturbed test instance. Therefore, ARO solves the following optimization problem in the training stage to hedge against adversarial perturbations at the test stage:
\begin{align}\label{adv}\tag{ARO}
    \begin{array}{cl}
        \displaystyle \underset{\bm{\beta}\in \mathbb{R}^n}{\inf} \; \displaystyle \mathbb{E}_{{\mathbb{P}}_{N}} [\underset{\bm{z}: \lVert \bm{z} \rVert_p \leq \alpha}{\sup} \{\ell_{\bm{\beta}}(\bm{x} + \bm{z}, y)\} ] .
    \end{array}
\end{align}
\ref{adv} reduces to~\ref{erm} when $\alpha = 0$. Note that~\ref{adv} is identical to feature robust training~\citep{bertsimas2019robust} which is not motivated by adversarial attacks, but by the presence of noisy observations in the training set~\citep{BTEGN09:rob_opt,gorissen2015practical}.

\textbf{DRO-ARO connection.} A connection between ARO and DRO is noted in the literature~(\citealt[Proposition 3.1]{staib2017distributionally}, \citealt[Lemma 22]{khim2018adversarial}, \citealt[Lemma 5.1]{pydi2021many}, \citealt[Proposition 2.1]{regniez2022a}, \citealt[Lemma 3]{frank2024existence}, and \citealt[\S 3]{bennouna2023certified}). Namely, problem~\ref{adv} is equivalent to a DRO problem
\begin{align}\label{literature_problem}
    \begin{array}{cl}
        \displaystyle \underset{\bm{\beta}\in \mathbb{R}^n}{\inf} \; \displaystyle \underset{\mathbb{Q} \in \mathfrak{B}^\infty_{\alpha}({\mathbb{P}}_N)}{\sup} \; \mathbb{E}_{\mathbb{Q}} [\ell_{\bm{\beta}}(\bm{x}, y)],
    \end{array}
\end{align}
where the ambiguity set $\mathfrak{B}^\infty_{\alpha}({\mathbb{P}}_N)$ is a type-$\infty$ Wasserstein ball~\citep{givens1984class} with radius $\alpha$. Hence, in non-adversarial settings,~\ref{adv} provides robustness with respect to the type-$\infty$ Wasserstein distance. In the case of adversarial attacks, it suffers from robust overfitting as discussed earlier. To address this issue, one straightforward approach is to revisit~\eqref{literature_problem} and replace $\alpha$ with some $\alpha' > \alpha$. This approach, however, does not provide improvements for the out-of-sample performance since \textit{(i)} the type-$\infty$ Wasserstein distance employed in problem~\eqref{literature_problem} uses a metric on the feature space, ignoring labels; \textit{(ii)} \mbox{type-$\infty$} Wasserstein distances do not provide strong out-of-sample performances in ML (unlike, \textit{e.g.}, the type-$1$ Wasserstein distance) since the required radii to provide meaningful robustness guarantees are typically too large \citep[\S 1.2.2, and references therein]{bennouna2022holistic}. We thus study the type-$1$ Wasserstein counterpart of~\ref{adv}, which we initiate in the next section.

\section{DISTRIBUTIONALLY AND ADVERSARIALLY ROBUST LR}\label{sectionDRO}
Here we derive the Wasserstein DR counterpart of~\ref{adv} that will set the ground for our main result in the next section. We impose the following assumption.
\begin{assumption}\label{ass:DRO}
    We are given a finite $\varepsilon > 0$ value satisfying $\mathrm{W}(\mathbb{P}^0, \mathbb{P}_N) \leq \varepsilon$.
\end{assumption}
The assumption implies that we know an $\varepsilon > 0$ value satisfying $\mathbb{P}^0 \in \mathfrak{B}_{\varepsilon}(\mathbb{P}_N)$. Typically, however, $\varepsilon$ is either estimated through cross-validation or finite sample statistics, with the assumption then regarded as holding with high confidence (see \S\ref{sec:radii} for a review of related results we can borrow). The distributionally and adversarially robust LR problem is thus:
\begin{align}\label{advdro}\tag{DR-ARO}
\mspace{-4mu}
    \underset{\bm{\beta}\in \mathbb{R}^n}{\inf} \underset{\mathbb{Q} \in \mathfrak{B}_{\varepsilon}({\mathbb{P}}_N)}{\sup} \mathbb{E}_{\mathbb{Q}} [\sup_{\bm{z}: \lVert \bm{z} \rVert_p \leq \alpha} \{\ell_{\bm{\beta}}(\bm{x} + \bm{z}, y)\} ].
\end{align}
By employing a simple duality trick for the inner $\sup$-problem, as commonly applied in robust optimization~\citep{BTEGN09:rob_opt,BdH22:rob_opt}, we can represent~\ref{advdro} as a standard non-adversarial DRO problem with an updated loss function, which we name the \textit{adversarial loss}.

\begin{observation}\label{obs:prelim}
    Problem~\textup{\ref{advdro}} is equivalent to
    \begin{align*}
        \begin{array}{cl}
        \underset{\bm{\beta}\in \mathbb{R}^n}{\inf} \underset{\mathbb{Q} \in \mathfrak{B}_{\varepsilon}({\mathbb{P}}_N)}{\sup} \mathbb{E}_{\mathbb{Q}}  [\ell^{\alpha}_{\bm{\beta}}(\bm{x}, y)],
        \end{array}
    \end{align*}
    where the \emph{adversarial loss} $\ell^{\alpha}_{\bm{\beta}}$ is defined as
    \begin{align*}
    \ell^{\alpha}_{\bm{\beta}}(\bm{x}, y) := \log(1 + \exp(  - y\cdot \bm{\beta}^\top \bm{x} + \alpha \cdot \lVert \bm{\beta} \rVert_{p^\star} )),
    \end{align*}
    for $p^\star$ satisfying $1/p + 1/p^\star = 1$. The univariate representation $L^{\alpha}(z) := \log(1 + \exp(- z +  \alpha \cdot \lVert \bm{\beta} \rVert_{p^\star} ))$ of $\ell^{\alpha}_{\bm{\beta}}$ is convex and has a Lipschitz modulus of $1$.
\end{observation}
As a corollary of Observation~\ref{obs:prelim}, we can directly employ the techniques proposed by~\cite{shafieezadeh2019regularization} to dualize the inner $\sup$-problem of~\ref{advdro} and obtain a tractable reformulation.
\begin{corollary}\label{corr:tractable_og}
    Problem \textup{\ref{advdro}} admits the following tractable convex optimization reformulation:
    \begin{align*}
        \begin{array}{cll}
        \underset{\substack{\bm{\beta}, \lambda, \bm{s}}}{\inf}   & \displaystyle \varepsilon \lambda + \dfrac{1}{N} \sum_{i=1}^N s_i  \\[1.2em]
        \mathrm{s.t.}  & \displaystyle \ell_{\bm{\beta}}^\alpha(\bm{x}^i, y^i) \leq s_i & \forall i \in [N] \\[0.6em]
        & \displaystyle \ell_{\bm{\beta}}^\alpha(\bm{x}^i, -y^i) - \lambda \kappa \leq s_i & \forall i \in [N] \\[0.6em]
        & \displaystyle \lVert \bm{\beta} \rVert_{q^\star} \leq \lambda \\[0.6em]
        & \displaystyle \bm{\beta} \in \mathbb{R}^n, \; \lambda \geq 0, \; \bm{s} \in  \mathbb{R}^N_{+},
        \end{array}
    \end{align*}
    for $q^\star$ satisfying $1/q + 1/q^\star = 1$.
\end{corollary}
The constraints of this problem are exponential cone representable (derivation is in the appendices) and for $q \in \{1,2,\infty\}$, the yielding problem can be solved with the exponential cone solver of MOSEK~\citep{mosek2016modeling} in polynomial time~\citep{N18:convex}. 

\section{MAIN RESULT}\label{sect_reducing_conserv}
In \S\ref{sectionDRO} we discussed the traditional DRO setting where we have access to an empirical distribution $\mathbb{P}_N$ constructed from $N$ i.i.d.\@ samples of the true data-generating distribution $\mathbb{P}^0$, and we are given (or we estimate) some $\varepsilon$ so that $\mathbb{P}^0 \in \mathfrak{B}_{\varepsilon}({\mathbb{P}}_N)$. Recently in DRO literature, it became a key focus to study the case where we have access to an additional auxiliary empirical distribution $\widehat{\mathbb{P}}_{\widehat{N}}$ constructed from $\widehat{N}$ i.i.d.\@ samples $\widehat{\bm{\xi}}^j = (\widehat{\bm{x}}^j, \widehat{y}^j), \ j \in [\widehat{N}]$, of some other distribution $\widehat{\mathbb{P}}$; given the increasing availability of useful auxiliary data in the ARO domain, we explore this direction here. We start with the following assumption.
\begin{assumption}\label{ass:strong}
    We are given finite $\varepsilon, \widehat{\varepsilon} > 0$ values satisfying $\mathrm{W}(\mathbb{P}^0, \mathbb{P}_N) \leq \varepsilon$ and $\mathrm{W}(\mathbb{P}^0, \widehat{\mathbb{P}}_{\widehat{N}}) \leq \widehat{\varepsilon}$.
\end{assumption}

The assumption implies that we know $\varepsilon, \widehat{\varepsilon} > 0$ values satisfying $\mathbb{P}^0 \in \mathfrak{B}_{\varepsilon}(\mathbb{P}_N) \cap \mathfrak{B}_{\widehat{\varepsilon}}(\widehat{\mathbb{P}}_{\widehat{N}})$. In practice, this assumption is ensured to hold with high confidence by estimating the $\varepsilon$ and $\widehat{\varepsilon}$ values; methods across various domains which we can adopt are reviewed in \S\ref{sec:radii}. We want to optimize the adversarial loss over the intersection $\mathfrak{B}_{\varepsilon}(\mathbb{P}_N) \cap \mathfrak{B}_{\widehat{\varepsilon}}(\widehat{\mathbb{P}}_{\widehat{N}})$:
\begin{align}\label{synth}\tag{Inter-ARO}
    \underset{\bm{\beta} \in \mathbb{R}^n}{\mathrm{inf}} \underset{\mathbb{Q} \in \mathfrak{B}_{\varepsilon}(\mathbb{P}_{N}) \cap \mathfrak{B}_{\widehat{\varepsilon}}(\widehat{\mathbb{P}}_{\widehat{N}})}{\sup} \mathbb{E}_{\mathbb{Q}}[\ell^{\alpha}_{\bm{\beta}} (\bm{x}, y)].
\end{align}
Note that Assumption~\ref{ass:strong} implies that $\varepsilon$ and $\widehat{\varepsilon}$ guarantee that $\mathfrak{B}_{\varepsilon}(\mathbb{P}_N) \cap \mathfrak{B}_{\widehat{\varepsilon}}(\widehat{\mathbb{P}}_{\widehat{N}})$ is nonempty, and problem~\ref{synth} is thus feasible. This problem is expected to outperform~\ref{advdro} as the ambiguity set is smaller while still including $\mathbb{P}^0$. However, problem~\ref{synth} is challenging to solve even in the absence of adversarial attacks ($\alpha = 0$) as we reviewed in~\S\ref{section:related}. To address this challenge, we first reformulate~\ref{synth} as a semi-infinite optimization problem with finitely many variables.
\begin{proposition}\label{prop:reformulate}
    \textup{\ref{synth}} is equivalent to:
    \begin{align*}
        \begin{array}{cl}
        \underset{\substack{\bm{\beta}, \lambda, \widehat{\lambda} \\ \bm{s}, \widehat{\bm{s}}}}{\inf}     & \displaystyle \varepsilon \lambda + \widehat{\varepsilon} \widehat{\lambda} + \dfrac{1}{N} \sum_{i=1}^N s_i + \dfrac{1}{\widehat{N}} \sum_{j=1}^{\widehat{N}} \widehat{s}_j \\
        \mathrm{s.t.} & \underset{\bm{x} \in \mathbb{R}^n}{\sup} \{\ell^{\alpha}_{\bm{\beta}}(\bm{x}, l) - \lambda \lVert \bm{x}^i - \bm{x} \rVert_q  - \widehat{\lambda} \lVert \widehat{\bm{x}}^j - \bm{x} \rVert_q\}  \\
        & \hfill \leq s_i + \dfrac{\kappa(1 - l y^i)}{2} \lambda + \widehat{s}_j + \dfrac{\kappa(1 -  l \widehat{y}^j)}{2} \widehat{\lambda} \\
        & \hfill \forall (i,j, l) \in [N] \times [\widehat{N}] \times \{ -1, 1 \}\\[1mm]
        & \bm{\beta} \in \mathbb{R}^n, \; \lambda \geq 0, \; \widehat{\lambda} \geq 0, \; \bm{s} \in \mathbb{R}^N_{+}, \; \widehat{\bm{s}} \in \mathbb{R}^{\widehat{N}}_{+}.
        \end{array}
    \end{align*}
\end{proposition}
Even though this problem recovers the tractable problem~\ref{advdro} as $\widehat{\varepsilon} \rightarrow \infty$, it is NP-hard in the finite radius settings. We reformulate~\ref{synth} as an adjustable robust optimization problem \citep{ben2004adjustable,yanikouglu2019survey}, and borrow tools from this literature to obtain the following result.
\begin{proposition}\label{prop:np-hard}
    \textup{\ref{synth}} is equivalent to an adjustable RO problem with $\mathcal{O}(N\cdot \widehat{N})$ two-stage robust constraints, which is NP-hard even when $N = \widehat{N} = 1$.
\end{proposition}
The adjustable RO literature has developed a rich arsenal of relaxations that can be leveraged for~\ref{synth}.
We adopt the `static relaxation technique'~\citep{bertsimas2015tight} to restrict the feasible region of~\ref{synth} and obtain a tractable approximation.
\begin{theorem}[main]\label{thm:main}
    The following convex optimization problem is a feasible relaxation of~\textup{\ref{synth}}:
    {\allowdisplaybreaks
    \begin{align}\label{safe}\tag{\textup{Inter-ARO$^\star$}}
        \begin{array}{cl}
        \underset{\substack{\bm{\beta}, \lambda, \widehat{\lambda}, \bm{s}, \widehat{\bm{s}} \\ \bm{z}_{ij}^+, \bm{z}_{ij}^-}}{\inf} & \displaystyle \varepsilon \lambda + \widehat{\varepsilon} \widehat{\lambda} + \frac{1}{N} \sum_{i=1}^N s_i + \frac{1}{\widehat{N}} \sum_{j=1}^{\widehat{N}} \widehat{s}_j \\
        \mathrm{s.t.} & L^{\alpha}(l\cdot\bm{\beta}^\top \bm{x}^i +  \bm{z}_{ij}^{l\top}(\widehat{\bm{x}}^j - \bm{x}^i)) \\
        & \hfill \leq  s_i + \dfrac{\kappa(1 - ly^i)}{2} \lambda + \widehat{s}_j + \dfrac{\kappa(1 - l\widehat{y}^j)}{2} \widehat{\lambda} \\
        & \hfill \forall (i,j, l) \in [N] \times [\widehat{N}] \times \{ -1, 1 \}\\[0.8em]
        & \lVert l\bm{\beta} - \bm{z}^{l}_{ij} \rVert_{q^\star} \leq \lambda \\
        & \hfill \forall (i,j, l) \in [N] \times [\widehat{N}] \times \{ -1, 1 \}\\[0.8em]
        &\lVert \bm{z}^{l}_{ij} \rVert_{q^\star} \leq \widehat{\lambda}  \\
        & \hfill \forall (i,j, l) \in [N] \times [\widehat{N}] \times \{ -1, 1 \}\\[0.8em]
        & \bm{\beta} \in \mathbb{R}^n, \; \lambda \geq 0, \; \widehat{\lambda} \geq 0, \; \bm{s} \in \mathbb{R}^N_{+}, \; \widehat{\bm{s}} \in \mathbb{R}^{\widehat{N}}_{+}\\[0.8em]
        & \bm{z}^{l}_{ij} \in \mathbb{R}^n, \  (i,j,l) \in [N] \times [\widehat{N}] \times \{ -1, 1\}.
        \end{array}
    \end{align}
    }
\end{theorem}
Similarly to~\ref{advdro}, the constraints of~\ref{safe} are exponential cone representable (\textit{cf.}~appendices).

Recall that for $\widehat{\varepsilon}$ large enough,~\ref{synth} reduces to~\ref{advdro}. The following corollary shows that, despite~\ref{safe} being a relaxation of~\ref{synth}, a similar property holds. That is, ``not learning anything from auxiliary data'' remains feasible: the static relaxation does not force learning from $\widehat{\mathbb{P}}_{\widehat{N}}$, and it learns from auxiliary data only if the objective improves.

\begin{corollary}\label{corr:relax}
    \underline{\emph{Feasibility of ignoring auxiliary data:}} Any feasible solution $(\bm{\beta}, \lambda, \bm{s})$ of~\textup{\ref{advdro}} can be used to recover a feasible solution $(\bm{\beta}, \lambda, \widehat{\lambda}, \bm{s}, \widehat{\bm{s}}, \bm{z}^{+}_{ij}, \bm{z}^{-}_{ij})$ for~\ref{safe} with $\widehat{\lambda} = 0$, $\widehat{\bm{s}} = \mathbf{0}$, and $\bm{z}^{+}_{ij} = \bm{z}^{-}_{ij} = \mathbf{0}$. \\
    \underline{\emph{Convergence to~\ref{synth}}}: The optimal value of \ref{safe} converges to the optimal value of~\textup{\ref{synth}}, with the same set of $\bm\beta$ solutions, as $\widehat\varepsilon \rightarrow \infty$.
\end{corollary}

In light of Corollary~\ref{corr:relax}, Appendix~\ref{app_uai_data_sect} discusses that some simulations in our numerical experiments chose not to incorporate the auxiliary data by setting a sufficiently large $\widehat{\varepsilon}$. We close the section by discussing how~\ref{synth} can recover some problems in the DRO and ARO literature. Firstly, recall that~\ref{synth} can ignore the auxiliary data once $\widehat\varepsilon$ is set large enough, reducing this problem to~\ref{advdro}. Moreover, notice that $\alpha = 0$ reduces $\ell_{\bm{\beta}}^\alpha$ to $\ell_{\bm{\beta}}$, hence for $\alpha = 0$ and $\widehat{\varepsilon} = \infty$~\ref{synth} recovers the Wasserstein LR model of~\cite{NIPS2015}.
We next relate~\ref{synth} to the problems in the ARO literature that use auxiliary data. 
The works in this  literature~\citep{gowal2021improving,xing2022artificially} solve the following
\begin{align}\label{synth_literature}
    \mspace{-10mu}
    \begin{array}{r}
        \displaystyle \underset{\bm{\beta} \in \mathbb{R}^n}{\inf} \   \dfrac{1}{N + w \widehat{N}} \big[ \sum_{i \in [N]}  \underset{\bm{z}^i \in \mathcal{B}_{p}(\alpha)}{\sup} \{\ell_{\bm{\beta}}(\bm{x}^i + \bm{z}^i, y^i)\}  +  \\
        \displaystyle w \sum_{j \in [\widehat N]}  \underset{\bm{z}^j \in {\mathcal{B}_p(\alpha)}}{\sup} \{\ell_{\bm{\beta}}(\bm{\widehat x}^j + \bm{z}^j, \widehat{y}^j)\} \big],
    \end{array}
\end{align} 
for some $w > 0$, where $ \mathcal{B}_p(\alpha) := \{ \bm{z} \in \mathbb{R}^n : \lVert \bm{z} \rVert_p \leq \alpha\}$. We observe that~\eqref{synth_literature} resembles a variant of~\ref{adv} that replaces the empirical distribution $\mathbb{P}_N$ with its mixture with $\widehat{\mathbb{P}}_{\widehat{N}}$:
\begin{observation}\label{obs:literature}
    Problem~\eqref{synth_literature} is equivalent to
    \begin{align}\label{eq:with_mixture}
    \begin{array}{cl}
        \displaystyle \underset{\bm{\beta} \in \mathbb{R}^n}{\inf} &  \displaystyle \mathbb{E}_{\mathbb{Q}_{\mathrm{mix}}} [\ell^{\alpha}_{\bm{\beta}}(\bm{x}, y)]
    \end{array}
\end{align}
where $\mathbb{Q}_{\mathrm{mix}} := \lambda \cdot \mathbb{P}_N +  (1-\lambda)\cdot \widehat{\mathbb{P}}_{\widehat{N}}$ for $\lambda = \frac{N}{N + w \widehat{N}}$.
\end{observation}
We give a condition on $\varepsilon$ and $\widehat{\varepsilon}$ to guarantee that the mixture distribution introduced in Proposition~\ref{obs:literature} lives in $\mathfrak{B}_{\varepsilon}({\mathbb{P}}_N) \cap \mathfrak{B}_{\widehat \varepsilon}({\widehat{\mathbb{P}}}_{\widehat{N}})$, that is, the distribution $\mathbb{Q}_{\mathrm{mix}}$ will be feasible in the $\sup$ problem of~\ref{synth}.
\begin{proposition}\label{prop:included}
    For any $\lambda \in (0,1)$ and $\mathbb{Q}_{\mathrm{mix}} = \lambda \cdot \mathbb{P}_N +  (1-\lambda)\cdot \widehat{\mathbb{P}}_{\widehat{N}}$, we have $\mathbb{Q}_{\mathrm{mix}} \in \mathfrak{B}_{\varepsilon}({\mathbb{P}}_N) \cap \mathfrak{B}_{\widehat \varepsilon}({\widehat{\mathbb{P}}}_{\widehat{N}})$ whenever 
    $\varepsilon + \widehat{\varepsilon} \geq \mathrm{W}(\mathbb{P}_N, \widehat{\mathbb{P}}_{\widehat{N}})$ and $\frac{\widehat{\varepsilon}}{\varepsilon} = \frac{\lambda}{1 - \lambda}$.
\end{proposition}
For $\lambda = \frac{N}{N + \widehat{N}}$, if the intersection $\mathfrak{B}_{\varepsilon}({\mathbb{P}}_N) \cap \mathfrak{B}_{\widehat \varepsilon}({\widehat{\mathbb{P}}}_{\widehat{N}})$ is nonempty, Proposition~\ref{prop:included} implies that a sufficient condition for this intersection to include $\mathbb{Q}_{\mathrm{mix}}$ is $\widehat{\varepsilon} / \varepsilon = N / \widehat{N}$, which is intuitive since the radii of Wasserstein ambiguity sets are chosen inversely proportional to the number of samples~\citep[Theorem 18]{kuhn2019wasserstein}.
\section{SETTING WASSERSTEIN RADII}\label{sec:radii}
Thus far, we have assumed knowledge of DRO ball radii $\varepsilon$ and $\widehat\varepsilon$ that satisfy Assumptions \ref{ass:DRO} and \ref{ass:strong}. In this section, we employ Wasserstein finite-sample statistics techniques to estimate these values.
 
\textbf{Setting $\epsilon$ for \ref{advdro}.} 
In the following theorem, we present tight characterizations for $\varepsilon$ so that the ball $\mathfrak{B}_{\varepsilon}({{\mathbb{P}}}_{{N}})$ includes the true distribution $\mathbb{P}^0$ with arbitrarily high confidence. We show that for an $\varepsilon$ chosen in such a manner,~\ref{advdro} is well-defined. The full description of this result is available in our appendices.
\begin{theorem}[abridged collection of results from \citealt{fournier2015rate,kuhn2019wasserstein,YKW21:linear_optimization_wasserstein}]\label{thm:stats1}
    For light-tailed distribution $\mathbb{P}^0$ and $\varepsilon \geq \mathcal{O}(\frac{\log(\eta^{-1})}{N})^{1/n}$ for $\eta \in (0,1)$, we have: \emph{(i)} $\mathbb{P}^0 \in \mathfrak{B}_{\varepsilon}(\mathbb{P}_N)$ with $1 - \eta$ confidence; \emph{(ii)} \textup{\ref{advdro}} overestimates the expected loss for $\mathbb{P}^0$ with $1-\eta$ confidence; \emph{(iii)} \textup{\ref{advdro}} is asymptotically consistent $\mathbb{P}^0$-a.s.; \emph{(iv)} worst-case distributions for optimal solutions of~\textup{\ref{advdro}} are supported on at most $N+1$ outcomes.
\end{theorem}
We next derive an analogous result for~\ref{synth}.

\textbf{Choosing $\epsilon$ and $\widehat{\varepsilon}$ in \ref{synth}.} Recall that \ref{synth} revises~\ref{advdro} by intersecting $\mathfrak{B}_{\varepsilon}(\mathbb{P}_N)$ with $\mathfrak{B}_{\widehat\varepsilon}(\widehat{\mathbb{P}}_{\widehat{N}})$. We need a nonempty intersection for \ref{synth} to be well-defined. A necessary and sufficient condition follows from the triangle inequality $\varepsilon + \widehat{\varepsilon} \geq \mathrm{W}(\mathbb{P}_N, \widehat{\mathbb{P}}_{\widehat{N}})$, where $\mathrm{W}(\mathbb{P}_N, \widehat{\mathbb{P}}_{\widehat{N}})$ can be computed with linear optimization as both distributions are discrete. We also want this intersection to include $\mathbb{P}^0$ with high confidence, in order to satisfy Assumption~\ref{ass:strong}. We next provide a tight characterization for such $\varepsilon, \widehat{\varepsilon}$. The full description of this result is available in our appendices.
\begin{theorem}[abridged]\label{thm:stats2}
    For light-tailed $\mathbb{P}^0$ and $\widehat{\mathbb{P}}$, if $\varepsilon \geq \mathcal{O}(\frac{\log(\eta_1^{-1})}{N})^{1/n}$ and $\widehat{\varepsilon} \geq \mathrm{W}(\mathbb{P}^0, \widehat{\mathbb{P}}) + \mathcal{O}(\frac{\log(\eta_2^{-1})}{\widehat{N}})^{1/n}$ for $\eta_1, \eta_2 \in (0,1)$ with $\eta := \eta_1 + \eta_2 < 1$, we have: \emph{(i)} $\mathbb{P}^0 \in \mathfrak{B}_{\varepsilon}({\mathbb{P}}_N) \cap \mathfrak{B}_{\widehat \varepsilon}({\widehat{\mathbb{P}}}_{\widehat{N}})$ with $1 - \eta$ confidence; \emph{(ii)} \textup{\ref{synth}} overestimates true loss with $1-\eta$ confidence.
\end{theorem}

\begin{remark}\textup{\ref{synth}} is not asymptotically consistent, given that $\widehat{N} \rightarrow \infty$ will let $\widehat{\varepsilon} \rightarrow \mathrm{W}(\mathbb{P}^0, \widehat{\mathbb{P}})$ due to the 
non-zero constant distance between the true distribution $\mathbb{P}^0$ and the auxiliary distribution $\widehat{\mathbb{P}}$. 
\ref{synth} is thus not useful in asymptotic data regimes.
\end{remark}

\begin{remark}The assumption that true data-generating distributions are light-tailed is satisfied when $\Xi$ is compact, and it is a common assumption for even simple sample average approximation techniques~\citep[Assumption 3.3]{med17}.
\end{remark}

\paragraph{Knowledge of $\mathrm{W}(\mathbb{P}^0, \widehat{\mathbb{P}})$.} In Theorem~\ref{thm:stats2}, we use $\mathrm{W}(\mathbb{P}^0, \widehat{\mathbb{P}})$ explicitly. This distance, however, is typically unknown, and a common approach is to cross-validate it\footnote{In practice, distance between the unknown true and auxiliary data-generating distributions is also cross-validated in the transfer learning and domain adaptation literature~\citep{zhong2010cross}.}. This would be applicable in our setting thanks to Corollary~\ref{corr:relax}, because the relaxation~\ref{safe} does not force learning from the auxiliary data unless it is useful, that is, one can seek evidence for the usefulness of the auxiliary data via cross-validation. Moreover, there are several domains where $\mathrm{W}(\mathbb{P}^0, \widehat{\mathbb{P}})$ is known exactly. For some special cases, we can use direct domain knowledge (\textit{e.g.}, the ``Uber vs Lyft'' example of~\citealt{taskesen2021sequential}). A recent example comes from learning from multi-source data, where $\mathbb{P}^0$ is named the target distribution and $\widehat{\mathbb{P}}$ is the source distribution~\citep[\S 1]{rychener2024wasserstein}. Another domain is private data release, where a data holder shares a subset of opt-in data to form $\mathbb{P}_{N}$, and generates a privacy-preserving synthetic dataset from the rest. The (privately generated) synthetic distribution has a known nonzero Wasserstein distance from the true data-generating distribution \citep{dwork2014algorithmic,ullman2020pcps}. See~\citep[\S 5]{rychener2024wasserstein} for more scenarios that enable quantifying $\mathrm{W}(\mathbb{P}^0, \widehat{\mathbb{P}})$. Alternatively, one can directly rely on $\mathrm{W}(\mathbb{P}_N, \widehat{\mathbb{P}})$ if it is known, especially when synthetic data generators are trained on the empirical dataset. By employing Wasserstein GANs, which minimize the Wasserstein-1 distance, the distance between the generated distribution and the training distribution is minimized. This ensures that the synthetic distribution remains within a radius of the training distribution~\citep{WGAN}.

\section{EXPERIMENTS} \label{section:experiments}
We conduct a series of experiments, each having a different source of auxiliary data, to test the proposed DR ARO models. We use the following abbreviations, where `solution' refers to the optimal $\bm{\beta}$ to make decisions: 
\begin{itemize}[itemsep=1pt,topsep=0pt]
    \item[-] \texttt{ERM}: Solution of problem~\ref{erm} (\textit{i.e.}, nai\"ve LR);
    \item[-] \texttt{ARO}: Solution of problem~\ref{adv} (\textit{i.e.}, adversarially robust LR);
    \item[-] \texttt{ARO+Aux}: Solution of problem~\eqref{synth_literature} (\textit{i.e.}, replacing the empirical distribution of \ref{adv} with its mixture with auxiliary data);
    \item[-] \texttt{DRO+ARO}: Solution of~\ref{advdro} (\textit{i.e.}, the Wasserstein DR counterpart of~\ref{adv});
    \item[-] \texttt{DRO+ARO+Aux}: Solution of~\ref{safe} (\textit{i.e.}, relaxation of~\ref{synth} that intersects the ambiguity set of~\ref{advdro} with an auxiliary Wasserstein ball);
\end{itemize}
All parameters are $5$-fold cross-validated from various grids. The Wasserstein radii use the grid $\{10^{-6}, 10^{-5}, 10^{-4}, 10^{-3}, 10^{-2}, 10^{-1}, 0,1,2,5,10 \}$, which is sufficient to ensure that the rule-of-thumb $\varepsilon = \mathcal{O}(1/\sqrt{N})$ is included around the center of this grid for all experiments conducted. To ensure that the intersections of Wasserstein balls are nonempty, we compute $\mathrm{W}(\mathbb{P}_N, \widehat{\mathbb{P}}_{\widehat{N}})$ once, and discard all combinations $(\varepsilon, \widehat{\varepsilon})$ with $\varepsilon + \widehat{\varepsilon} < \mathrm{W}(\mathbb{P}_N, \widehat{\mathbb{P}}_{\widehat{N}})$. The weight parameter $\omega$ of \texttt{ARO+Aux} is cross-validated from the grid $\{10^{-6}, 10^{-5}, 10^{-4}, 10^{-3}, 10^{-2}, 10^{-1}, 0, 1\}$. We fix the norms defining the feature-label metric and the adversarial attacks to $\ell_1$- and $\ell_2$-norms, respectively. The parameter $\kappa$ (\textit{cf.}~Definition~\ref{def:feature-label}) is cross-validated from the grid $\{1, \sqrt{n}, n\}$, and since $n$ is the number of features, this grid includes cases where label uncertainty is equivalent to uncertainty of a single feature ($\kappa = 1$), label uncertainty is equivalent to uncertainty of all features combined ($\kappa = n$), and an intermediary case ($\kappa = \sqrt{n}$). The case of ignoring label uncertainty ($\kappa = \infty$) is purposely not included in the grid, since one of the key reasons behind robust overfitting is that while \texttt{ARO} is equivalent to a distributionally robust model, the underlying ambiguity is only around the features (\textit{cf.} our discussion in \S\ref{section:related}). All simulated adversarial attacks are worst-case $\ell_p$-attacks that are instance-wise at test time, and the experiments assume we know the strength $\alpha$ and norm $\ell_p$ of the adversarial attacks. 

All experiments are conducted in Julia and executed on Intel Xeon 2.66GHz processors with 16GB memory in single-core mode. We use MOSEK's exponential cone optimizer to solve all problems. To interpret the results accurately, recall that \texttt{DRO+ARO} and \texttt{DRO+ARO+Aux} are the DR models that we propose. Note also that \texttt{ERM}, \texttt{ARO}, and \texttt{DRO+ARO} do not utilize auxiliary data, while \texttt{DRO+ARO+Aux} and \texttt{ARO+Aux} have access to the same auxiliary datasets across all experiments (\textit{i.e.}, we do not sample different auxiliary distributions for different methods to ensure that our comparisons are made \textit{ceteris paribus}). Moreover, while \texttt{ARO} does not have access to auxiliary data, one can interpret \texttt{ARO+Aux} as a generalization of \texttt{ARO} that also has access to auxiliary datasets since it takes a mixture (with mixture weight $\omega$) of the empirical dataset with the auxiliary dataset. For example, $\omega = 1$ would simply revise \texttt{ARO} by appending the empirical dataset with the auxiliary dataset.

\subsection{UCI Datasets (Auxiliary Data is Synthetically Generated)}\label{sec_uci}
We compare the out-of-sample error rates of each method on 10 UCI datasets for binary classification~\citep{kelly2023uci}. For each dataset, we run 10 simulations as follows: \textit{(i)} Select $40\%$ of the data as a test set ($N_{\mathrm{te}} \propto 0.4$); \textit{(ii)} Sample $25\%$ of the remaining to form a training set ($N \propto 0.6 \cdot 0.25$); \textit{(iii)}  The rest ($\widehat{N} \propto 0.6 \cdot 0.75$) is used to fit a synthetic generator Gaussian Copula from the SDV package~\citep{SDV_paper}, which is then used to generate auxiliary data. The mean errors on the test set are reported in Table~\ref{tab:UCI} for $\ell_2$-attacks of strength $\alpha = 0.05$. The best error is always achieved by \texttt{DRO+ARO+Aux}, followed by \texttt{DRO+ARO}, \texttt{DRO+Aux}, \texttt{ARO}, \texttt{ERM}, respectively. In our appendices, we report similar results for attack strengths $\alpha \in \{0, 0.05, 0.2\}$, and share data preprocessing details and standard deviations of out-of-sample errors.
\begin{table*}[t]
\centering
\caption{Out-of-sample errors of UCI experiments with $\ell_2$-attacks of strength $\alpha = 0.05$.}
\begin{tabular}{lcccccc}
\toprule
Data   &  \texttt{ERM} & \texttt{ARO} & \texttt{ARO+Aux} & \texttt{DRO+ARO} & \texttt{DRO+ARO+Aux} \\
\midrule 
    {absent} & 44.02\%  & 38.82\%  & 35.95\%  &  34.22\%  & \textbf{32.64\%}\\
    {anneal} & 18.08\%  & 16.61\%  &  14.97\%  & 13.50\%  & \textbf{12.78\%}  \\
    {audio} & 21.43\%  & 21.54\% & 17.03\%  & 11.76\%  & \textbf{\phantom{0}9.01\%}  \\
      {breast-c} & \phantom{0}4.74\%  & \phantom{0}4.93\% & \phantom{0}3.87\% & \phantom{0}3.06\% & \textbf{\phantom{0}2.52\%}  \\
     {contrac} & 44.14\%  & 42.86\% & 40.98\%  & 40.00\%  & \textbf{39.65\%}  \\
     {derma}  & 15.97\%  & 16.46\%  & 13.47\% & 12.78\%  & \textbf{10.84\%}   \\
     {ecoli}  & 16.30\%  & 14.67\%  & 13.26\%  &  11.11\%  & \textbf{\phantom{0}9.78\%} \\
     {spam} & 11.35\%  & 10.23\%  & 10.16\%  & 9.83\%  & \textbf{9.81\%} \\
     {spect}  &  33.75\%  & 29.69\%  &  25.78\%  & 25.47\%  & \textbf{21.56\%}  \\
     {p-tumor}  & 21.84\%  & 20.81\%  & 17.35\%  & 16.18\%  & \textbf{14.78\%}  \\
\bottomrule
\end{tabular}
\label{tab:UCI}
\end{table*}

\subsection{MNIST/EMNIST Datasets (Auxiliary Data is Out-of-Domain)}
We use the MNIST digits dataset~\citep{mnist} to classify whether a digit is 1 or 7. For an auxiliary dataset, we use the larger EMNIST digits dataset~\citep{emnist}, whose authors summarize that this dataset has additional samples collected from a different group of individuals (high school students). Since EMNIST digits include MNIST digits, we remove the latter from the EMNIST dataset. We simulate the following 25 times: \textit{(i)} Sample 1,000 instances from the MNIST dataset as a training set; \textit{(ii)} The remaining instances in the MNIST dataset are our test set; \textit{(iii)} Sample 1,000 instances from the \mbox{EMNIST} dataset as an auxiliary dataset. Table~\ref{tab:mnist} reports the mean out-of-sample errors in various adversarial attack regimes. The results are analogous to the UCI experiments. Additionally, note that in the absence of adversarial attacks ($\alpha = 0$), \texttt{DRO+ARO} coincides with the Wasserstein LR model of~\cite{NIPS2015}, and the results thus imply that even without adversarial attacks, we can improve the state-of-the-art DR model by revising its ambiguity set in light of auxiliary data.

\begin{table*}[t]
\caption{Out-of-sample errors of MNIST/EMNIST experiments with various attacks.}
\centering
\begin{tabular}{lcccccc}
\toprule
Attack   &  \texttt{ERM}  &  \texttt{ARO} & \texttt{ARO+Aux} & \texttt{DRO+ARO} & \texttt{DRO+ARO+Aux} \\
\midrule 
    No attack ($\alpha = 0$) & \phantom{10}1.55\%  & \phantom{10}1.55\%  & \phantom{10}1.19\%  &  \phantom{10}0.64\%  & \phantom{10}\textbf{0.53\%}\\
    $\ell_1$ ($\alpha = 68/255$) & \phantom{10}2.17\% & \phantom{10}1.84\% & \phantom{10}1.33\% & \phantom{10}0.66\% & \phantom{10}\textbf{0.57\%} \\
    $\ell_2$ ($\alpha = 128/255$) & \phantom{1}99.93\% & \phantom{10}3.36\% & \phantom{10}2.54\% & \phantom{10}2.40\% & \phantom{10}\textbf{2.12\%}  \\
    $\ell_\infty$ ($\alpha = 8/255$) & 100.00\% & \phantom{10}2.60\% & \phantom{10}2.38\% & \phantom{10}2.20\% & \phantom{10}\textbf{1.95\%}  \\
\bottomrule
\end{tabular}
\label{tab:mnist}
\end{table*}

\subsection{Artificial Experiments (Auxiliary Data is Perturbed)}
\begin{figure*}[!t]
    \centering
        \includegraphics[width=0.4\textwidth]{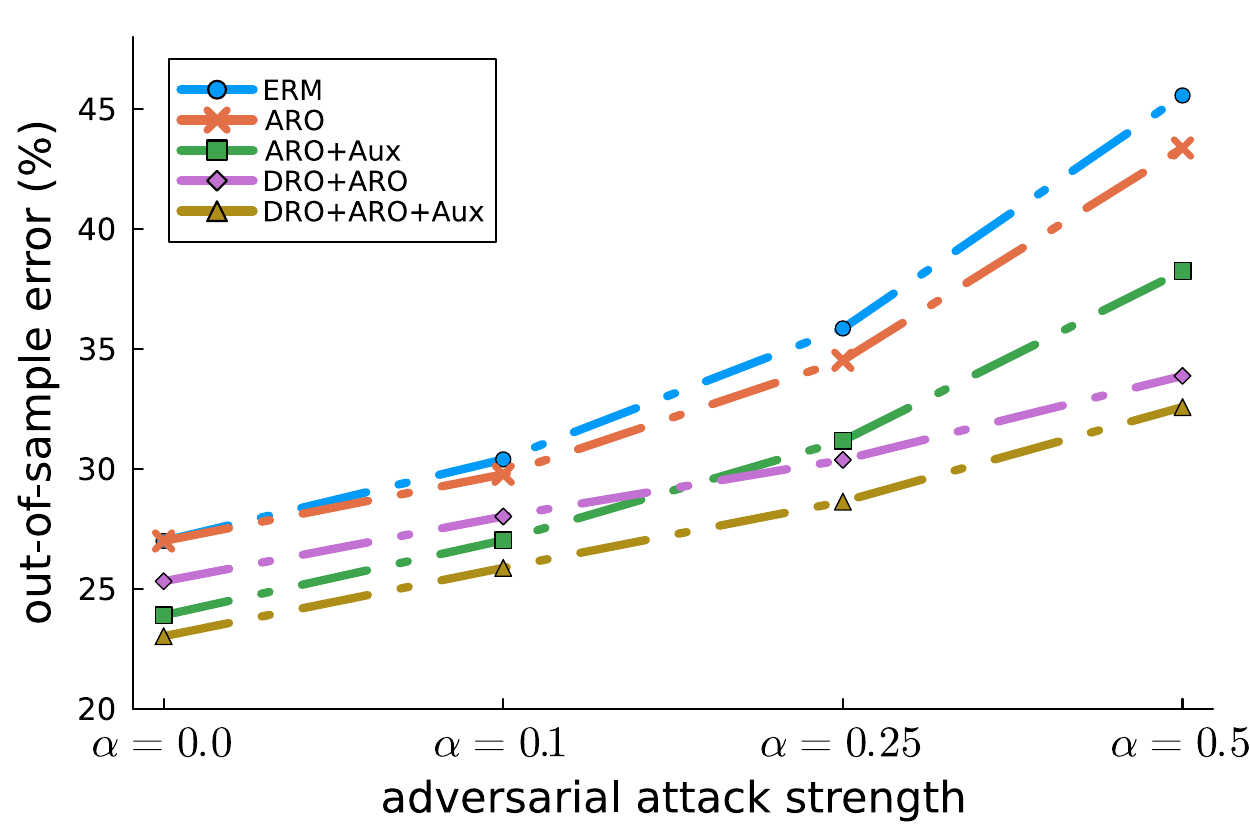}
        \hspace{1cm}
        \includegraphics[width=0.4\textwidth]{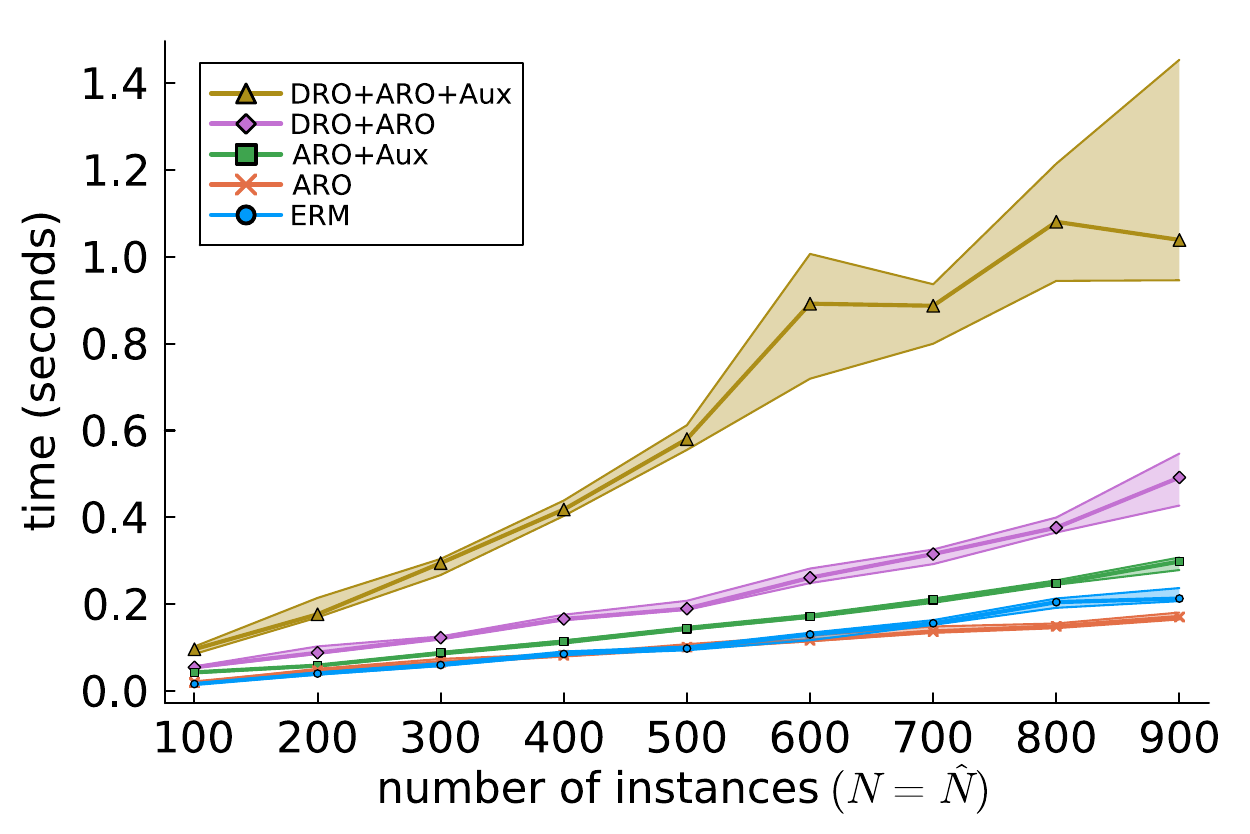}
    \caption{\textit{Out-of-sample errors under varying attack strengths (left) and runtimes under varying numbers of empirical and auxiliary instances (right) of artificial experiments.}}
    \label{fig:toy_example}
\end{figure*}

We generate empirical and auxiliary datasets by controlling their data-generating distributions (more details in the appendices). We simulate $25$ cases, each with $N = 100$ training, $\widehat{N} = 200$ auxiliary, and $N_{\mathrm{te}} = 10,000$ test instances and $n=100$ features. The performance of benchmark models with varying attacks is available in Figure~\ref{fig:toy_example} (left). \texttt{ERM} provides the worst performance, followed by \texttt{ARO}. The relationship between \texttt{DRO+ARO} and \texttt{ARO+Aux} is not monotonic: the former works better in larger attack regimes, conforming to the robust overfitting phenomenon. Finally, \texttt{Adv+DRO+Aux} always performs the best. We conduct a similar simulation for datasets with $n=100$, and gradually increase $N = \widehat{N}$ to report median ($50\% \pm 15\%$ quantiles shaded) runtimes of each method (\textit{cf.}~Figure~\ref{fig:toy_example}, right). The fastest methods is \texttt{ARO}, followed by \texttt{ERM}, \texttt{ARO+Aux}, \texttt{DRO+ARO}, and \texttt{DRO+ARO+Aux}. The slowest is \texttt{DRO+ARO+Aux}, but the runtime scales graciously. 

\section{CONCLUSIONS}
We formulate the distributionally robust counterpart of adversarially robust LR. Additionally, we demonstrate how to effectively utilize appropriately curated auxiliary data by intersecting Wasserstein balls. We illustrate the superiority of the proposed approach in terms of out-of-sample performance and confirm its scalability in practical settings.

From a theoretical point of view, it would be natural to extend our work to more loss functions, as is typical for DRO studies stemming from LR. To be able to optimize~\ref{safe} for very large-dimensional datasets, an interesting future work is to investigate first-order optimization methods that do not rely on off-the-shelf solvers. We also believe a cutting-plane method tailored for~\ref{safe} can also help us scale this problem for large-dimensional problems, since we would avoid monolithically optimizing a problem with $\mathcal{O}(N \cdot \widehat{N})$ exponential cone constraints. 

From a practical perspective, the ability to optimize~\ref{safe} in high-dimensional settings would also enable fine-tuning the final layer of a pre-trained neural network for binary classification, since this corresponds to logistic regression under a sigmoid activation. In our image recognition experiments, we used the MNIST dataset, as EMNIST served as a natural choice for auxiliary data. Identifying a suitable auxiliary dataset for CIFAR~\citep{krizhevsky2009learning} could similarly support new experimental directions.

Finally, recent breakthroughs in foundation models naturally pose the question of whether our ideas in this work apply to these models. For example, \cite{ye2022zerogen} use a pre-trained language model (PLM) to generate synthetic pairs of text sequences and labels which are then used to train downstream models. It would be interesting to adapt our ideas to the text domain to explore robustness in the presence of two PLMs.

\begin{acknowledgements}
Aras Selvi's work was done during an internship at JP Morgan AI Research. The authors gratefully acknowledge the detailed and constructive feedback
provided by the anonymous reviewers and the anonymous
area chair. Revising the paper has substantially improved the quality of this work.
\end{acknowledgements}

\section*{Disclaimer}
This paper was prepared for informational purposes by the Artificial Intelligence Research group of JPMorgan Chase \& Co. and its affiliates (``JP Morgan'') and is not a product of the Research Department of JP Morgan. JP Morgan makes no representation and warranty whatsoever and disclaims all liability, for the completeness, accuracy or reliability of the information contained herein. This document is not intended as investment research or investment advice, or a recommendation, offer or solicitation for the purchase or sale of any security, financial instrument, financial product or service, or to be used in any way for evaluating the merits of participating in any transaction, and shall not constitute a solicitation under any jurisdiction or to any person, if such solicitation under such jurisdiction or to such person would be unlawful.

\bibliography{uai2025-template}

\begin{thebibliography}{88}
\providecommand{\natexlab}[1]{#1}
\providecommand{\url}[1]{\texttt{#1}}
\expandafter\ifx\csname urlstyle\endcsname\relax
  \providecommand{\doi}[1]{doi: #1}\else
  \providecommand{\doi}{doi: \begingroup \urlstyle{rm}\Url}\fi

\bibitem[Arjovsky et~al.(2017)Arjovsky, Chintala, and Bottou]{WGAN}
Martin Arjovsky, Soumith Chintala, and L{\'e}on Bottou.
\newblock {W}asserstein generative adversarial networks.
\newblock In \emph{International Conference on Machine Learning}, 2017.

\bibitem[Awasthi et~al.(2022)Awasthi, Jung, and Morgenstern]{datajoin}
Pranjal Awasthi, Christopher Jung, and Jamie Morgenstern.
\newblock Distributionally robust data join.
\newblock \emph{arXiv:2202.05797}, 2022.

\bibitem[Belbasi et~al.(2023)Belbasi, Selvi, and Wiesemann]{belbasi2023s}
Reza Belbasi, Aras Selvi, and Wolfram Wiesemann.
\newblock It's all in the mix: Wasserstein machine learning with mixed features.
\newblock \emph{arXiv:2312.12230}, 2023.

\bibitem[Ben-Tal et~al.(2004)Ben-Tal, Goryashko, Guslitzer, and Nemirovski]{ben2004adjustable}
Aharon Ben-Tal, Alexander Goryashko, Elana Guslitzer, and Arkadi Nemirovski.
\newblock Adjustable robust solutions of uncertain linear programs.
\newblock \emph{Mathematical Programming}, 99\penalty0 (2):\penalty0 351--376, 2004.

\bibitem[Ben-Tal et~al.(2009)Ben-Tal, Ghaoui, and Nemirovski]{BTEGN09:rob_opt}
Aharon Ben-Tal, Laurent~El Ghaoui, and Arkadi Nemirovski.
\newblock \emph{Robust Optimization}.
\newblock Princeton University Press, 2009.

\bibitem[Bennouna and Van~Parys(2022)]{bennouna2022holistic}
Amine Bennouna and Bart Van~Parys.
\newblock Holistic robust data-driven decisions.
\newblock \emph{arXiv:2207.09560}, 2022.

\bibitem[Bennouna et~al.(2023)Bennouna, Lucas, and Van~Parys]{bennouna2023certified}
Amine Bennouna, Ryan Lucas, and Bart Van~Parys.
\newblock Certified robust neural networks: Generalization and corruption resistance.
\newblock In \emph{International Conference on Machine Learning}, 2023.

\bibitem[Bertsimas and {Den Hertog}(2022)]{BdH22:rob_opt}
Dimitris Bertsimas and Dick {Den Hertog}.
\newblock \emph{Robust and Adaptive Optimization}.
\newblock Dynamic Ideas, 2022.

\bibitem[Bertsimas et~al.(2015)Bertsimas, Goyal, and Lu]{bertsimas2015tight}
Dimitris Bertsimas, Vineet Goyal, and Brian~Y. Lu.
\newblock A tight characterization of the performance of static solutions in two-stage adjustable robust linear optimization.
\newblock \emph{Mathematical Programming}, 150\penalty0 (2):\penalty0 281--319, 2015.

\bibitem[Bertsimas et~al.(2019)Bertsimas, Dunn, Pawlowski, and Zhuo]{bertsimas2019robust}
Dimitris Bertsimas, Jack Dunn, Colin Pawlowski, and Ying~Daisy Zhuo.
\newblock Robust classification.
\newblock \emph{INFORMS Journal on Optimization}, 1\penalty0 (1):\penalty0 2--34, 2019.

\bibitem[Biggio et~al.(2013)Biggio, Corona, Maiorca, Nelson, {\v{S}}rndi{\'c}, Laskov, Giacinto, and Roli]{biggio2013evasion}
Battista Biggio, Igino Corona, Davide Maiorca, Blaine Nelson, Nedim {\v{S}}rndi{\'c}, Pavel Laskov, Giorgio Giacinto, and Fabio Roli.
\newblock Evasion attacks against machine learning at test time.
\newblock In \emph{Machine Learning and Knowledge Discovery in Databases: European Conference}, pages 387--402, 2013.

\bibitem[Bishop(2006)]{bishop2006pattern}
Christopher Bishop.
\newblock \emph{Pattern Recognition and Machine Learning}.
\newblock Springer, 2006.

\bibitem[Bui et~al.(2022)Bui, Le, Tran, Zhao, and Phung]{bui2022unified}
Tuan~Anh Bui, Trung Le, Quan Tran, He~Zhao, and Dinh Phung.
\newblock A unified {Wasserstein} distributional robustness framework for adversarial training.
\newblock \emph{arXiv:2202.13437}, 2022.

\bibitem[Carlini et~al.(2019)Carlini, Athalye, Papernot, Brendel, Rauber, Tsipras, Goodfellow, Madry, and Kurakin]{carlini2019evaluating}
Nicholas Carlini, Anish Athalye, Nicolas Papernot, Wieland Brendel, Jonas Rauber, Dimitris Tsipras, Ian Goodfellow, Aleksander Madry, and Alexey Kurakin.
\newblock On evaluating adversarial robustness.
\newblock \emph{arXiv:1902.06705}, 2019.

\bibitem[Carmon et~al.(2019)Carmon, Raghunathan, Schmidt, Duchi, and Liang]{carmon2022unlabeled}
Yair Carmon, Aditi Raghunathan, Ludwig Schmidt, John~C. Duchi, and Percy~S. Liang.
\newblock Unlabeled data improves adversarial robustness.
\newblock In \emph{Advances in Neural Information Processing Systems}, 2019.

\bibitem[Chen et~al.(2020)Chen, Zhang, Liu, Chang, and Wang]{chen2020robust}
Tianlong Chen, Zhenyu Zhang, Sijia Liu, Shiyu Chang, and Zhangyang Wang.
\newblock Robust overfitting may be mitigated by properly learned smoothening.
\newblock In \emph{International Conference on Learning Representations}, 2020.

\bibitem[Cohen et~al.(2017)Cohen, Afshar, Tapson, and Van~Schaik]{emnist}
Gregory Cohen, Saeed Afshar, Jonathan Tapson, and Andre Van~Schaik.
\newblock {EMNIST}: {Extending} {MNIST} to handwritten letters.
\newblock In \emph{International Joint Conference on Neural Networks}, 2017.

\bibitem[Croce et~al.(2020)Croce, Andriushchenko, Sehwag, Debenedetti, Flammarion, Chiang, Mittal, and Hein]{croce2020robustbench}
Francesco Croce, Maksym Andriushchenko, Vikash Sehwag, Edoardo Debenedetti, Nicolas Flammarion, Mung Chiang, Prateek Mittal, and Matthias Hein.
\newblock Robustbench: a standardized adversarial robustness benchmark.
\newblock \emph{arXiv:2010.09670}, 2020.

\bibitem[Delage and Ye(2010)]{DY10:distr_rob_opt}
Erick Delage and Yinyu Ye.
\newblock Distributionally robust optimization under moment uncertainty with application to data-driven problems.
\newblock \emph{Operations Research}, 58\penalty0 (3):\penalty0 596--612, 2010.

\bibitem[DeMiguel and Nogales(2009)]{demiguel2009portfolio}
Victor DeMiguel and Francisco~J. Nogales.
\newblock Portfolio selection with robust estimation.
\newblock \emph{Operations research}, 57\penalty0 (3):\penalty0 560--577, 2009.

\bibitem[Deng et~al.(2021)Deng, Zhang, Ghorbani, and Zou]{deng2021improving}
Zhun Deng, Linjun Zhang, Amirata Ghorbani, and James Zou.
\newblock Improving adversarial robustness via unlabeled out-of-domain data.
\newblock In \emph{International Conference on Artificial Intelligence and Statistics}, 2021.

\bibitem[Dwork and Roth(2014)]{dwork2014algorithmic}
Cynthia Dwork and Aaron Roth.
\newblock The algorithmic foundations of differential privacy.
\newblock \emph{Foundations and Trends{\textregistered} in Theoretical Computer Science}, 9\penalty0 (3--4):\penalty0 211--407, 2014.

\bibitem[Fournier and Guillin(2015)]{fournier2015rate}
Nicolas Fournier and Arnaud Guillin.
\newblock On the rate of convergence in {W}asserstein distance of the empirical measure.
\newblock \emph{Probability Theory and Related Fields}, 162\penalty0 (3-4):\penalty0 707--738, 2015.

\bibitem[Frank and Niles-Weed(2024)]{frank2024existence}
Natalie~S. Frank and Jonathan Niles-Weed.
\newblock Existence and minimax theorems for adversarial surrogate risks in binary classification.
\newblock \emph{Journal of Machine Learning Research}, 25\penalty0 (58):\penalty0 1--41, 2024.

\bibitem[Gao(2023)]{gao2023finite}
Rui Gao.
\newblock Finite-sample guarantees for {Wasserstein} distributionally robust optimization: Breaking the curse of dimensionality.
\newblock \emph{Operations Research}, 71\penalty0 (6):\penalty0 2291--2306, 2023.

\bibitem[Gao et~al.(2019)Gao, Cai, Li, Hsieh, Wang, and Lee]{gao2019convergence}
Ruiqi Gao, Tianle Cai, Haochuan Li, Cho-Jui Hsieh, Liwei Wang, and Jason~D. Lee.
\newblock Convergence of adversarial training in overparametrized neural networks.
\newblock In \emph{Advances in Neural Information Processing Systems}, 2019.

\bibitem[Givens and Shortt(1984)]{givens1984class}
Clark~R. Givens and Rae~M. Shortt.
\newblock A class of {Wasserstein} metrics for probability distributions.
\newblock \emph{Michigan Mathematical Journal}, 31\penalty0 (2):\penalty0 231--240, 1984.

\bibitem[Goodfellow et~al.(2015)Goodfellow, Shlens, and Szegedy]{goodfellow2014explaining}
Ian~J. Goodfellow, Jonathon Shlens, and Christian Szegedy.
\newblock Explaining and harnessing adversarial examples.
\newblock In \emph{International Conference on Learning Representations}, 2015.

\bibitem[Gorissen et~al.(2015)Gorissen, Yan{\i}ko{\u{g}}lu, and Den~Hertog]{gorissen2015practical}
Bram~L. Gorissen, {\.I}hsan Yan{\i}ko{\u{g}}lu, and Dick Den~Hertog.
\newblock A practical guide to robust optimization.
\newblock \emph{Omega}, 53:\penalty0 124--137, 2015.

\bibitem[Gowal et~al.(2021)Gowal, Rebuffi, Wiles, Stimberg, Calian, and Mann]{gowal2021improving}
Sven Gowal, Sylvestre-Alvise Rebuffi, Olivia Wiles, Florian Stimberg, Dan~A. Calian, and Timothy~A. Mann.
\newblock Improving robustness using generated data.
\newblock In \emph{Advances in Neural Information Processing Systems}, 2021.

\bibitem[Guslitser(2002)]{Guslitser2002}
Elana Guslitser.
\newblock Uncertainty-immunized solutions in linear programming.
\newblock Master's thesis, Technion -- Israeli Institute of Technology, 2002.

\bibitem[Hastie et~al.(2017)Hastie, Tibshirani, and Friedman]{hastie2009elements}
Trevor Hastie, Robert Tibshirani, and Jerome Friedman.
\newblock \emph{The elements of statistical learning: {D}ata mining, inference, and prediction}.
\newblock Springer, 2017.

\bibitem[Kelly et~al.(2023)Kelly, Longjohn, and Nottingham]{kelly2023uci}
Markelle Kelly, Rachel Longjohn, and Kolby Nottingham.
\newblock The {UCI} machine learning repository.
\newblock {\small\url{https://archive. ics. uci. edu}}, 2023.

\bibitem[Khim and Loh(2018)]{khim2018adversarial}
Justin Khim and Po-Ling Loh.
\newblock Adversarial risk bounds via function transformation.
\newblock \emph{arXiv:1810.09519}, 2018.

\bibitem[Krizhevsky et~al.(2009)Krizhevsky, Hinton, et~al.]{krizhevsky2009learning}
Alex Krizhevsky, Geoffrey Hinton, et~al.
\newblock Learning multiple layers of features from tiny images.
\newblock 2009.

\bibitem[Kuhn et~al.(2019)Kuhn, {Mohajerin Esfahani}, Nguyen, and {Shafieezadeh-Abadeh}]{kuhn2019wasserstein}
Daniel Kuhn, Peyman {Mohajerin Esfahani}, Viet~Anh Nguyen, and Soroosh {Shafieezadeh-Abadeh}.
\newblock {W}asserstein distributionally robust optimization: {T}heory and applications in machine learning.
\newblock \emph{INFORMS TutORials in Operations Research}, pages 130--169, 2019.

\bibitem[LeCun et~al.(1998)LeCun, Bottou, Bengio, and Haffner]{mnist}
Yann LeCun, L{\'e}on Bottou, Yoshua Bengio, and Patrick Haffner.
\newblock Gradient-based learning applied to document recognition.
\newblock \emph{Proceedings of the IEEE}, 86\penalty0 (11):\penalty0 2278--2324, 1998.

\bibitem[Li and Li(2023)]{li2023clean}
Binghui Li and Yuanzhi Li.
\newblock Why clean generalization and robust overfitting both happen in adversarial training.
\newblock \emph{arXiv:2306.01271}, 2023.

\bibitem[Li and Spratling(2023)]{li2023understanding}
Lin Li and Michael Spratling.
\newblock Understanding and combating robust overfitting via input loss landscape analysis and regularization.
\newblock \emph{Pattern Recognition}, 136:\penalty0 1--11, 2023.

\bibitem[Lin et~al.(2023)Lin, Yu, and Liu]{lin2023eliminating}
Runqi Lin, Chaojian Yu, and Tongliang Liu.
\newblock Eliminating catastrophic overfitting via abnormal adversarial examples regularization.
\newblock \emph{Advances in Neural Information Processing Systems}, 36:\penalty0 67866--67885, 2023.

\bibitem[Lin et~al.(2024)Lin, Yu, Han, Su, and Liu]{lin2024layer}
Runqi Lin, Chaojian Yu, Bo~Han, Hang Su, and Tongliang Liu.
\newblock Layer-aware analysis of catastrophic overfitting: Revealing the pseudo-robust shortcut dependency.
\newblock In Ruslan Salakhutdinov, Zico Kolter, Katherine Heller, Adrian Weller, Nuria Oliver, Jonathan Scarlett, and Felix Berkenkamp, editors, \emph{Proceedings of the 41st International Conference on Machine Learning}, volume 235, pages 30427--30439, 21--27 Jul 2024.
\newblock URL \url{https://proceedings.mlr.press/v235/lin24v.html}.

\bibitem[Madry et~al.(2018)Madry, Makelov, Schmidt, Tsipras, and Vladu]{madry2017towards}
Aleksander Madry, Aleksandar Makelov, Ludwig Schmidt, Dimitris Tsipras, and Adrian Vladu.
\newblock Towards deep learning models resistant to adversarial attacks.
\newblock In \emph{International Conference on Learning Representations}, 2018.

\bibitem[Mohajerin~Esfahani and Kuhn(2018)]{med17}
Peyman Mohajerin~Esfahani and Daniel Kuhn.
\newblock Data-driven distributionally robust optimization using the wasserstein metric: Performance guarantees and tractable reformulations.
\newblock \emph{Mathematical Programming}, 171\penalty0 (1):\penalty0 115--166, 2018.

\bibitem[{MOSEK ApS}(2023)]{mosek2016modeling}
{MOSEK ApS}.
\newblock Modeling cookbook.
\newblock {\small\url{https://docs.mosek.com/MOSEKModelingCookbook-letter.pdf}}, 2023.

\bibitem[Murphy(2022)]{murphy2022}
Kevin~P. Murphy.
\newblock \emph{Probabilistic machine learning: an introduction}.
\newblock {MIT} press, 2022.

\bibitem[Nesterov(2018)]{N18:convex}
Yurii Nesterov.
\newblock \emph{Lectures on Convex Optimization}.
\newblock Springer, 2018.

\bibitem[Pang et~al.(2022)Pang, Lin, Yang, Zhu, and Yan]{PangLYZY22}
Tianyu Pang, Min Lin, Xiao Yang, Jun Zhu, and Shuicheng Yan.
\newblock Robustness and accuracy could be reconcilable by (proper) definition.
\newblock In \emph{International Conference on Machine Learning}, 2022.

\bibitem[Patki et~al.(2016)Patki, Wedge, and Veeramachaneni]{SDV_paper}
Neha Patki, Roy Wedge, and Kalyan Veeramachaneni.
\newblock The synthetic data vault.
\newblock In \emph{{IEEE} International Conference on Data Science and Advanced Analytics}, 2016.

\bibitem[Phan et~al.(2023)Phan, Le, Phung, Bui, Ho, and Phung]{phan2023global}
Hoang Phan, Trung Le, Trung Phung, Anh~Tuan Bui, Nhat Ho, and Dinh Phung.
\newblock Global-local regularization via distributional robustness.
\newblock In \emph{International Conference on Artificial Intelligence and Statistics}, 2023.

\bibitem[Pydi and Jog(2021)]{pydi2021many}
Muni~Sreenivas Pydi and Varun Jog.
\newblock The many faces of adversarial risk.
\newblock In \emph{Advances in Neural Information Processing Systems}, 2021.

\bibitem[Rade and Moosavi-Dezfooli(2022)]{rade2022reducing}
Rahul Rade and Seyed-Mohsen Moosavi-Dezfooli.
\newblock Reducing excessive margin to achieve a better accuracy vs. robustness trade-off.
\newblock In \emph{International Conference on Learning Representations}, 2022.

\bibitem[Raghunathan et~al.(2019)Raghunathan, Xie, Yang, Duchi, and Liang]{raghunathan2019adversarial}
Aditi Raghunathan, Sang~Michael Xie, Fanny Yang, John~C. Duchi, and Percy Liang.
\newblock Adversarial training can hurt generalization.
\newblock \emph{arXiv:1906.06032}, 2019.

\bibitem[Regniez et~al.(2022)Regniez, Gidel, and Berard]{regniez2022a}
Chiara Regniez, Gauthier Gidel, and Hugo Berard.
\newblock A distributional robustness perspective on adversarial training with the $\infty$-{W}asserstein distance.
\newblock {\small\url{https://openreview.net/forum?id=z7DAilcTx7}}, 2022.

\bibitem[Rice et~al.(2020)Rice, Wong, and Kolter]{rice2020overfitting}
Leslie Rice, Eric Wong, and Zico Kolter.
\newblock Overfitting in adversarially robust deep learning.
\newblock In \emph{International Conference on Machine Learning}, 2020.

\bibitem[Rockafellar(1997)]{R97:convex_analysis}
R.~Tyrrell Rockafellar.
\newblock \emph{Convex Analysis}.
\newblock Princeton University Press, 1997.

\bibitem[Rychener et~al.(2024)Rychener, Esteban-P{\'e}rez, Morales, and Kuhn]{rychener2024wasserstein}
Yves Rychener, Adri{\'a}n Esteban-P{\'e}rez, Juan~M Morales, and Daniel Kuhn.
\newblock Wasserstein distributionally robust optimization with heterogeneous data sources.
\newblock \emph{arXiv:2407.13582}, 2024.

\bibitem[Sehwag et~al.(2022)Sehwag, Mahloujifar, Handina, Dai, Xiang, Chiang, and Mittal]{sehwag2022robust}
Vikash Sehwag, Saeed Mahloujifar, Tinashe Handina, Sihui Dai, Chong Xiang, Mung Chiang, and Prateek Mittal.
\newblock Robust learning meets generative models: Can proxy distributions improve adversarial robustness?
\newblock In \emph{International Conference on Learning Representations}, 2022.

\bibitem[Selvi et~al.(2022)Selvi, Belbasi, Haugh, and Wiesemann]{selvi2022wasserstein}
Aras Selvi, Mohammad~Reza Belbasi, Martin Haugh, and Wolfram Wiesemann.
\newblock Wasserstein logistic regression with mixed features.
\newblock In \emph{Advances in Neural Information Processing Systems}, 2022.

\bibitem[Shafahi et~al.(2019)Shafahi, Najibi, Ghiasi, Xu, Dickerson, Studer, Davis, Taylor, and Goldstein]{shafahi2019adversarial}
Ali Shafahi, Mahyar Najibi, Mohammad~Amin Ghiasi, Zheng Xu, John Dickerson, Christoph Studer, Larry~S Davis, Gavin Taylor, and Tom Goldstein.
\newblock Adversarial training for free!
\newblock In \emph{Advances in Neural Information Processing Systems}, 2019.

\bibitem[Shafahi et~al.(2020)Shafahi, Saadatpanah, Zhu, Ghiasi, Studer, Jacobs, and Goldstein]{shafahi2020adversarially}
Ali Shafahi, Parsa Saadatpanah, Chen Zhu, Amin Ghiasi, Christoph Studer, David~W. Jacobs, and Tom Goldstein.
\newblock Adversarially robust transfer learning.
\newblock In \emph{International Conference on Learning Representations}, 2020.

\bibitem[{Shafieezadeh-Abadeh} et~al.(2015){Shafieezadeh-Abadeh}, Mohajerin~Esfahani, and Kuhn]{NIPS2015}
Soroosh {Shafieezadeh-Abadeh}, Peyman Mohajerin~Esfahani, and Daniel Kuhn.
\newblock Distributionally robust logistic regression.
\newblock \emph{Advances in Neural Information Processing Systems}, 2015.

\bibitem[{Shafieezadeh-Abadeh} et~al.(2019){Shafieezadeh-Abadeh}, Kuhn, and Esfahani]{shafieezadeh2019regularization}
Soroosh {Shafieezadeh-Abadeh}, Daniel Kuhn, and Peyman~Mohajerin Esfahani.
\newblock Regularization via mass transportation.
\newblock \emph{Journal of Machine Learning Research}, 20\penalty0 (103):\penalty0 1--68, 2019.

\bibitem[Shafieezadeh-Abadeh et~al.(2023)Shafieezadeh-Abadeh, Aolaritei, D{\"o}rfler, and Kuhn]{shafieezadeh2023new}
Soroosh Shafieezadeh-Abadeh, Liviu Aolaritei, Florian D{\"o}rfler, and Daniel Kuhn.
\newblock New perspectives on regularization and computation in optimal transport-based distributionally robust optimization.
\newblock \emph{arXiv:2303.03900}, 2023.

\bibitem[Shapiro(2001)]{shapiro2001duality}
Alexander Shapiro.
\newblock On duality theory of conic linear problems.
\newblock \emph{Nonconvex Optimization and its Applications}, 57:\penalty0 135--155, 2001.

\bibitem[Sinha et~al.(2018)Sinha, Namkoong, and Duchi]{sinha2017certifying}
Aman Sinha, Hongseok Namkoong, and John~C. Duchi.
\newblock Certifying some distributional robustness with principled adversarial training.
\newblock In \emph{International Conference on Learning Representations}, 2018.

\bibitem[Smith and Winkler(2006)]{smith2006optimizer}
James~E. Smith and Robert~L. Winkler.
\newblock The optimizer’s curse: Skepticism and postdecision surprise in decision analysis.
\newblock \emph{Management Science}, 52\penalty0 (3):\penalty0 311--322, 2006.

\bibitem[Song et~al.(2019)Song, He, Wang, and Hopcroft]{song2018improving}
Chuanbiao Song, Kun He, Liwei Wang, and John~E. Hopcroft.
\newblock Improving the generalization of adversarial training with domain adaptation.
\newblock In \emph{International Conference on Learning Representations}, 2019.

\bibitem[Staib and Jegelka(2017)]{staib2017distributionally}
Matthew Staib and Stefanie Jegelka.
\newblock Distributionally robust deep learning as a generalization of adversarial training.
\newblock In \emph{NIPS workshop on Machine Learning and Computer Security}, 2017.

\bibitem[Subramanyam et~al.(2020)Subramanyam, Gounaris, and Wiesemann]{subramanyam2020k}
Anirudh Subramanyam, Chrysanthos~E. Gounaris, and Wolfram Wiesemann.
\newblock {$K$}-adaptability in two-stage mixed-integer robust optimization.
\newblock \emph{Mathematical Programming Computation}, 12:\penalty0 193--224, 2020.

\bibitem[Szegedy et~al.(2014)Szegedy, Zaremba, Sutskever, Bruna, Erhan, Goodfellow, and Fergus]{advattackpaper}
Christian Szegedy, Wojciech Zaremba, Ilya Sutskever, Joan Bruna, Dumitru Erhan, Ian~J. Goodfellow, and Rob Fergus.
\newblock Intriguing properties of neural networks.
\newblock In \emph{International Conference on Learning Representations}, 2014.

\bibitem[Tanoumand et~al.(2023)Tanoumand, Bodur, and Naoum-Sawaya]{intersectballs}
Neda Tanoumand, Merve Bodur, and Joe Naoum-Sawaya.
\newblock Data-driven distributionally robust optimization: Intersecting ambiguity sets, performance analysis and tractability.
\newblock \emph{Optimization Online 22567}, 2023.

\bibitem[Taskesen et~al.(2021)Taskesen, Yue, Blanchet, Kuhn, and Nguyen]{taskesen2021sequential}
Bahar Taskesen, Man-Chung Yue, Jose Blanchet, Daniel Kuhn, and Viet~Anh Nguyen.
\newblock Sequential domain adaptation by synthesizing distributionally robust experts.
\newblock In \emph{International Conference on Machine Learning}, 2021.

\bibitem[Toland(1978)]{toland1978duality}
John~F. Toland.
\newblock Duality in nonconvex optimization.
\newblock \emph{Journal of Mathematical Analysis and Applications}, 66\penalty0 (2):\penalty0 399--415, 1978.

\bibitem[Uesato et~al.(2018)Uesato, O’donoghue, Kohli, and Oord]{uesato2018adversarial}
Jonathan Uesato, Brendan O’donoghue, Pushmeet Kohli, and Aaron Oord.
\newblock Adversarial risk and the dangers of evaluating against weak attacks.
\newblock In \emph{International Conference on Machine Learning}, 2018.

\bibitem[Ullman and Vadhan(2020)]{ullman2020pcps}
Jonathan Ullman and Salil Vadhan.
\newblock {PCP}s and the hardness of generating synthetic data.
\newblock \emph{Journal of Cryptology}, 33\penalty0 (4):\penalty0 2078--2112, 2020.

\bibitem[Vapnik(1999)]{vapnik1999nature}
Vladimir Vapnik.
\newblock \emph{The nature of statistical learning theory}.
\newblock Springer, 1999.

\bibitem[Villani(2009)]{villani2009optimal}
C{\'e}dric Villani.
\newblock \emph{Optimal transport: Old and new}.
\newblock Springer, 2009.

\bibitem[Wang et~al.(2024)Wang, Chen, and Wang]{wang2024contextual}
Tianyu Wang, Ningyuan Chen, and Chun Wang.
\newblock Contextual optimization under covariate shift: A robust approach by intersecting wasserstein balls.
\newblock \emph{arXiv:2406.02426}, 2024.

\bibitem[Wong et~al.(2020)Wong, Rice, and Kolter]{wong2020fast}
Eric Wong, Leslie Rice, and J~Zico Kolter.
\newblock Fast is better than free: Revisiting adversarial training.
\newblock In \emph{International Conference on Learning Representations}, 2020.

\bibitem[Wu et~al.(2020)Wu, Xia, and Wang]{wu2020adversarial}
Dongxian Wu, Shu-Tao Xia, and Yisen Wang.
\newblock Adversarial weight perturbation helps robust generalization.
\newblock In \emph{Advances in Neural Information Processing Systems}, 2020.

\bibitem[Xing et~al.(2022{\natexlab{a}})Xing, Song, and Cheng]{xing2022artificially}
Yue Xing, Qifan Song, and Guang Cheng.
\newblock Why do artificially generated data help adversarial robustness.
\newblock In \emph{Advances in Neural Information Processing Systems}, 2022{\natexlab{a}}.

\bibitem[Xing et~al.(2022{\natexlab{b}})Xing, Song, and Cheng]{xing2022unlabeled}
Yue Xing, Qifan Song, and Guang Cheng.
\newblock Unlabeled data help: Minimax analysis and adversarial robustness.
\newblock In \emph{International Conference on Artificial Intelligence and Statistics}, 2022{\natexlab{b}}.

\bibitem[Yan{\i}ko{\u{g}}lu et~al.(2019)Yan{\i}ko{\u{g}}lu, Gorissen, and den Hertog]{yanikouglu2019survey}
{\.I}hsan Yan{\i}ko{\u{g}}lu, Bram~L. Gorissen, and Dick den Hertog.
\newblock A survey of adjustable robust optimization.
\newblock \emph{European Journal of Operational Research}, 277\penalty0 (3):\penalty0 799--813, 2019.

\bibitem[Ye et~al.(2022)Ye, Gao, Li, Xu, Feng, Wu, Yu, and Kong]{ye2022zerogen}
Jiacheng Ye, Jiahui Gao, Qintong Li, Hang Xu, Jiangtao Feng, Zhiyong Wu, Tao Yu, and Lingpeng Kong.
\newblock Zerogen: {E}fficient zero-shot learning via dataset generation.
\newblock In \emph{Conference on Empirical Methods in Natural Language Processing}, 2022.

\bibitem[Yu et~al.(2022)Yu, Han, Shen, Yu, Gong, Gong, and Liu]{yu2022understanding}
Chaojian Yu, Bo~Han, Li~Shen, Jun Yu, Chen Gong, Mingming Gong, and Tongliang Liu.
\newblock Understanding robust overfitting of adversarial training and beyond.
\newblock In \emph{International Conference on Machine Learning}, 2022.

\bibitem[Yue et~al.(2022)Yue, Kuhn, and Wiesemann]{YKW21:linear_optimization_wasserstein}
{Man-Chung} Yue, Daniel Kuhn, and Wolfram Wiesemann.
\newblock On linear optimization over {W}asserstein balls.
\newblock \emph{Mathematical Programming}, 195\penalty0 (1-2):\penalty0 1107--1122, 2022.

\bibitem[Zhang et~al.(2019)Zhang, Yu, Jiao, Xing, El~Ghaoui, and Jordan]{zhang2019theoretically}
Hongyang Zhang, Yaodong Yu, Jiantao Jiao, Eric Xing, Laurent El~Ghaoui, and Michael Jordan.
\newblock Theoretically principled trade-off between robustness and accuracy.
\newblock In \emph{International Conference on Machine Learning}, 2019.

\bibitem[Zhong et~al.(2010)Zhong, Fan, Yang, Verscheure, and Ren]{zhong2010cross}
Erheng Zhong, Wei Fan, Qiang Yang, Olivier Verscheure, and Jiangtao Ren.
\newblock Cross validation framework to choose amongst models and datasets for transfer learning.
\newblock In \emph{Machine Learning and Knowledge Discovery in Databases: European Conference}, volume 6323, 2010.

\end{thebibliography}

\newpage

\onecolumn

\title{Distributionally and Adversarially Robust Logistic Regression via Intersecting Wasserstein Balls \\
{\color{red!50!black}(Supplementary Material)}}
\maketitle
\appendix

\section{NOTATION}
Throughout the paper, bold lowercase letters denote vectors, while standard lowercase letters are reserved for scalars. A generic data instance is modeled as $\bm{\xi} = (\bm{x}, y) \in \Xi := \mathbb{R}^n \times \{-1, +1\}$. For any $p>0$, $\lVert \bm{x} \rVert_{p}$ denotes the $p$-norm $\left(\sum_{i=1}^n \lvert x_i \rvert^{p} \right)^{1/p}$ and $\lVert \bm{x}\rVert_{p^\star}$ is its dual norm where $1/p +1/p^\star = 1$ with the convention of $1/1 + 1/\infty = 1$. The set of probability distributions supported on $\Xi$ is denoted by $\mathcal{P}(\Xi)$. The Dirac measure supported on $\bm{\xi}$ is denoted by $\delta_{\bm{\xi}}$. The logloss is defined as $\ell_{\bm{\beta}}(\bm{x}, y) = \log(1+ \exp(-y \cdot \bm{\beta}^\top \bm{x}))$ and its associated univariate loss is $L(z) = \log(1 + \exp(-z))$ so that $L(y \cdot \bm{\beta}^\top \bm{x}) = \ell_{\bm{\beta}}(\bm{x}, y)$. The exponential cone is denoted by $\mathcal{K}_{\exp} = \mathrm{cl}(\{\bm{\omega} \in \mathbb{R}^3 : \omega_1 \geq \omega_2 \cdot \exp(\omega_3/\omega_2), \ \omega_1 > 0, \ \omega_2 > 0 \})$ where $\mathrm{cl}$ is the closure operator. The Lipschitz modulus of a univariate function $f$ is defined as $\mathrm{Lip}(f) := \sup_{z, z' \in \mathbb{R}} \left\{ \lvert f(z) - f(z')\rvert / \lvert z-z'\rvert : z \neq z \right\}$ whereas its effective domain is $\mathrm{dom}(f) = \{ z : f(z) < +\infty \}$. For a function $f:\mathbb{R}^n \mapsto \mathbb{R}$, its convex conjugate is $f^{*}(\bm{z}) = \sup_{\bm{x}\in\mathbb{R}^n} \bm{z}^\top \bm{x} - f(\bm{x})$. We reserve $\alpha \geq 0$ for the radii of the norms of adversarial attacks on the features and $\varepsilon \geq 0$ for the radii of distributional ambiguity sets.

\section{MISSING PROOFS}
\subsection{Proof of Observation~\ref{obs:prelim}}
For any $\bm{\beta} \in \mathbb{R}^n$, with standard robust optimization arguments~\citep{BTEGN09:rob_opt,BdH22:rob_opt}, we can show that
\begin{align*}
        & \underset{\bm{z}: \lVert \bm{z} \rVert_p \leq \alpha}{\sup}\{\ell_{\bm{\beta}}(\bm{x} + \bm{z}, y)\} \\
        \iff & \displaystyle  \underset{\bm{z}: \lVert \bm{z} \rVert_p \leq \alpha}{\sup} \{ \log(1 + \exp(-y \cdot \bm{\beta}^\top (\bm{x} + \bm{z}))) \} \\
    \iff & \displaystyle \log\left(1 + \exp\left(  \underset{\bm{z}: \lVert \bm{z} \rVert_p \leq \alpha}{\sup} \{ -y \cdot \bm{\beta}^\top (\bm{x} + \bm{z}) \}\right)\right) \\
    \iff & \displaystyle \log\left(1 + \exp\left(  - y\cdot \bm{\beta}^\top \bm{x} +  \alpha \cdot \underset{\bm{z}: \lVert \bm{z} \rVert_p \leq 1}{\sup} \{ -y \cdot \bm{\beta}^\top \bm{z} \}\right)\right) \\
    \iff & \displaystyle \log(1 + \exp(  - y\cdot \bm{\beta}^\top \bm{x} +  \alpha \cdot \lVert -y \cdot \bm{\beta} \rVert_{p^\star} )) \\
    \iff & \displaystyle \log(1 + \exp(  - y\cdot \bm{\beta}^\top \bm{x} +  \alpha \cdot \lVert \bm{\beta} \rVert_{p^\star} )),
    \end{align*}
where the first step follows from the definition of logloss, the second step follows from the fact that $\log$ and $\exp$ are increasing functions, the third step takes the constant terms out of the $\sup$ problem and exploits the fact that the optimal solution of maximizing a linear function will be at an extreme point of the $\ell_p$-ball, the fourth step uses the definition of dual norm, and finally, the redundant $-y\in \{-1,+1\}$ is omitted from the dual norm. We can therefore define the adversarial loss $\ell^{\alpha}_{\bm{\beta}}(\bm{x}, y) := \log(1 + \exp(  - y\cdot \bm{\beta}^\top \bm{x} +  \alpha \cdot \lVert \bm{\beta} \rVert_{p^\star} ))$ where $\alpha$ models the strength of the adversary, $\bm{\beta}$ is the decision vector, and $(\bm{x}, y)$ is an instance. Replacing ${\sup_{\bm{z}: \lVert \bm{z} \rVert_p \leq \alpha}}\{\ell_{\bm{\beta}}(\bm{x} + \bm{z}, y)\}$ in~\ref{advdro} with $\ell^{\alpha}_{\bm{\beta}}(\bm{x}, y)$ concludes the equivalence of the optimization problem.

Furthermore, to see $\mathrm{Lip}(L^{\alpha}) = 1$, firstly note that since $L^{\alpha}(z) = \log(1 + \exp(-z + \alpha\cdot \lVert \bm{\beta} \Vert_{p^\star}))$ is differentiable everywhere in $z$ and its gradient  ${L^{\alpha}}'$ is bounded everywhere, we have that $\mathrm{Lip}(L^{\alpha})$ is equal to $\sup_{z \in \mathbb{R}} \{ |{L^{\alpha}}'(z)| \}$. We thus have:
    \begin{align*}
        {L^{\alpha}}'(z) = \dfrac{-\exp(-z + \alpha \cdot \lVert \bm{\beta} \rVert_{p^\star})}{1 + \exp(-z + \alpha \cdot \lVert \bm{\beta} \rVert_{p^\star})} = \dfrac{-1}{1 + \exp(z - \alpha \cdot \lVert \bm{\beta} \rVert_{p^\star})} \in (-1, 0)
    \end{align*}
and $|{L^{\alpha}}'(z)| = \left[ 1 + \exp(z - \alpha \cdot \lVert \bm{\beta} \rVert_{p^\star}) \right]^{-1} \longrightarrow 1$ as $z \longrightarrow -\infty$.  \qed

\subsection{Proof of Corollary~\ref{corr:tractable_og}}
Observation~\ref{obs:prelim} lets us represent~\ref{advdro} as the DR counterpart of empirical minimization of $\ell^{\alpha}_{\bm{\beta}}$:
\begin{align}\label{adversarial}
    \begin{array}{cl}
        \displaystyle \underset{\bm{\beta}}{\mathrm{minimize}} & \displaystyle \underset{\mathbb{Q} \in \mathfrak{B}_{\varepsilon}({\mathbb{P}}_N)}{\sup} \quad  \mathbb{E}_{\mathbb{Q}} \left[  \ell^{\alpha}_{\bm{\beta}}(\bm{x},y) \right] \\
        \mathrm{subject\;to} & \bm{\beta} \in \mathbb{R}^n.
    \end{array}
\end{align}
Since the univariate loss $L^{\alpha}(z) := \log(1+ \exp(-z + \alpha \cdot \lVert \bm{\beta}\rVert_{p^\star}))$ satisfying the identity $L^{\alpha}(\langle y\cdot \bm{x}, \bm{\beta}\rangle) = \ell^{\alpha}_{\bm{\beta}}(\bm{x}, y)$ is Lipschitz continuous, Theorem 14~\textit{(ii)} of \cite{shafieezadeh2019regularization} is immediately applicable. We can therefore rewrite~\eqref{adversarial} as:
\begin{align*}
          \begin{array}{cll}
            \underset{\bm{\beta},\ \lambda, \ \bm{s}}{\mathrm{minimize}} & \displaystyle \lambda \cdot \varepsilon + \dfrac{1}{N} \sum_{i\in[N]} s_i &  \\
            \mathrm{subject\;to} & L^{\alpha}(\langle y^i \cdot \bm{x}, \bm{\beta}\rangle ) \leq s_i & \forall i \in [N] \\
            & L^{\alpha}(\langle -y^i \cdot \bm{x}, \bm{\beta}\rangle ) - \lambda \cdot \kappa \leq s_i & \forall i \in [N] \\
            & \mathrm{Lip}(L^\alpha) \cdot \lVert \bm{\beta} \rVert_{q^\star} \leq \lambda & \\
            & \bm{\beta} \in \mathbb{R}^n, \ \lambda \geq 0, \ \bm{s} \in \mathbb{R}^N.
        \end{array}
\end{align*}
Replacing $\mathrm{Lip}(L^\alpha) = 1$ and substituting the definition of $L^\alpha$ concludes the proof. \qed

\subsection{Proof of Proposition~\ref{prop:reformulate}}\label{appendix_proof_prop:reformulate}
We prove Proposition~\ref{prop:reformulate} by constructing the optimization problem in its statement. We will thus dualize the inner $\sup$-problem of~\ref{synth} for fixed $\bm{\beta}$. To this end, we present a sequence of reformulations to the inner problem and then exploit strong semi-infinite duality.

By interchanging $\bm{\xi} = (\bm{x}, y)$, we first rewrite the inner problem as
\begin{align*}
    \begin{array}{cll}
        \displaystyle \underset{\mathbb{Q}, \Pi, \widehat{\Pi}}{\mathrm{maximize}} & \displaystyle \int_{\bm{\xi} \in \Xi} \ell^{\alpha}_{\bm{\beta}}(\bm{\xi}) \mathbb{Q}(\diff \bm{\xi}) & \\[5mm]
        \mathrm{subject\;to} & \displaystyle \int_{\bm{\xi}, \bm{\xi'} \in \Xi^2} d(\bm{\xi}, \bm{\xi'}) \Pi(\diff \bm{\xi}, \diff \bm{\xi'}) \leq \varepsilon & \\[5mm]
        &\displaystyle \int_{\bm{\xi} \in \Xi} \Pi(\diff \bm{\xi}, \diff \bm{\xi'}) = \mathbb{P}_N(\diff \bm{\xi'}) & \forall \bm{\xi'} \in \Xi \\[5mm]
        &\displaystyle \int_{\bm{\xi'} \in \Xi} \Pi(\diff \bm{\xi}, \diff \bm{\xi'}) = \mathbb{Q}(\diff \bm{\xi}) & \forall \bm{\xi} \in \Xi \\[5mm]
        & \displaystyle \int_{\bm{\xi}, \bm{\xi'} \in \Xi^2} d(\bm{\xi}, \bm{\xi'}) \widehat{\Pi}(\diff \bm{\xi}, \diff \bm{\xi'}) \leq \widehat{\varepsilon} & \\[5mm]
        &\displaystyle \int_{\bm{\xi} \in \Xi} \widehat{\Pi}(\diff \bm{\xi}, \diff \bm{\xi'}) = \widehat{\mathbb{P}}_{\widehat{N}}(\diff \bm{\xi'}) & \forall \bm{\xi'} \in \Xi \\[5mm]
        &\displaystyle \int_{\bm{\xi'} \in \Xi} \widehat{\Pi}(\diff \bm{\xi}, \diff \bm{\xi'}) = \mathbb{Q}(\diff \bm{\xi}) & \forall \bm{\xi} \in \Xi \\[5mm]
        & \mathbb{Q} \in \mathcal{P}(\Xi), \; \Pi \in \mathcal{P}(\Xi^2), \; \widehat{\Pi} \in \mathcal{P}(\Xi^2).
    \end{array}
\end{align*}
Here, the first three constraints specify that $\mathbb{Q}$ and $\mathbb{P}_N$ have a Wasserstein distance bounded by $\varepsilon$ from each other, modeled via their coupling $\Pi$. The latter three constraints similarly specify that $\mathbb{Q}$ and $\widehat{\mathbb{P}}_{\widehat{N}}$ are at most $\widehat{\varepsilon}$ away from each other, modeled via their coupling $\widehat{\Pi}$. As $\mathbb{Q}$ lies in the intersection of two Wasserstein balls in~\ref{synth}, the marginal $\mathbb{Q}$ is shared between $\Pi$ and $\widehat{\Pi}$. We can now substitute the third constraint into the objective and the last constraint and obtain:
\begin{align*}
    \begin{array}{cll}
        \displaystyle \underset{\Pi, \widehat{\Pi}}{\mathrm{maximize}} & \displaystyle \int_{\bm{\xi} \in \Xi} \ell^{\alpha}_{\bm{\beta}}(\bm{\xi}) \int_{\bm{\xi'} \in \Xi} \Pi(\diff \bm{\xi}, \diff \bm{\xi'}) & \\[5mm]
        \mathrm{subject\;to} & \displaystyle \int_{\bm{\xi}, \bm{\xi'} \in \Xi^2} d(\bm{\xi}, \bm{\xi'}) \Pi(\diff \bm{\xi}, \diff \bm{\xi'}) \leq \varepsilon & \\[5mm]
        &\displaystyle \int_{\bm{\xi} \in \Xi} \Pi(\diff \bm{\xi}, \diff \bm{\xi'}) = \mathbb{P}_N(\diff \bm{\xi'}) & \forall \bm{\xi'} \in \Xi \\[5mm]
        & \displaystyle \int_{\bm{\xi}, \bm{\xi'} \in \Xi^2} d(\bm{\xi}, \bm{\xi'}) \widehat{\Pi}(\diff \bm{\xi}, \diff \bm{\xi'}) \leq \widehat{\varepsilon} & \\[5mm]
        &\displaystyle \int_{\bm{\xi} \in \Xi} \widehat{\Pi}(\diff \bm{\xi}, \diff \bm{\xi'}) = \widehat{\mathbb{P}}_{\widehat{N}}(\diff \bm{\xi'}) & \forall \bm{\xi'} \in \Xi \\[5mm]
        &\displaystyle \int_{\bm{\xi'} \in \Xi} \widehat{\Pi}(\diff \bm{\xi}, \diff \bm{\xi'}) = \int_{\bm{\xi'} \in \Xi} \Pi(\diff \bm{\xi}, \diff \bm{\xi'}) & \forall \bm{\xi} \in \Xi \\[5mm]
        & \Pi \in \mathcal{P}(\Xi^2), \ \widehat{\Pi} \in \mathcal{P}(\Xi^2).
    \end{array}
\end{align*}
Denoting by $\mathbb{Q}^i(\diff \bm{\xi}) := \Pi(\diff \bm{\xi} \mid \bm{\xi}^i)$ the conditional distribution of $\Pi$ upon the realization of $\bm{\xi'} = \bm{\xi}^i$ and exploiting the fact that ${\mathbb{P}}_N$ is a discrete distribution supported on the $N$ data points $\{\bm{\xi}^i\}_{i \in [N]}$, we can use the marginalized representation $ \Pi(\diff \bm{\xi},\diff \bm{\xi'}) =\frac{1}{N} \sum_{i=1}^N \delta_{\bm{\xi}^i}(\diff \bm{\xi}')\mathbb{Q}^i(\diff \bm{\xi})$. Similarly, we can introduce $\widehat{\mathbb{Q}}^i(\diff \bm{\xi})  := \widehat{\Pi}(\diff \bm{\xi} \mid \widehat{\bm{\xi}}^i) $ for $\{\widehat{\bm{\xi}}^i\}_{i \in [\widehat{N}]}$ to exploit the marginalized representation $ \widehat{\Pi}(\diff \bm{\xi},\diff \bm{\xi'}) =\frac{1}{\widehat{N}} \sum_{j=1}^{\widehat N} \delta_{\widehat{\bm{\xi}}^j}(\diff \bm{\xi}')\widehat{\mathbb{Q}}^j(\diff \bm{\xi})$. By using this marginalization representation, we can use the following simplification for the objective function:
\begin{align*}
\hspace{-0.5cm}
   \displaystyle \int_{\bm{\xi} \in \Xi} \ell^{\alpha}_{\bm{\beta}}(\bm{\xi}) \int_{\bm{\xi'} \in \Xi} \Pi(\diff \bm{\xi}, \diff \bm{\xi'}) \; & = \; \displaystyle \frac{1}{N} \sum_{i=1}^N  \int_{\bm{\xi} \in \Xi} \ell^{\alpha}_{\bm{\beta}}(\bm{\xi}) \int_{\bm{\xi'} \in \Xi} \delta_{\bm{\xi}^i}(\diff \bm{\xi}')\mathbb{Q}^i(\diff \bm{\xi}) \; = \; \displaystyle \frac{1}{N} \sum_{i=1}^N  \int_{\bm{\xi} \in \Xi} \ell^{\alpha}_{\bm{\beta}}(\bm{\xi}) \mathbb{Q}^i(\diff \bm{\xi}).
\end{align*}
Applying analogous reformulations to the constraints leads to the following reformulation of the inner $\sup$ problem of~\ref{synth}:
\begin{align*}
    \begin{array}{cll}
        \displaystyle \underset{\mathbb{Q}, \widehat{\mathbb{Q}}}{\mathrm{maximize}} & \displaystyle \dfrac{1}{N} \sum_{i=1}^N \int_{\bm{\xi} \in \Xi} \ell^{\alpha}_{\bm{\beta}}(\bm{\xi}) \mathbb{Q}^i(\diff \bm{\xi}) & \\[5mm]
        \mathrm{subject\;to} & \displaystyle \dfrac{1}{N} \sum_{i=1}^N \int_{\bm{\xi} \in \Xi} d(\bm{\xi}, \bm{\xi}^i) \mathbb{Q}^i (\diff \bm{\xi}) \leq \varepsilon & \\[5mm]
        & \displaystyle \dfrac{1}{\widehat{N}} \sum_{j=1}^{\widehat N} \int_{\bm{\xi} \in \Xi} d(\bm{\xi}, \widehat{\bm{\xi}}^j) \widehat{\mathbb{Q}}^j (\diff \bm{\xi}) \leq \widehat{\varepsilon} & \\[5mm]
        &\displaystyle \dfrac{1}{N} \sum_{i=1}^N \mathbb{Q}^i(\diff \bm{\xi}) = \dfrac{1}{\widehat{N}} \sum_{j=1}^{\widehat{N}} \widehat{\mathbb{Q}}^j(\diff \bm{\xi}) & \forall \bm{\xi} \in \Xi  \\[5mm]
        & \mathbb{Q}^i \in \mathcal{P}(\Xi), \ \widehat{\mathbb{Q}}^{j} \in \mathcal{P}(\Xi) & \forall i \in [N], \ \forall j \in [\widehat{N}].
    \end{array}
\end{align*}
We now decompose each $\mathbb{Q}^i$ into two measures corresponding to $y = \pm 1$, so that $\mathbb{Q}^i(\diff (\bm{x}, y)) = \mathbb{Q}_{+1}^i(\diff \bm{x})$ for $y = +1$ and $\mathbb{Q}^i(\diff (\bm{x}, y)) = \mathbb{Q}_{-1}^i(\diff \bm{x})$ for $y = -1$. We similarly represent each $\widehat{\mathbb{Q}}^j$ via $\widehat{\mathbb{Q}}_{+1}^j$ and $\widehat{\mathbb{Q}}_{-1}^j$ depending on $y$. Note that these new measures are not probability measures as they do not integrate to $1$, but non-negative measures supported on $\mathbb{R}^n$ (denoted $\in \mathcal{P}_{+}(\mathbb{R}^n)$). We get:
\begin{align*}
    \begin{array}{cll}
        \displaystyle \underset{\mathbb{Q}_{\pm 1}, \widehat{\mathbb{Q}}_{\pm 1}}{\mathrm{maximize}} & \displaystyle \dfrac{1}{N} \sum_{i=1}^N \int_{\bm{x} \in \mathbb{R}^n} [ \ell^{\alpha}_{\bm{\beta}}(\bm{x}, + 1) \mathbb{Q}_{+1}^i (\diff \bm{x}) + \ell^{\alpha}_{\bm{\beta}}(\bm{x}, -1)\mathbb{Q}_{-1}^i(\diff \bm{x}) ] & \\[5mm]
        \mathrm{subject\;to} & \displaystyle \dfrac{1}{N} \sum_{i=1}^N \int_{\bm{x} \in \mathbb{R}^n} [d( (\bm{x},+1), \bm{\xi}^i) \mathbb{Q}_{+1}^i(\diff \bm{x}) + d( (\bm{x},-1), \bm{\xi}^i) \mathbb{Q}_{-1}^i(\diff \bm{x})]  \leq \varepsilon & \\[5mm]
        & \displaystyle \dfrac{1}{\widehat{N}} \sum_{j=1}^{\widehat{N}} \int_{\bm{x} \in \mathbb{R}^n} [d( (\bm{x},+1), \widehat{\bm{\xi}}^j) \widehat{\mathbb{Q}}_{+1}^j(\diff \bm{x}) + d( (\bm{x},-1), \widehat{\bm{\xi}}^j) \widehat{\mathbb{Q}}_{-1}^j(\diff \bm{x})]  \leq \widehat{\varepsilon} & \\[5mm]
        & \displaystyle \int_{\bm{x} \in \mathbb{R}^n} \mathbb{Q}_{+1}^i(\diff \bm{x}) + \mathbb{Q}_{-1}^i(\diff \bm{x}) = 1 & \forall i \in [N] \\[5mm]
        & \displaystyle \int_{\bm{x} \in \mathbb{R}^n} \widehat{\mathbb{Q}}_{+1}^j(\diff \bm{x}) + \widehat{\mathbb{Q}}_{-1}^j(\diff \bm{x}) = 1 & \forall j \in [\widehat{N}] \\[5mm]
        &\displaystyle \dfrac{1}{N} \sum_{i=1}^N \mathbb{Q}^i_{+1}(\diff \bm{x}) = \dfrac{1}{\widehat{N}} \sum_{j=1}^{\widehat{N}} \widehat{\mathbb{Q}}^j_{+1}(\diff \bm{x}) & \forall \bm{x} \in \mathbb{R}^n  \\[5mm]
        &\displaystyle \dfrac{1}{N} \sum_{i=1}^N \mathbb{Q}^i_{-1}(\diff \bm{x}) = \dfrac{1}{\widehat{N}} \sum_{j=1}^{\widehat{N}} \widehat{\mathbb{Q}}^j_{-1}(\diff \bm{x}) & \forall \bm{x} \in \mathbb{R}^n \\[5mm]
        & \mathbb{Q}^i_{\pm 1} \in \mathcal{P}_{+}(\mathbb{R}^n), 
        \ \widehat{\mathbb{Q}}^j_{\pm 1} \in \mathcal{P}_{+}(\mathbb{R}^n) & \forall i \in [N], \ j \in [\widehat{N}].
    \end{array}
\end{align*}
Next, we explicitly write the definition of the metric $d(\cdot, \cdot)$ in the first two constraints as well as use auxiliary measures $\mathbb{A}_{\pm 1}\in\mathcal{P}_{+}(\mathbb{R}^n)$ to break down the last two equality constraints:

{\allowdisplaybreaks 
\begin{align*}
    \begin{array}{cll}
        \displaystyle \underset{\mathbb{A}_{\pm 1}, \mathbb{Q}_{\pm 1}, \widehat{\mathbb{Q}}_{\pm 1}}{\mathrm{maximize}} & \displaystyle \dfrac{1}{N} \sum_{i=1}^N \int_{\bm{x} \in \mathbb{R}^n} [ \ell^{\alpha}_{\bm{\beta}}(\bm{x}, + 1) \mathbb{Q}_{+1}^i (\diff \bm{x}) + \ell^{\alpha}_{\bm{\beta}}(\bm{x}, -1)\mathbb{Q}_{-1}^i(\diff \bm{x}) ] & \\[5mm]
        \mathrm{subject\;to} & \displaystyle \dfrac{1}{N} \int_{\bm{x} \in \mathbb{R}^n} \Big[  \kappa \cdot \sum_{i\in [N] : y^i = -1} \mathbb{Q}_{+1}^i(\diff \bm{x}) + \kappa \cdot \sum_{i\in [N] : y^i = +1} \mathbb{Q}_{-1}^i(\diff \bm{x}) +  \\[3mm]
        & 
        {\displaystyle \sum_{i = 1}^N \lVert \bm{x} - \bm{x}^i \rVert_{q} \cdot[ \mathbb{Q}_{+1}^i(\diff \bm{x}) + \mathbb{Q}_{-1}^i(\diff \bm{x}) ] \Big] \leq \varepsilon} & \\[5mm]
        & \displaystyle \dfrac{1}{\widehat{N}} \int_{\bm{x} \in \mathbb{R}^n} \Big[  \kappa \cdot \sum_{j\in [N] : \widehat{y}^j = -1} \widehat{\mathbb{Q}}_{+1}^j(\diff \bm{x}) + \kappa \cdot \sum_{j\in [N] : \widehat{y}^j = +1} \widehat{\mathbb{Q}}_{-1}^j(\diff \bm{x}) +  \\[3mm]
        & 
        {\displaystyle \sum_{j = 1}^{\widehat{N}} \lVert \bm{x} - \widehat{\bm{x}}^j\rVert_{q} \cdot [ \widehat{\mathbb{Q}}_{+1}^j(\diff \bm{x}) + \widehat{\mathbb{Q}}_{-1}^j(\diff \bm{x}) ] \Big] \leq \widehat{\varepsilon}} & \\[5mm]
        & \displaystyle \int_{\bm{x} \in \mathbb{R}^n} \mathbb{Q}_{+1}^i(\diff \bm{x}) + \mathbb{Q}_{-1}^i(\diff \bm{x}) = 1 & \forall i \in [N] \\[5mm]
        & \displaystyle \int_{\bm{x} \in \mathbb{R}^n} \widehat{\mathbb{Q}}_{+1}^j(\diff \bm{x}) + \widehat{\mathbb{Q}}_{-1}^j(\diff \bm{x}) = 1 & \forall j \in [\widehat{N}] \\[5mm]
        &\displaystyle \dfrac{1}{N} \sum_{i=1}^N \mathbb{Q}^i_{+1}(\diff \bm{x}) = \mathbb{A}_{+1}(\diff \bm{x}) & \forall \bm{x} \in \mathbb{R}^n  \\[5mm]
        &\displaystyle \dfrac{1}{\widehat{N}} \sum_{j=1}^{\widehat{N}} \widehat{\mathbb{Q}}^j_{+1}(\diff \bm{x}) = \mathbb{A}_{+1}(\diff \bm{x}) & \forall \bm{x} \in \mathbb{R}^n  \\[5mm]
        &\displaystyle \dfrac{1}{N} \sum_{i=1}^N \mathbb{Q}^i_{-1}(\diff \bm{x}) = \mathbb{A}_{-1}(\diff \bm{x}) & \forall \bm{x} \in \mathbb{R}^n  \\[5mm]
        &\displaystyle \dfrac{1}{\widehat{N}} \sum_{j=1}^{\widehat{N}} \widehat{\mathbb{Q}}^j_{-1}(\diff \bm{x}) = \mathbb{A}_{-1}(\diff \bm{x}) & \forall \bm{x} \in \mathbb{R}^n  \\[5mm]
        & \mathbb{A}_{\pm 1} \in \mathcal{P}_{+}(\mathbb{R}^n), \ \mathbb{Q}^i_{\pm 1} \in \mathcal{P}_{+}(\mathbb{R}^n), 
        \ \widehat{\mathbb{Q}}^j_{\pm 1} \in \mathcal{P}_{+}(\mathbb{R}^n) & \forall i \in [N], \ j \in [\widehat{N}].
    \end{array}
\end{align*}
}
The following semi-infinite optimization problem, obtained by standard algebraic duality, is a strong dual to the above problem since $\varepsilon, \widehat{\varepsilon} > 0$ \citep{shapiro2001duality}. \newpage

\begin{align*}
    \begin{array}{cll}
        \displaystyle \underset{\lambda, \widehat{\lambda}, \bm{s}, \widehat{\bm{s}}, p_{\pm 1}, \widehat{p}_{\pm 1}}{\mathrm{minimize}}  & \displaystyle \dfrac{1}{N}\left[ N\varepsilon \lambda + \widehat{N} \widehat{\varepsilon}  \widehat{\lambda} + \sum_{i=1}^N s_i + \sum_{j=1}^{\widehat{N}} \widehat{s}_j \right] & \\[5mm]
        \mathrm{subject\;to} & \displaystyle \kappa  \dfrac{1 - y^i}{2}  \lambda + \lambda \lVert \bm{x}^i - \bm{x} \rVert_{q}  + s_i + \dfrac{p_{+1}(x)}{N} \geq \ell^{\alpha}_{\bm{\beta}}(\bm{x}, + 1) & \forall i \in [N] , \ \forall \bm{x} \in \mathbb{R}^n \\[5mm] 
        & \displaystyle \kappa  \dfrac{1 - \widehat{y}^j}{2}  \widehat{\lambda} + \widehat{\lambda} \lVert \widehat{\bm{x}}^j - \bm{x} \rVert_{q} + \widehat{s}_j + \dfrac{\widehat{p}_{+1}(x)}{\widehat{N}} \geq 0 & \forall j \in [\widehat{N}] , \ \forall \bm{x} \in \mathbb{R}^n \\[5mm]
        & \displaystyle \kappa  \dfrac{1 + y^i}{2}  \lambda + \lambda \lVert \bm{x}^i -  \bm{x} \rVert_{q} + s_i + \dfrac{p_{-1}(x)}{N} \geq \ell^{\alpha}_{\bm{\beta}}(\bm{x}, - 1) & \forall i \in [N] , \ \forall \bm{x} \in \mathbb{R}^n \\[5mm]
         & \displaystyle \kappa  \dfrac{1 + \widehat{y}^j}{2}  \widehat{\lambda} + \widehat{\lambda} \lVert \widehat{\bm{x}}^j - \bm{x} \rVert_{q} + \widehat{s}_j + \dfrac{\widehat{p}_{-1}(x)}{\widehat{N}} \geq 0 & \forall j \in [\widehat{N}] , \ \forall \bm{x} \in \mathbb{R}^n \\[5mm]
         & p_{+1}(\bm{x}) + \widehat{p}_{+1}(\bm{x}) \leq 0 & \\[5mm]
         & p_{-1}(\bm{x}) + \widehat{p}_{-1}(\bm{x}) \leq 0 & \\[5mm]
        &\displaystyle \lambda \in \mathbb{R}_{+}, \ \widehat{\lambda} \in \mathbb{R}{+}, \ \bm{s} \in \mathbb{R}^{N}, \ \widehat{\bm{s}} \in \mathbb{R}^{\widehat{N}} & \\[5mm]
        &\displaystyle p_{\pm 1}: \mathbb{R}^n \mapsto \mathbb{R}, \  \widehat{p}_{\pm 1}: \mathbb{R}^n \mapsto \mathbb{R}.
    \end{array}
\end{align*}

To eliminate the (function) variables $p_{+1}$ and $\widehat{p}_{+1}$, we first summarize the constraints they appear
\begin{align*}
    & \begin{cases}
        p_{+1}(\bm{x}) \geq N \cdot \left[\ell^{\alpha}_{\bm{\beta}}(\bm{x}, + 1) - s_i - \lambda  \lVert \bm{x}^i - \bm{x} \rVert_{q} - \kappa  \dfrac{1 - y^i}{2}  \lambda \right] & \forall i \in [N], \ \forall \bm{x} \in \mathbb{R}^n \\[3mm]
        \widehat{p}_{+1}(\bm{x}) \geq \widehat{N} \cdot \left[ -\widehat{s}_j - \widehat{\lambda}  \lVert \widehat{\bm{x}}^j - \bm{x} \rVert_{q} - \kappa  \dfrac{1 - \widehat{y}^j}{2}  \widehat{\lambda} \right] & \forall j \in [\widehat{N}] , \ \forall \bm{x} \in \mathbb{R}^n \\
        p_{+1}(\bm{x}) + \widehat{p}_{+1}(\bm{x}) \leq 0 & \forall \bm{x} \in \mathbb{R}^n,
    \end{cases}
\end{align*}
and notice that this system is equivalent to the epigraph-based reformulation of the following constraint 
\begin{align*}
&& \ell^{\alpha}_{\bm{\beta}}(\bm{x}, + 1) - s_i - \lambda  \lVert \bm{x}^i - \bm{x} \rVert_{q} - \kappa  \dfrac{1 - y^i}{2}  \lambda + {\dfrac{\widehat{N}}{N} \cdot \left[ -\widehat{s}_j- \widehat{\lambda}  \lVert \widehat{\bm{x}}^j - \bm{x} \rVert_{q} - \kappa  \dfrac{1 - \widehat{y}^j}{2}  \widehat{\lambda} \right] \leq 0}&   \\
&&  \hfill \forall i \in [N], \ \forall j \in [\widehat{N}], \ \forall \bm{x} \in \mathbb{R}^n.
\end{align*}
We can therefore eliminate $p_{+1}$ and $\widehat{p}_{+1}$. We can also eliminate $p_{-1}$ and $\widehat{p}_{-1}$ since we similarly have:
\begin{align*}
    & \begin{cases}
        p_{-1}(\bm{x}) \geq N \cdot \left[\ell^{\alpha}_{\bm{\beta}}(\bm{x}, - 1) - s_i - \lambda \lVert \bm{x}^i - \bm{x} \rVert_{q} - \kappa  \dfrac{1 + y^i}{2}  \lambda \right] & \forall i \in [N], \ \forall \bm{x} \in \mathbb{R}^n \\[3mm]
        \widehat{p}_{-1}(\bm{x}) \geq \widehat{N} \cdot \left[ -\widehat{s}_j - \widehat{\lambda}  \lVert \widehat{\bm{x}}^j - \bm{x} \rVert_{q} - \kappa  \dfrac{1 + \widehat{y}^j}{2}  \widehat{\lambda} \right] & \forall j \in [\widehat{N}] , \ \forall \bm{x} \in \mathbb{R}^n \\
        p_{-1}(\bm{x}) + \widehat{p}_{-1}(\bm{x}) \leq 0 & \forall \bm{x} \in \mathbb{R}^n
    \end{cases}\\
 \iff &  \ell^{\alpha}_{\bm{\beta}}(\bm{x}, - 1) - s_i - \lambda  \lVert \bm{x}^i - \bm{x} \rVert_{q} - \kappa  \dfrac{1 + y^i}{2}  \lambda + {\dfrac{\widehat{N}}{N} \cdot \left[ -\widehat{s}_j- \widehat{\lambda}  \lVert \widehat{\bm{x}}^j - \bm{x} \rVert_{q} - \kappa  \dfrac{1 + \widehat{y}^j}{2}  \widehat{\lambda} \right] \leq 0} \\
& 
{\forall i \in [N], \ \forall j \in [\widehat{N}], \ \forall \bm{x} \in \mathbb{R}^n.}
\end{align*}

This trick of eliminating $p_{\pm 1}, \ \widehat{p}_{\pm 1}$ is due to the auxiliary distributions $\mathbb{A}_{\pm 1}$ that we introduced; without them, the dual problem is substantially harder to work with. We therefore obtain the following reformulation of the dual problem
\begin{align*}
    \begin{array}{cll}
        \displaystyle \underset{\lambda, \widehat{\lambda}, \bm{s}, \widehat{\bm{s}}}{\mathrm{minimize}}  & \displaystyle \dfrac{1}{N}\left[ N\varepsilon \lambda + \widehat{N} \widehat{\varepsilon}  \widehat{\lambda} + \sum_{i=1}^N s_i + \sum_{j=1}^{\widehat{N}} \widehat{s}_j \right] & \\[5mm]
        \mathrm{subject\;to} & \displaystyle \underset{\bm{x} \in \mathbb{R}^n}{\sup} \{\ell^{\alpha}_{\bm{\beta}}(\bm{x}, + 1) - \lambda  \lVert \bm{x}^i - \bm{x} \rVert_{q} - \dfrac{\widehat{N}}{N}  \widehat{\lambda}  \lVert \widehat{\bm{x}}^j - \bm{x} \rVert_{q}\} \leq \\[3mm]
        & 
        {\quad s_i + \kappa  \dfrac{1 - y^i}{2} \lambda + \dfrac{\widehat{N}}{N}\cdot \left[ \widehat{s}_j + \kappa  \dfrac{1 - \widehat{y}^j}{2} \widehat{\lambda} \right]} & \forall i \in [N], \ \forall j \in [\widehat{N}]  \\[5mm]
        & \displaystyle \underset{\bm{x} \in \mathbb{R}^n}{\sup}\{\ell^{\alpha}_{\bm{\beta}}(\bm{x}, - 1) - \lambda  \lVert \bm{x}^i - \bm{x} \rVert_{q} - \dfrac{\widehat{N}}{N} \widehat{\lambda}   \lVert \widehat{\bm{x}}^j - \bm{x} \rVert_{q}\} \leq \\[3mm]
        & 
        {\quad s_i + \kappa \dfrac{1 + y^i}{2} \lambda + \dfrac{\widehat{N}}{N}\cdot \left[ \widehat{s}_j + \kappa \dfrac{1 + \widehat{y}^j}{2} \widehat{\lambda} \right]} & \forall i \in [N], \ \forall j \in [\widehat{N}] \\[5mm]
        & \lambda \geq 0, \ \widehat{\lambda}\geq 0, \ \bm{s} \in \mathbb{R}^{N}_{+}, \ \widehat{\bm{s}} \in \mathbb{R}^{\widehat{N}}_{+},
    \end{array}
\end{align*}
where we replaced the $\forall \bm{x} \in \mathbb{R}^n$ with the worst case realizations by taking the suprema of the constraints over $\bm{x}$. We also added non-negativity on the definition of $\bm{s}$ and $\widehat{\bm{s}}$ which is without loss of generality since this is implied by the first two constraints, which is due to the fact that in the primal reformulation the ``integrates to $1$'' constraints (whose associated dual variables are $\bm{s}$ and $\widehat{\bm{s}}$) can be written as
\begin{align*}
    \begin{array}{lll}
 & \displaystyle \int_{\bm{x} \in \mathbb{R}^n} \mathbb{Q}_{+1}^i(\diff \bm{x}) + \mathbb{Q}_{-1}^i(\diff \bm{x}) \leq 1 & \forall i \in [N], \\[5mm]
  & \displaystyle \int_{\bm{x} \in \mathbb{R}^n} \widehat{\mathbb{Q}}_{+1}^j(\diff \bm{x}) + \widehat{\mathbb{Q}}_{-1}^j(\diff \bm{x}) \leq 1 & \forall j \in [\widehat{N}],
    \end{array}
\end{align*}
due to the objective pressure. Relabeling $\dfrac{\widehat{N}}{N} \widehat{\lambda}$ as $\widehat{\lambda}$ and $\dfrac{\widehat{N}}{N} \widehat{s}_j$ as $\widehat{s}_j$ simplifies the problem to:
\begin{align*}
    \begin{array}{cll}
        \displaystyle \underset{\lambda, \widehat{\lambda}, \bm{s}, \widehat{\bm{s}}}{\mathrm{minimize}}  & \displaystyle \varepsilon \lambda + \widehat{\varepsilon}  \widehat{\lambda} + \dfrac{1}{N}\sum_{i=1}^N s_i + \dfrac{1}{\widehat{N}}\sum_{i=1}^{\widehat{N}} \widehat{s}_j & \\[5mm]
        \mathrm{subject\;to} & \displaystyle \underset{\bm{x} \in \mathbb{R}^n}{\sup} \{\ell^{\alpha}_{\bm{\beta}}(\bm{x}, + 1) - \lambda  \lVert \bm{x}^i - \bm{x} \rVert_{q} - \widehat{\lambda}  \lVert \widehat{\bm{x}}^j - \bm{x} \rVert_{q}\} \leq \\[3mm]
        & 
        {\quad s_i + \kappa  \dfrac{1 - y^i}{2} \lambda + \widehat{s}_j + \kappa  \dfrac{1 - \widehat{y}^j}{2}  \widehat{\lambda}} & \forall i \in [N], \ \forall j \in [\widehat{N}]  \\[5mm]
        & \displaystyle \underset{\bm{x} \in \mathbb{R}^n}{\sup}\{\ell^{\alpha}_{\bm{\beta}}(\bm{x}, - 1) - \lambda  \lVert \bm{x}^i - \bm{x} \rVert_{q} - \widehat{\lambda}   \lVert \widehat{\bm{x}}^j - \bm{x} \rVert_{q}\} \leq \\[3mm]
        & 
        {\quad s_i + \kappa  \dfrac{1 + y^i}{2} \lambda + \widehat{s}_j + \kappa  \dfrac{1 + \widehat{y}^j}{2}  \widehat{\lambda}} & \forall i \in [N], \ \forall j \in [\widehat{N}] \\[5mm]
        & \lambda \geq 0, \ \widehat{\lambda}\geq 0, \ \bm{s} \in \mathbb{R}^{N}_{+}, \ \widehat{\bm{s}} \in \mathbb{R}^{\widehat{N}}_{+}.
    \end{array}
\end{align*}
Combining all the $\sup$-constraints with the help of an an auxiliary parameter $l \in \{-1, 1\}$ and replacing this problem with the inner problem of~\ref{synth} concludes the proof.\qed

\subsection{Proof of Proposition~\ref{prop:np-hard}}\label{proof prop:np-hard}
We first present a technical lemma that will allow us to rewrite a specific type of difference of convex functions (DC) maximization problem that appears in the constraints of~\ref{synth}. 
Rewriting such DC maximization problems is one of the key steps in reformulating Wasserstein DRO problems, and our lemma is inspired from~\citet[Lemma 47]{shafieezadeh2019regularization},~\citet[Theorem 3.8]{shafieezadeh2023new}, and~\citet[Lemma 1]{belbasi2023s} who reformulate maximizing the difference of a convex function and a norm. 
Our DRO problem~\ref{synth}, however, comprises two ambiguity sets, hence the DC term that we investigate will be the difference between a convex function and the sum of \textit{two norms}. 
This requires a new analysis and we will see that~\ref{synth} is NP-hard due to this additional difficulty.
\begin{lemma}\label{lemma:main}
    Suppose that $L: \mathbb{R} \mapsto \mathbb{R}$ is a closed convex function, and $\lVert \cdot \rVert_{q}$ is a norm. For vectors $\bm{\omega}, \bm{a}, \widehat{\bm{a}} \in \mathbb{R}^n$ and scalars $\lambda, \widehat{\lambda} > 0$, we have:
    \begin{align*}
        & \displaystyle \underset{\bm{x} \in \mathbb{R}^n}{\sup} \{L(\bm{\omega}^\top \bm{x}) - \lambda \lVert \bm{a} - \bm{x} \rVert_{q} - \widehat{\lambda} \lVert \bm{\widehat{a}} - \bm{x} \rVert_{q}\} \\
=  \;     & \underset{\theta \in \mathrm{dom}(L^*)}{\sup} \ - L^{*}(\theta) + \theta\cdot\bm{\omega}^\top \bm{a} + \theta \cdot \underset{\bm{z} \in \mathbb{R}^n}{\inf} \{ \bm{z}^\top (\widehat{\bm{a}} - \bm{a}) \ : \ \lvert\theta\rvert \cdot \lVert \bm{\omega} - \bm{z} \rVert_{q^\star} \leq \lambda, \ \lvert\theta\rvert \cdot \lVert \bm{z} \rVert_{q^\star} \leq \widehat{\lambda} \}
    \end{align*}
\end{lemma}
\begin{proof}
    We denote by $f_{\bm{\omega}}(\bm{x}) = \bm{\omega}^\top \bm{x}$ and by $g$ the convex function $g(\bm{x}) = g_1(\bm{x}) + g_2(\bm{x})$ where $g_1(\bm{x}) := \lambda \lVert \bm{a} - \bm{x} \rVert_{q}$ and $g_2(\bm{x}):= \widehat{\lambda} \lVert \widehat{\bm{a}} - \bm{x} \rVert_{q}$, and reformulate the $\sup$ problem as
\begin{align*}
    \underset{\bm{x}\in \mathbb{R}^n}{\sup} \; L(\bm{\omega}^\top \bm{x}) - g(\bm{x}) \; = \;  \underset{\bm{x}\in \mathbb{R}^n}{\sup} \; (L \circ f_{\bm{\omega}}) (\bm{x}) - g(\bm{x}) \; = \; \underset{\bm{z}\in \mathbb{R}^n}{\sup} \;  g^{*}(\bm{z}) - (L \circ f_{\bm{\omega}})^{*}(\bm{z}),
\end{align*}
where the first identity follows from the definition of composition and the second identity employs Toland's duality~\citep{toland1978duality} to rewrite difference of convex functions optimization.

By using infimal convolutions \citep[Theorem 16.4]{R97:convex_analysis}, we can reformulate $g^{*}$:
\begin{align*}
    g^{*}(\bm{z}) &= \underset{\bm{z}_1, \bm{z}_2}{\inf} \{g_1^{*}(\bm{z}_1) + g_2^{*}(\bm{z}_2) \ : \ \bm{z}_1 + \bm{z}_2 = \bm{z}\} \\ 
    & = \underset{\bm{z}_1, \bm{z}_2}{\inf} \{\bm{z}_1^\top \bm{a} + \bm{z}_2^\top \widehat{\bm{a}} \ : \ \bm{z}_1 + \bm{z}_2 = \bm{z}, \ \lVert \bm{z}_1 \rVert_{q^{\star}} \leq \lambda, \ \lVert \bm{z}_2 \rVert_{q^{\star}} \leq \widehat{\lambda} \}, 
\end{align*}
where the second step uses the definitions of $g_1^{*}(\bm{z}_1)$ and $g_2^{*}(\bm{z}_2)$. Moreover, we show
\begin{align*}
    (L \circ f_{\bm{\omega}})^{*}(\bm{z}) \; & = \; \sup_{\bm{x} \in \mathbb{R}^n} \ \bm{z}^\top \bm{x} - L(\bm{\omega}^\top \bm{x}) \\
    & = \; \sup_{t\in \mathbb{R}, \ \bm{x}\in\mathbb{R}^n} \{ \bm{z}^\top \bm{x} - L(t) \ : \ t = \bm{\omega}^\top \bm{x}\} \\
    & = \; \inf_{\theta \in \mathbb{R}} \  \sup_{t\in \mathbb{R}, \ \bm{x}\in\mathbb{R}^n}  \bm{z}^\top \bm{x} - L(t) - \theta \cdot (\bm{\omega}^\top \bm{x} - t) \\
    & = \; \inf_{\theta \in \mathbb{R}} \  \sup_{t\in \mathbb{R}}  \ \sup_{\bm{x}\in\mathbb{R}^n} (\bm{z} - \theta \cdot \bm{\omega})^\top \bm{x} - L(t) + \theta \cdot t \\
    & = \; \inf_{\theta \in \mathbb{R}} \  \sup_{t\in \mathbb{R}} \begin{cases}
    -L(t) + \theta \cdot t & \text{if } \theta \cdot \bm{\omega} = \bm{z}\\
    + \infty & \text{otherwise.}
    \end{cases} 
    \\
    & = \; \inf_{\theta \in \mathbb{R}} \begin{cases}
    L^*(\theta) & \text{if } \theta \cdot \bm{\omega} = \bm{z}\\
    + \infty & \text{otherwise.}
    \end{cases} 
    \\
    & = \; \inf_{\theta \in \mathrm{dom}(L^{*})} \{ L^*(\theta) \ : \ \theta\cdot\bm{\omega} = \bm{z} \},
\end{align*}
where the first identity follows from the definition of the convex conjugate, the second identity introduces an additional variable to make this an equality-constrained optimization problem, the third identity takes the Lagrange dual (which is a strong dual since the problem maximizes a concave objective with a single equality constraint), the fourth identity rearranges the expressions, the fifth identity exploits unboundedness of $\bm{x}$, the sixth identity uses the definition of convex conjugates and the final identity replaces the feasible set $\theta \in \mathbb{R}$ with the domain of $L^\star$ without loss of generality as this is an $\inf$-problem.

Replacing the conjugates allows us to conclude that the maximization problem equals
\begin{align*}
\hspace{-0.8cm}
\begin{array}{ll}
 &\underset{\bm{z} \in \mathbb{R}^n}{\sup} \ g^{*}(\bm{z}) + \underset{\theta \in \mathrm{dom}(L^*)}{\sup} \{- L^{*}(\theta) \ : \ \theta \cdot \bm{\omega} = \bm{z}\}\\
=   \; &  \underset{\bm{z} \in \mathbb{R}^n, \ \theta \in \mathrm{dom}(L^*)}{\sup} \{ g^{*}(\bm{z}) - L^{*}(\theta) \ : \ \theta \cdot \bm{\omega} = \bm{z} \}\\
=   \; & \underset{\theta \in \mathrm{dom}(L^*)}{\sup} \ g^{*}(\theta \cdot \bm{\omega}) - L^{*}(\theta) \\
=   \; & \underset{\theta \in \mathrm{dom}(L^*)}{\sup} \ - L^{*}(\theta) + \underset{\bm{z}_1, \bm{z}_2 \in \mathbb{R}^{n}}{\inf} \{ \bm{z}_1^\top \bm{a} + \bm{z}_2^\top \widehat{\bm{a}} \ : \ \bm{z}_1 + \bm{z}_2 = \theta\cdot \bm{\omega} , \ \lVert \bm{z}_1 \rVert_{q^\star} \leq \lambda, \ \lVert \bm{z}_2 \rVert_{q^\star} \leq \widehat{\lambda} \} \\
=   \; & \underset{\theta \in \mathrm{dom}(L^*)}{\sup} \ - L^{*}(\theta) + \theta \cdot \underset{\bm{z}_1, \bm{z}_2 \in \mathbb{R}^{n}}{\inf} \{ \bm{z}_1^\top \bm{a} + \bm{z}_2^\top \widehat{\bm{a}} \ : \ \bm{z}_1 + \bm{z}_2 = \bm{\omega} , \ \lvert\theta\rvert \cdot \lVert \bm{z}_1 \rVert_{q^\star} \leq \lambda, \ \lvert\theta\rvert \cdot \lVert \bm{z}_2 \rVert_{q^\star} \leq \widehat{\lambda} \} \\
=   \; & \underset{\theta \in \mathrm{dom}(L^*)}{\sup} \ - L^{*}(\theta) + \theta \cdot \bm{\omega}^\top \bm{a} + \theta \cdot \underset{\bm{z} \in \mathbb{R}^n}{\inf} \{ \bm{z}^\top (\widehat{\bm{a}} - \bm{a}) \ : \ \lvert\theta\rvert \cdot \lVert \bm{\omega} - \bm{z} \rVert_{q^\star} \leq \lambda, \ \lvert\theta\rvert \cdot \lVert \bm{z} \rVert_{q^\star} \leq \widehat{\lambda} \}.
\end{array}
\end{align*}
Here, the first identity follows from writing the problem as a single maximization problem, the second identity follows from the equality constraint, the third identity follows from the definition of the conjugate $g^{*}$, the fourth identity is due to relabeling $\bm{z}_{1} = \theta \cdot \bm{z}_1$ and $\bm{z}_{2} = \theta \cdot \bm{z}_2$, and the fifth identity is due to a variable change ($\bm{z}_1 = \bm{\omega} - \bm{z}_2$ relabeled as $\bm{z}$). 
\end{proof}
DC maximization terms similar to the one dealt by Lemma~\ref{lemma:main} appear on the left-hand side of the constraints of~\ref{synth} (\emph{cf.}~formulation in Proposition~\ref{prop:reformulate}). 
These constraints would admit a tractable reformulation for the case without auxiliary data because the $\inf$-term in the reformulation presented in Lemma~\ref{lemma:main} does not appear in such cases. 
To see this, eliminate the second norm (the one associated with auxiliary data) by taking $\widehat{\lambda} = 0$, which will cause the constraint $\lvert\theta\rvert \cdot \lVert \bm{z} \rVert_{q^\star} \leq \widehat{\lambda}$ to force $\bm{z} = \mathbf{0}$, and the alternative formulation will thus be:
\begin{align*}
    & \begin{cases}
        \underset{\theta \in \mathrm{dom}(L^*)}{\sup} \{ - L^{*}(\theta) + \theta \cdot \bm{\omega}^\top \bm{a}\} & \text{if } \sup_{\theta \in \mathrm{dom}(L^*)} \{ \lvert \theta \rvert\} \cdot \lVert \bm{z} \rVert_{q^\star} \leq \lambda \\
        +\infty & \text{otherwise}
    \end{cases} \\ 
    \; = \; 
    & \begin{cases}
        L(\bm{\omega}^\top \bm{a}) & \text{if } \mathrm{Lip}(L) \cdot \lVert \bm{z} \rVert_{q^\star} \leq \lambda \\
        +\infty & \text{otherwise},
    \end{cases}
\end{align*}
where we used the fact that $L = L^{**}$ and $\sup_{\theta \in \mathrm{dom}(L)} \lvert \theta \rvert = \mathrm{Lip}(L)$ since $L$ is closed convex~\citep[Corollary 13.3.3]{R97:convex_analysis}. 
Hence, the DC maximization can be represented with a convex function with an additional convex inequality, making the constraints tractable for the case without auxiliary data. 
For the case with auxiliary data, however, the $\sup_{\theta} \inf_{\bm{z}}$ structure makes these constraints equivalent to two-stage robust constraints (with uncertain parameter $\theta$ and adjustable variable $\bm{z}$), bringing an adjustable robust optimization \citep{ben2004adjustable,yanikouglu2019survey} perspective to~\ref{synth}.
By using the univariate representation $\ell_{\bm{\beta}}^{\alpha}(\bm{x}, y) = L^\alpha(y \cdot \bm{\beta}^\top \bm{x})$,~\ref{synth} can be written as
\begin{align*}
    \begin{array}{cll}
        \displaystyle \underset{\bm{\beta}, \lambda, \widehat{\lambda}, \bm{s}, \widehat{\bm{s}}}{\mathrm{minimize}}  & \displaystyle \varepsilon \lambda + \widehat{\varepsilon}  \widehat{\lambda} + \dfrac{1}{N}\sum_{j=1}^N s_j + \dfrac{1}{\widehat{N}}\sum_{i=1}^{\widehat{N}} \widehat{s}_i & \\[5mm]
        \mathrm{subject\;to} & \displaystyle \underset{\bm{x} \in \mathbb{R}^n}{\sup} \{ L^\alpha(\bm{\beta}^\top\bm{x}) - \lambda  \lVert \bm{x}^i - \bm{x} \rVert_{q} - \widehat{\lambda}  \lVert \widehat{\bm{x}}^j - \bm{x} \rVert_{q}\} \leq \\[3mm]
        & 
        {\quad s_i + \kappa  \dfrac{1 - y^i}{2} \lambda + \widehat{s}_j + \kappa  \dfrac{1 - \widehat{y}^j}{2}  \widehat{\lambda}} & \forall i \in [N], \ \forall j \in [\widehat{N}]  \\[5mm]
        & \displaystyle \underset{\bm{x} \in \mathbb{R}^n}{\sup}\{L^\alpha(-\bm{\beta}^\top\bm{x}) - \lambda  \lVert \bm{x}^i - \bm{x} \rVert_{q} - \widehat{\lambda}   \lVert \widehat{\bm{x}}^j - \bm{x} \rVert_{q}\} \leq \\[3mm]
        & 
        {\quad s_i + \kappa  \dfrac{1 + y^i}{2} \lambda + \widehat{s}_j + \kappa  \dfrac{1 + \widehat{y}^j}{2}  \widehat{\lambda}} & \forall i \in [N], \ \forall j \in [\widehat{N}] \\[5mm]
        & \bm{\beta} \in \mathbb{R}^n, \ \lambda \geq 0, \ \widehat{\lambda}\geq 0, \ \bm{s} \in \mathbb{R}^{N}_{+}, \ \widehat{\bm{s}} \in \mathbb{R}^{\widehat{N}}_{+},
    \end{array}
\end{align*}
and applying Lemma~\ref{lemma:main} to the left-hand side of the constraints gives:
\begin{align}\label{pre-ARO}
    \raisetag{2em}
    \begin{array}{cll}
        \displaystyle \underset{\bm{\beta}, \lambda, \widehat{\lambda}, \bm{s}, \widehat{\bm{s}}}{\mathrm{minimize}}  & \displaystyle \varepsilon \lambda + \widehat{\varepsilon}  \widehat{\lambda} + \dfrac{1}{N}\sum_{j=1}^N s_j + \dfrac{1}{\widehat{N}}\sum_{i=1}^{\widehat{N}} \widehat{s}_i & \\[5mm]
        \mathrm{subject\;to} & \displaystyle \underset{\theta \in \mathrm{dom}(L^*)}{\sup} \ - L^{\alpha *}(\theta) + \theta \cdot \bm{\beta}^\top \bm{x}^i + \theta \cdot \underset{\bm{z} \in \mathbb{R}^n}{\inf} \{ \bm{z}^\top (\widehat{\bm{x}}^j - \bm{x}^i) \ : \ \lvert\theta\rvert \cdot \lVert \bm{\beta} - \bm{z} \rVert_{q^\star} \leq \lambda, \ \lvert\theta\rvert \cdot \lVert \bm{z} \rVert_{q^\star} \leq \widehat{\lambda} \} \leq \\[3mm]
        & 
        {\quad s_i + \kappa  \dfrac{1 - y^i}{2} \lambda + \widehat{s}_j + \kappa  \dfrac{1 - \widehat{y}^j}{2}  \widehat{\lambda} \quad \forall i \in [N], \ \forall j \in [\widehat{N}]} &  \\[5mm]
        & \displaystyle \underset{\theta \in \mathrm{dom}(L^*)}{\sup} \ - L^{\alpha *}(\theta) - \theta \cdot \bm{\beta}^\top \bm{x}^i + \theta \cdot \underset{\bm{z} \in \mathbb{R}^n}{\inf} \{ \bm{z}^\top (\widehat{\bm{x}}^j - \bm{x}^i) \ : \ \lvert\theta\rvert \cdot \lVert -\bm{\beta} - \bm{z} \rVert_{q^\star} \leq \lambda, \ \lvert\theta\rvert \cdot \lVert \bm{z} \rVert_{q^\star} \leq \widehat{\lambda} \} \leq \\[3mm]
        & 
        {\quad s_i + \kappa  \dfrac{1 + y^i}{2} \lambda + \widehat{s}_j + \kappa  \dfrac{1 + \widehat{y}^j}{2}  \widehat{\lambda} \quad \forall i \in [N], \ \forall j \in [\widehat{N}]} &  \\[5mm]
        & \bm{\beta} \in \mathbb{R}^n, \ \lambda \geq 0, \ \widehat{\lambda}\geq 0, \ \bm{s} \in \mathbb{R}^{N}_{+}, \ \widehat{\bm{s}} \in \mathbb{R}^{\widehat{N}}_{+}.
    \end{array}
\end{align}
Which, equivalently, can be written as the following problem with $2 N\cdot \widehat{N}$ two-stage robust constraints: 
\begin{align}\label{ARO}\tag{Inter-adjustable}
\hspace{-0.5cm}
    \raisetag{1.5em}
    \begin{array}{cll}
        \displaystyle \underset{\bm{\beta}, \lambda, \widehat{\lambda}, \bm{s}, \widehat{\bm{s}}}{\mathrm{minimize}}  & \displaystyle \varepsilon \lambda + \widehat{\varepsilon}  \widehat{\lambda} + \dfrac{1}{N}\sum_{j=1}^N s_j + \dfrac{1}{\widehat{N}}\sum_{i=1}^{\widehat{N}} \widehat{s}_i & \\[5mm]
        \mathrm{subject\;to} & \displaystyle 
        \left[
            \begin{aligned}
                \forall \theta \in \mathrm{dom}(L^*), \ \exists \bm{z} \in \mathbb{R}^n \ : \begin{cases} \ - L^{\alpha *}(\theta) + \theta \cdot \bm{\beta}^\top \bm{x}^i + \theta \cdot \bm{z}^\top (\widehat{\bm{x}}^j - \bm{x}^i) \leq s_i + \kappa  \dfrac{1 - y^i}{2} \lambda + \widehat{s}_j + \kappa  \dfrac{1 - \widehat{y}^j}{2}  \widehat{\lambda} \\  \lvert\theta\rvert \cdot \lVert \bm{\beta} - \bm{z} \rVert_{q^\star} \leq \lambda \\ \lvert\theta\rvert \cdot \lVert \bm{z} \rVert_{q^\star} \leq \widehat{\lambda}
                \end{cases}
            \end{aligned}
            \right]\\[8.0mm]
        & 
        \hfill { \forall i \in [N], \ \forall j \in [\widehat{N}]} &  \\[5mm]
        & \displaystyle 
        \left[
            \begin{aligned}
                \forall \theta \in \mathrm{dom}(L^*), \ \exists \bm{z} \in \mathbb{R}^n \ : \begin{cases} \ - L^{\alpha *}(\theta) - \theta \cdot \bm{\beta}^\top \bm{x}^i + \theta \cdot \bm{z}^\top (\widehat{\bm{x}}^j - \bm{x}^i) \leq s_i + \kappa  \dfrac{1 + y^i}{2} \lambda + \widehat{s}_j + \kappa  \dfrac{1 + \widehat{y}^j}{2}  \widehat{\lambda} \\  
                    \lvert\theta\rvert \cdot \lVert -\bm{\beta} - \bm{z} \rVert_{q^\star} \leq \lambda \\ 
                    \lvert\theta\rvert \cdot \lVert \bm{z} \rVert_{q^\star} \leq \widehat{\lambda}
                \end{cases}
            \end{aligned}
            \right]\\[8.0mm]
        & 
        \hfill { \forall i \in [N], \ \forall j \in [\widehat{N}]} &  \\[5mm]
        & \bm{\beta} \in \mathbb{R}^n, \ \lambda \geq 0, \ \widehat{\lambda}\geq 0, \ \bm{s} \in \mathbb{R}^{N}_{+}, \ \widehat{\bm{s}} \in \mathbb{R}^{\widehat{N}}_{+}.
    \end{array}
\end{align}

By using adjustable robust optimization theory, we show that this problem is NP-hard even in the simplest setting. 
To this end, take $N = \widehat{N} = 1$ as well as $\kappa = 0$; the formulation presented in Proposition~\ref{prop:reformulate} reduces to:
\begin{align*}
        \begin{array}{cl}
        \underset{\substack{\bm{\beta}, \lambda, \widehat{\lambda},s, \widehat{s}}}{\mathrm{minimize}}     & \varepsilon \lambda + \widehat{\varepsilon} \widehat{\lambda} + s + \widehat{s} \\
        \mathrm{subject\;to}  & \displaystyle \underset{\bm{x} \in \mathbb{R}^n}{\sup} \{\ell^{\alpha}_{\bm{\beta}}(\bm{x}, l) - \lambda \lVert \bm{x}^1 - \bm{x} \rVert_q  - \widehat{\lambda} \lVert \widehat{\bm{x}}^1 - \bm{x} \rVert_q\} \leq s_1 + \widehat{s}_1 \quad \forall l \in \{-1, 1\}\\
        & \bm{\beta} \in \mathbb{R}^n, \; \lambda \geq 0, \; \widehat{\lambda} \geq 0, \; s \geq 0, \; \widehat{s} \geq 0.
        \end{array}
\end{align*}
The worst case realization of $l \in \{-1,1\}$ will always make $\ell_{\bm{\beta}}^\alpha(\bm{x},l) = \log(1 + \exp(-l\cdot \bm{\beta}^\top\bm{x} + \alpha\cdot \lVert \bm{\beta}\rVert_{p^\star}))$ equal to $\varsigma_{\bm{\beta}}^\alpha(\bm{x}) =  \log(1 + \exp(\lvert l\cdot \bm{\beta}^\top\bm{x}\rvert + \alpha\cdot \lVert \bm{\beta}\rVert_{p^\star}))$, where $\varsigma$ inherits similar properties from $\ell$: it is convex in $\bm{\beta}$ and its univariate representation $S^\alpha$ has the same Lipschitz constant with $L^\alpha$. \mbox{We can thus represent the above problem as}
\begin{align*}
    \begin{array}{cl}
    \underset{\substack{\bm{\beta}, \lambda, \widehat{\lambda},s, \widehat{s}}}{\mathrm{minimize}}     & \varepsilon \lambda + \widehat{\varepsilon} \widehat{\lambda} + s + \widehat{s} \\
    \mathrm{subject\;to}  & \displaystyle \underset{\bm{x} \in \mathbb{R}^n}{\sup} \{S^\alpha(\bm{\beta}^\top \bm{x}) - \lambda \lVert \bm{x}^1 - \bm{x} \rVert_q  - \widehat{\lambda} \lVert \widehat{\bm{x}}^1 - \bm{x} \rVert_q\} \leq s + \widehat{s}\\
    & \bm{\beta} \in \mathbb{R}^n, \; \lambda \geq 0, \; \widehat{\lambda} \geq 0, \; s \geq 0, \; \widehat{s} \geq 0.
    \end{array}
\end{align*}

Substituting $s + \widehat{s}$ into the objective (due to the objective pressure) \mbox{allows us to reformulate the above problem as}
\begin{align}\label{NP-hard}
    \begin{array}{cl}
    \underset{\substack{\bm{\beta}, \lambda, \widehat{\lambda}}}{\mathrm{minimize}}     & \varepsilon \lambda + \widehat{\varepsilon} \widehat{\lambda} + \underset{\bm{x} \in \mathbb{R}^n}{\sup} \{S^\alpha(\bm{\beta}^\top \bm{x}) - \lambda \lVert \bm{x}^1 - \bm{x} \rVert_q  - \widehat{\lambda} \lVert \widehat{\bm{x}}^1 - \bm{x} \rVert_q\} \\
    \mathrm{subject\;to} & \bm{\beta} \in \mathbb{R}^n, \; \lambda \geq 0, \; \widehat{\lambda} \geq 0,
    \end{array}
\end{align}
and an application of Lemma~\ref{lemma:main} leads us to the following reformulation:
\begin{align*}
    \underset{\substack{\bm{\beta} \in \mathbb{R}^n \\ \lambda \geq 0, \widehat{\lambda} \geq 0}}{\inf} \  \underset{\theta \in \mathrm{dom}(S^*)}{\sup} \ \inf_{\bm{z} \in \mathbb{R}^n} \left\{\varepsilon \lambda + \widehat{\varepsilon} \widehat{\lambda} - S^{\alpha*}(\theta) + \theta \cdot \bm{\beta}^\top \bm{x}^1 + \underbrace{\theta \cdot \bm{z}^\top (\widehat{\bm{x}}^1 - \bm{x}^1)}_{\text{(1)}} \ : \ \underbrace{\lvert\theta\rvert \cdot \lVert \bm{\beta} - \bm{z} \rVert_{q^\star} \leq \lambda}_{\text{(2)}}, \ \lvert\theta\rvert \cdot \lVert \bm{z} \rVert_{q^\star} \leq \widehat{\lambda} \right\}. \\
\end{align*}
The objective term (1) has a product of the uncertain parameter $\theta$ and the adjustable variable $\bm{z}$, and even when (2) is linear such as in the case of $q = 1$ the product of the uncertain parameter with both the decision variable $\bm{\beta}$ and the adjustable variable $\bm{z}$ still appear since:
\begin{align*}
    \lvert\theta\rvert \cdot \lVert \bm{\beta} - \bm{z} \rVert_{\infty} \leq \lambda \iff -\lambda \leq \theta \bm{\beta} - \theta\bm{z} \leq \lambda.
\end{align*}
This reduces problem~\eqref{NP-hard} to a generic two-stage robust optimization problem with random recourse~\citep[Problem 1]{subramanyam2020k} which is proven to be NP-hard even if $S^{\alpha*}$ was constant~\citep{Guslitser2002}. \qed

\subsection{Proof of Theorem~\ref{thm:main}}\label{appendix_proof_thm:main}
Consider the reformulation~\ref{ARO} of~\ref{synth} that we introduced in the proof of Proposition~\ref{prop:np-hard}. For any $i \in [N]$ and $j \in [\widehat{N}]$, the corresponding constraint in the first group of `adjustable robust' ($\forall, \ \exists$) constraints will be:
\begin{align*}
\hspace{-0.6cm}
    \forall \theta \in \mathrm{dom}(L^*), \exists \bm{z} \in \mathbb{R}^n : 
    \begin{cases} - L^{\alpha *}(\theta) + \theta \cdot \bm{\beta}^\top \bm{x}^i + \theta \cdot \bm{z}^\top (\widehat{\bm{x}}^j - \bm{x}^i) \leq s_i + \kappa  \dfrac{1 - y^i}{2} \lambda + \widehat{s}_j + \kappa  \dfrac{1 - \widehat{y}^j}{2}  \widehat{\lambda} \\  \lvert\theta\rvert \cdot \lVert \bm{\beta} - \bm{z} \rVert_{q^\star} \leq \lambda \\ \lvert\theta\rvert \cdot \lVert \bm{z} \rVert_{q^\star} \leq \widehat{\lambda}.
    \end{cases}
\end{align*}
By changing the order of $\forall$ and $\exists$, we obtain:
\begin{align*}
\hspace{-0.6cm}
    \exists \bm{z} \in \mathbb{R}^n, \forall \theta \in \mathrm{dom}(L^*) : 
    \begin{cases} - L^{\alpha *}(\theta) + \theta \cdot \bm{\beta}^\top \bm{x}^i + \theta \cdot \bm{z}^\top (\widehat{\bm{x}}^j - \bm{x}^i) \leq s_i + \kappa  \dfrac{1 - y^i}{2} \lambda + \widehat{s}_j + \kappa  \dfrac{1 - \widehat{y}^j}{2}  \widehat{\lambda} \\  \lvert\theta\rvert \cdot \lVert \bm{\beta} - \bm{z} \rVert_{q^\star} \leq \lambda \\ \lvert\theta\rvert \cdot \lVert \bm{z} \rVert_{q^\star} \leq \widehat{\lambda}.
    \end{cases}
\end{align*}
Notice that this is a safe approximation, since any fixed $\bm{z}$ satisfying the latter system is a feasible static solution in the former system, meaning that for every realization of $\theta$ in the first system, the inner $\exists \bm{z}$ can always `play' the same $\bm{z}$ that is feasible in the latter system (hence the latter is named the \textit{static} relaxation,~\citealt{bertsimas2015tight}). 
In the relaxed system, we can drop $\forall \theta$ and keep its worst-case realization instead:
\begin{align*}
\hspace{-0.5cm}
    \exists \bm{z} \in \mathbb{R}^n : 
    \begin{cases} \sup_{\theta \in \mathrm{dom}(L^*)}\{- L^{\alpha *}(\theta) + \theta \cdot \bm{\beta}^\top \bm{x}^i + \theta \cdot \bm{z}^\top (\widehat{\bm{x}}^j - \bm{x}^i)\} \leq s_i + \kappa  \dfrac{1 - y^i}{2} \lambda + \widehat{s}_j + \kappa  \dfrac{1 - \widehat{y}^j}{2}  \widehat{\lambda} \\  
        \sup_{\theta \in \mathrm{dom}(L^*)}\{\lvert\theta\rvert\} \cdot \lVert \bm{\beta} - \bm{z} \rVert_{q^\star} \leq \lambda \\
        \sup_{\theta \in \mathrm{dom}(L^*)}\{\lvert\theta\rvert\} \cdot \lVert \bm{z} \rVert_{q^\star} \leq \widehat{\lambda}.
    \end{cases}
\end{align*}
The term $\sup_{\theta \in \mathrm{dom}(L^*)}\{- L^{\alpha *}(\theta) + \theta \cdot \bm{\beta}^\top \bm{x}^i + \theta \cdot \bm{z}^\top (\widehat{\bm{x}}^j - \bm{x}^i)\}$ is the definition of the biconjugate $L^{\alpha**}(\bm{\beta}^\top\bm{x}^i + \bm{z}^\top (\widehat{\bm{x}}^j - \bm{x}^i))$. Since $L^{\alpha}$ is a closed convex function, we have $L^{\alpha**} = L^{\alpha}$~\citep[Corollary 12.2.1]{R97:convex_analysis}. 
Moreover, $\sup_{\theta \in \mathrm{dom}(L^*)}\{\lvert\theta\rvert\}$ is an alternative representation of the Lipschitz constant of the function $L^\alpha$~\citep[Corollary 13.3.3]{R97:convex_analysis}, which is equal to $1$ as we showed earlier. The adjustable robust constraint thus reduces to:
\begin{align*}
    \exists \bm{z} \in \mathbb{R}^n : 
    \begin{cases} L^{\alpha}(\bm{\beta}^\top\bm{x}^i + \bm{z}^\top (\widehat{\bm{x}}^j - \bm{x}^i)) \leq s_i + \kappa  \dfrac{1 - y^i}{2} \lambda + \widehat{s}_j + \kappa  \dfrac{1 - \widehat{y}^j}{2}  \widehat{\lambda} \\  
        \lVert \bm{\beta} - \bm{z} \rVert_{q^\star} \leq \lambda \\
        \lVert \bm{z} \rVert_{q^\star} \leq \widehat{\lambda}
    \end{cases}
\end{align*}
as a result of the static relaxation. This relaxed reformulation applies to all $i \in [N]$ and $j\in[\widehat{N}]$ as well as to the second group of adjustable robust constraints analogously. Replacing each constraint of~\ref{ARO} with this system concludes the proof.\qed
\subsection{Proof of Corollary~\ref{corr:relax}}\label{appendix_proof_corr:relax}
To prove the first statement, take $\widehat{\lambda} = 0$ and observe the constraint $\lVert\bm{z}^{l}_{ij}\rVert_{q^\star} \leq \widehat{\lambda}$ implies $\bm{z}^{l}_{ij} = \mathbf{0}$ for all $l \in \{-1, 1\}, \ i \in [N], \ j \in [\widehat{N}]$. The optimization problem can thus be written without those variables:
\begin{align*}
    \begin{array}{cll}
    \displaystyle \underset{\substack{\bm{\beta}, \lambda, \bm{s}, \widehat{\bm{s}}}}{\mathrm{minimize}}     & \displaystyle\varepsilon \lambda + \frac{1}{N} \sum_{i=1}^N s_i + \frac{1}{\widehat{N}} \sum_{j=1}^{\widehat{N}} \widehat{s}_j \\[1.3em]
    \mathrm{subject\;to}  & L^{\alpha}(l\bm{\beta}^\top \bm{x}^i) \leq s_i + \kappa\dfrac{1 - ly^i}{2} \lambda + \widehat{s}_j & \forall l \in \{-1, 1\}, \  \forall i \in [N],\ \forall j \in [\widehat{N}]\\[0.5em]
    & \lVert \bm{\beta} \rVert_{q^\star} \leq \lambda \\
    & \bm{\beta} \in \mathbb{R}^n, \; \lambda \geq 0, \; \bm{s} \in \mathbb{R}^N_{+}, \; \widehat{\bm{s}} \in \mathbb{R}^{\widehat{N}}_{+}.
    \end{array}
\end{align*}
Notice that optimal solutions should satisfy $\widehat{s}_j  = \widehat{s}_{j'}$ for all $j, j' \in [N]$. To see this, assume for contradiction that $\exists j, j' \in [N]$ such that $\widehat{s}_j < \widehat{s}_{j'}$. If a constraint indexed with $(l, i, j)$ for arbitrary $l \in \{-1,1 \}$ and $i \in [N]$ is feasible, it means the consraint indexed with $(l, i,j')$ cannot be tight given that these constraints are identical except for the $\widehat{s}_j$ or $\widehat{s}_{j'}$ appearing on the right hand side. Hence, such a solution cannot be optimal as this is a minimization problem, and updating $\widehat{s}_{j'}$ as $\widehat{s}_j$ preserves the feasibility of the problem while decreasing the objective value. We can thus use a single variable $\tau \in \mathbb{R}_+$ and rewrite the problem as
\begin{align*}
    \begin{array}{cll}
    \displaystyle \underset{\substack{\bm{\beta}, \lambda, \bm{s}, \widehat{\bm{s}}}}{\mathrm{minimize}}     & \displaystyle\varepsilon \lambda + \frac{1}{N} \sum_{i=1}^N (s_i + \tau) \\[1.3em]
    \mathrm{subject\;to}  & L^{\alpha}(\bm{\beta}^\top \bm{x}^i) \leq s_i + \kappa\dfrac{1 - y^i}{2} \lambda + \tau & \forall i \in [N]\\[0.5em]
    & L^{\alpha}(-\bm{\beta}^\top \bm{x}^i) \leq s_i + \kappa\dfrac{1 + y^i}{2} \lambda + \tau & \forall i \in [N]\\[0.5em]
    & \lVert \bm{\beta} \rVert_{q^\star} \leq \lambda \\
    & \bm{\beta} \in \mathbb{R}^n, \; \lambda \geq 0, \; \bm{s} \in \mathbb{R}^N_{+}, \; \widehat{\bm{s}} \in \mathbb{R}^{\widehat{N}}_{+},
    \end{array}
\end{align*}
where we also eliminated the index $l \in \{-1, 1\}$ by writing the constraints explicitly. 
Since $s_i$ and $\tau$ both appear as $s_i + \tau$ in this problem, we can use a variable change where we relabel $s_i + \tau$ as $s_i$ (or, equivalently set $\tau = 0$ without any optimality loss).
Moreover, the constraints with index $i \in [N]$ are
\begin{align*}
    \begin{cases}
        L^{\alpha}(\bm{\beta}^\top \bm{x}^i) \leq s_i + \tau\\
        L^{\alpha}(-\bm{\beta}^\top \bm{x}^i) \leq s_i + \kappa \lambda + \tau
    \end{cases} \; = \; 
    \begin{cases}
        L^{\alpha}(y^i\cdot\bm{\beta}^\top \bm{x}^i) \leq s_i + \tau\\
        L^{\alpha}(-y^i\cdot\bm{\beta}^\top \bm{x}^i) \leq s_i + \kappa \lambda + \tau
    \end{cases}
\end{align*}
if $y^i = 1$, and similarly they are
\begin{align*}
    \begin{cases}
        L^{\alpha}(\bm{\beta}^\top \bm{x}^i) \leq s_i + \kappa \lambda + \tau\\
        L^{\alpha}(-\bm{\beta}^\top \bm{x}^i) \leq s_i +  \tau
    \end{cases} \; = \;
    \begin{cases}
        L^{\alpha}(-y^i \cdot \bm{\beta}^\top \bm{x}^i) \leq s_i + \kappa \lambda + \tau\\
        L^{\alpha}(y^i \cdot \bm{\beta}^\top \bm{x}^i) \leq s_i +  \tau
    \end{cases}
\end{align*}
if $y^i = -1$. Since these are identical, the problem can finally be written as
\begin{align*}
    \begin{array}{cll}
    \displaystyle \underset{\substack{\bm{\beta}, \lambda, \bm{s}}}{\mathrm{minimize}}     & \displaystyle\varepsilon \lambda + \frac{1}{N} \sum_{i=1}^N s_i \\[1.3em]
    \mathrm{subject\;to}  & \log(1 + \exp(-y^i \cdot \bm{\beta}^\top \bm{x}^i + \alpha \cdot \lVert \bm{\beta} \rVert_{p^\star} )) \leq s_i & \forall i \in [N]\\[0.5em] 
    & \log(1 + \exp(y^i \cdot \bm{\beta}^\top \bm{x}^i + \alpha \cdot \lVert \bm{\beta} \rVert_{p^\star} ))  - \lambda\kappa \leq s_i  & \forall i \in [N] \\[0.5em]
    & \lVert \bm{\beta} \rVert_{q^\star} \leq \lambda \\
    & \bm{\beta} \in \mathbb{R}^n, \; \lambda \geq 0, \; \bm{s} \in \mathbb{R}^N_{+},
    \end{array}
\end{align*}
where we also used the definition of $L^\alpha$. This problem is identical to~\ref{advdro}, which means that feasible solutions of~\ref{advdro} are feasible for~\ref{safe} if the additional variables $(\widehat{\lambda}, \widehat{\bm{s}}, \bm{z}^l_{ij})$ are set to zero, concluding the first statement of the corollary. 

The second statement is immediate since $\widehat{\varepsilon} \rightarrow \infty$ forces $\widehat\lambda = 0$ due to the term $\widehat{\varepsilon}\widehat{\lambda}$ in the objective of~\ref{safe}, and this proof shows in such a case~\ref{safe} reduces to~\ref{advdro} (which is identical to~\ref{synth} when $\varepsilon \rightarrow \infty$ by definition). \qed

\subsection{Proof of Observation~\ref{obs:literature}}\label{proof_mixture_representation}
By standard linearity arguments and from the definition of $\mathbb{Q}_{\mathrm{mix}}$, we have
\begin{align*}
    & \displaystyle \mathbb{E}_{\mathbb{Q}_{\mathrm{mix}}} \Bigg[\underset{\bm{z} \in \mathcal{B}_p(\alpha)}{\sup} \{\ell_{\bm{\beta}}(\bm{x} + \bm{z}, y)\} \Bigg] \\
\iff & \int_{(\bm{x}, y)\in \mathbb{R}^n \times \{-1, +1\}}  \  \underset{\bm{z} \in \mathcal{B}_p(\alpha)}{\sup} \{\ell_{\bm{\beta}}(\bm{x} + \bm{z}, y)\} \diff \mathbb{Q}_{\mathrm{mix}}((\bm{x}, y)) \\   
 \iff & \frac{N}{N + w\widehat{N}}  \int_{(\bm{x}, y)\in \mathbb{R}^n \times \{-1, +1\}}  \  \underset{\bm{z} \in \mathcal{B}_p(\alpha)}{\sup} \{\ell_{\bm{\beta}}(\bm{x} + \bm{z}, y)\} \diff \mathbb{P}_N((\bm{x}, y)) + \\
 & \qquad  \frac{w\widehat{N}}{N + w\widehat{N}}  \int_{(\bm{x}, y)\in \mathbb{R}^n \times \{-1, +1\}}  \  \underset{\bm{z} \in \mathcal{B}_p(\alpha)}{\sup} \{\ell_{\bm{\beta}}(\bm{x} + \bm{z}, y)\} \diff \widehat{\mathbb{P}}_{\widehat{N}}((\bm{x}, y))  \\
 \iff & \frac{N}{N + w\widehat{N}}  \cdot  \displaystyle \dfrac{1}{N} \sum_{i \in [N]}  \underset{\bm{z}^i \in \mathcal{B}_p(\alpha)}{\sup} \{\ell_{\bm{\beta}}(\bm{x}^i + \bm{z}^i, y^i)\}  +
 \frac{w\widehat{N}}{N + w\widehat{N}} \cdot \displaystyle \dfrac{1}{\widehat{N}} \sum_{j \in [\widehat{N}]}  \underset{\bm{z}^j \in \mathcal{B}_p(\alpha)}{\sup} \{\ell_{\bm{\beta}}(\widehat{\bm{x}}^j  + \bm{z}^j, \widehat{y}^j)\} \\
 \iff & \displaystyle  \dfrac{1}{N + w\widehat{N}} \left[ \sum_{i \in [N]}  \underset{\bm{z}^i \in \mathcal{B}_p(\alpha)}{\sup} \{\ell_{\bm{\beta}}(\bm{x}^i + \bm{z}^i, y^i)\}  +  \displaystyle w \cdot \sum_{j \in [\widehat N]}   \underset{\bm{z}^j \in \mathcal{B}_p(\alpha)}{\sup} \{\ell_{\bm{\beta}}(\widehat{\bm{x}}^j  + \bm{z}^j, \widehat{y}^j)\}   \right],
\end{align*}
which coincides with the objective function of~\eqref{synth_literature}. Since we have
$$\displaystyle \mathbb{E}_{\mathbb{Q}_{\mathrm{mix}}} \Bigg[\underset{\bm{z} \in \mathcal{B}_p(\alpha)}{\sup} \{\ell_{\bm{\beta}}(\bm{x} + \bm{z}, y)\} \Bigg] = \displaystyle \mathbb{E}_{\mathbb{Q}_{\mathrm{mix}}} [\ell^\alpha_{\bm{\beta}}(\bm{x}, y)]$$
we can conclude the proof. \qed

\subsection{Proof of Proposition~\ref{prop:included}}\label{proof_prop_intersection}
We first prove auxiliary results on mixture distributions. To this end, denote by $\mathcal{C}(\mathbb{Q}, \mathbb{P}) \subseteq \mathcal{P}(\Xi \times \Xi)$ the set of couplings of the distributions $\mathbb{Q} \in \mathcal{P}(\Xi)$ and $\mathbb{P} \in \mathcal{P}(\Xi)$. 
\begin{lemma}\label{lem_mixture}
Let $\mathbb{Q}, \mathbb{P}^1, \mathbb{P}^2 \in \mathcal{P}(\Xi)$ be probability distributions. If $\Pi^1 \in \mathcal{C}(\mathbb{Q}, \mathbb{P}^1)$ and $\Pi^2 \in \mathcal{C}(\mathbb{Q}, \mathbb{P}^2)$, then, $\lambda \cdot \Pi^1 + (1 - \lambda) \cdot \Pi^2 \in \mathcal{C}(\mathbb{Q}, \lambda \cdot \mathbb{P}^1 + (1-\lambda) \cdot \mathbb{P}^2)$ for all $\lambda \in (0,1)$.
\end{lemma}
\begin{proof}
Let $\Pi = \lambda \cdot \Pi^1 + (1 - \lambda) \cdot \Pi^2$ and $\mathbb{P} = \lambda \cdot \mathbb{P}^1 + (1-\lambda) \cdot \mathbb{P}^2$. To have $\Pi \in \mathcal{C}(\mathbb{Q}, \mathbb{P})$ we need $\Pi(\diff \bm{\xi}, \Xi) = \mathbb{Q}(\diff \bm{\xi})$ and $\Pi(\Xi, \diff \bm{\xi'}) = \mathbb{P}(\diff \bm{\xi'})$. To this end, observe that
\begin{align*}
    \Pi(\diff \bm{\xi}, \Xi) &= \lambda \cdot \Pi^1(\diff \bm{\xi}, \Xi) + (1 - \lambda) \cdot \Pi^2(\diff \bm{\xi}, \Xi)\\
    & = \lambda \cdot \mathbb{Q} + (1- \lambda)\cdot \mathbb{Q} \; = \;  \mathbb{Q}
\end{align*}
where the second identity uses the fact that $\Pi^1 \in \mathcal{C}(\mathbb{Q}, \mathbb{P}^1)$. Similarly, we can show:
\begin{align*}
    \Pi(\Xi, \diff \bm{\xi}) &= \lambda \cdot \Pi^1(\Xi, \diff \bm{\xi}) + (1 - \lambda) \cdot \Pi^2(\Xi, \diff \bm{\xi})\\
    & = \lambda \cdot \mathbb{P}^1 + (1- \lambda)\cdot \mathbb{P}^2 \; = \; \mathbb{P},
\end{align*}
which concludes the proof.
\end{proof}
We further prove the following intermediary result.
\begin{lemma}\label{lem_coupling}
    Let $\mathbb{Q}, \mathbb{P}^1, \mathbb{P}^2 \in \mathcal{P}(\Xi)$ and $\mathbb{P} = \lambda \cdot \mathbb{P}^1 + (1- \lambda) \cdot \mathbb{P}^2 $ for some $\lambda \in (0,1)$. We have:
    \begin{align*}
        \mathrm{W}(\mathbb{Q}, \mathbb{P}) \leq  \lambda \cdot \mathrm{W}(\mathbb{Q}, \mathbb{P}^1) + (1- \lambda) \cdot \mathrm{W}(\mathbb{Q}, \mathbb{P}^2).
    \end{align*}
\end{lemma}
\begin{proof}
    The Wasserstein distance between $\mathbb{Q}, \mathbb{Q}' \in \mathcal{P}(\Xi)$ can be written as:
    \begin{align*}
        \mathrm{W}(\mathbb{Q}, \mathbb{Q}') = \underset{\Pi \in \mathcal{C}(\mathbb{Q}, \mathbb{Q}')}{\min} \left\{ \int_{\Xi \times \Xi} d(\bm{\xi}, \bm{\xi'}) \Pi(\diff \bm{\xi}, \diff \bm{\xi'}) \right\},
    \end{align*}
    and since $d$ is a feature-label metric (\textit{cf.} Definition~\ref{def:feature-label}) the minimum is well-defined~\citep[Theorem 4.1]{villani2009optimal}. We name the optimal solutions to the above problem the \textit{optimal couplings}. Let $\Pi^1$ be an optimal coupling of $\mathrm{W}(\mathbb{Q}, \mathbb{P}^1)$ and let $\Pi^2$ be an optimal coupling of $\mathrm{W}(\mathbb{Q}, \mathbb{P}^2)$ and define $\Pi^c = \lambda \cdot \Pi^1 + (1 - \lambda) \cdot \Pi^2$. We have
    \begin{align*}
        \mathrm{W}(\mathbb{Q}, \mathbb{P}) &= \underset{\Pi \in \mathcal{C}(\mathbb{Q},  \mathbb{P})}{\min} \left\{ \int_{\Xi \times \Xi} d(\bm{\xi}, \bm{\xi'}) \Pi(\diff \bm{\xi}, \diff \bm{\xi'}) \right\} \\
        & \leq  \int_{\Xi \times \Xi} d(\bm{\xi}, \bm{\xi'}) \Pi^c(\diff \bm{\xi}, \diff \bm{\xi'})\\
        & = \lambda \cdot \int_{\Xi \times \Xi} d(\bm{\xi}, \bm{\xi'}) \Pi^1(\diff \bm{\xi}, \diff \bm{\xi'}) + (1-\lambda) \cdot \int_{\Xi \times \Xi} d(\bm{\xi}, \bm{\xi'}) \Pi^2(\diff \bm{\xi}, \diff \bm{\xi'})  \\
        & = \lambda \cdot \mathrm{W}(\mathbb{Q}, \mathbb{P}^1) + (1- \lambda) \cdot \mathrm{W}(\mathbb{Q}, \mathbb{P}^2),
    \end{align*}
    where the first identity uses the definition of the Wasserstein metric, the inequality is due to Lemma~\ref{lem_mixture} as $\Pi^c$ is a feasible coupling (not necessarily optimal), the equality that follows uses the definition of $\Pi^c$ and the linearity of integrals, and the final identity uses the fact that $\Pi^1$ and $\Pi^2$ were constructed to be the optimal couplings.
\end{proof}
We now prove the proposition (we refer to $\mathbb{Q}_{\mathrm{mix}}$ in the statement of this lemma simply as $\mathbb{Q}$). To prove $\mathbb{Q} \in \mathfrak{B}_{\varepsilon}({\mathbb{P}}_N) \cap \mathfrak{B}_{\widehat \varepsilon}({\widehat{\mathbb{P}}}_{\widehat{N}})$, it is sufficient to show that $\mathrm{W}(\mathbb{P}_N, \mathbb{Q}) \leq \varepsilon$ and $\mathrm{W}(\widehat{\mathbb{P}}_{\widehat{N}}, \mathbb{Q}) \leq \widehat{\varepsilon}$ jointly hold. By using Lemma~\ref{lem_coupling}, we can derive the following inequalities:
\begin{align*}
    & \mathrm{W}(\mathbb{P}_N, \mathbb{Q})  \leq \lambda \cdot \underbrace{\mathrm{W}(\mathbb{P}_N, \mathbb{P}_N)}_{=0} + (1- \lambda) \cdot  \mathrm{W}(\mathbb{P}_N, \widehat{\mathbb{P}}_{\widehat{N}}) \\
    & \mathrm{W}(\widehat{\mathbb{P}}_{\widehat{N}}, \mathbb{Q})  \leq \lambda \cdot \mathrm{W}(\mathbb{P}_N, \widehat{\mathbb{P}}_{\widehat{N}}) + (1-\lambda) \cdot \underbrace{\mathrm{W}(\widehat{\mathbb{P}}_{\widehat{N}},\widehat{\mathbb{P}}_{\widehat{N}})}_{=0}.
\end{align*}
Therefore, sufficient conditions on $\mathrm{W}(\mathbb{P}_N, \mathbb{Q}) \leq \varepsilon$ and $\mathrm{W}(\widehat{\mathbb{P}}_{\widehat{N}}, \mathbb{Q}) \leq \widehat{\varepsilon}$ would be:
\begin{align*}
\begin{cases}
    (1- \lambda) \cdot  \mathrm{W}(\mathbb{P}_N, \widehat{\mathbb{P}}_{\widehat{N}}) \leq \varepsilon \\
    \lambda \cdot \mathrm{W}(\mathbb{P}_N, \widehat{\mathbb{P}}_{\widehat{N}}) \leq \widehat{\varepsilon}.
\end{cases}
\end{align*}
Moreover, given that $\varepsilon + \widehat{\varepsilon} \geq \mathrm{W}(\mathbb{P}_N, \widehat{\mathbb{P}}_{\widehat{N}})$, the sufficient conditions further simplify to
\begin{align*}
\begin{cases}
    (1- \lambda) \cdot  \widehat{\varepsilon} \leq \lambda \cdot \varepsilon \\
    \lambda \cdot \varepsilon \leq (1- \lambda)\cdot \widehat{\varepsilon}.
\end{cases} \iff \lambda \cdot \varepsilon = (1- \lambda)\cdot \widehat{\varepsilon},
\end{align*}
which is implied when $\dfrac{\lambda}{1 - \lambda} = \dfrac{\widehat{\varepsilon}}{\varepsilon}$, concluding the proof. \qed

\subsection{Proof of Theorem~\ref{thm:stats1}}\label{Sec:thmstats1proof}
Since each result in the statement of this theorem is abridged, we will present these results sequentially as separate results.
We review the existing literature to characterize $\mathfrak{B}_{\varepsilon}({\mathbb{P}}_{N})$, in a similar fashion with the results presented in~\citep[Appendix A]{selvi2022wasserstein} for the logistic loss, by revising them to the adversarial loss whenever necessary. The $N$-fold product distribution of $\mathbb{P}^0$ from which the training set $\mathbb{P}_N$ is constructed is denoted below by $[\mathbb{P}^0]^{N}$. 
\begin{theorem}\label{thm_includes}
    Assume there exist $a > 1$ and $A > 0$ such that $\mathbb{E}_{\mathbb{P}^0}[\exp(\lVert \bm{\xi} \rVert^a)] \leq A$ for a norm $\lVert \cdot \rVert$ on $\mathbb{R}^n$. Then, there are constants $c_1, c_2 > 0$ that only depend on $\mathbb{P}^0$ through $a$, $A$, and $n$, such that $[\mathbb{P}^0]^{N}( \mathbb{P}^0 \in \mathfrak{B}_{\varepsilon}({\mathbb{P}}_{N})) \geq 1 - \eta$ holds for any confidence level $\eta \in (0,1)$ as long as the Wasserstein ball radius satisfies the following optimal characterization
    \begin{align*}
        \varepsilon \geq 
        \begin{cases}
            \left(\dfrac{\log(c_1 / \eta)}{c_2 \cdot N} \right)^{1 / \max\{n, 2\}} & \text{if } N \geq \dfrac{\log(c_1 / \eta)}{c_2} \\
            \left(\dfrac{\log(c_1 / \eta)}{c_2 \cdot N} \right)^{1 / a} &\text{otherwise.}
        \end{cases}
    \end{align*}
\end{theorem}
\begin{proof}
    The statement follows from Theorem 18 of~\cite{kuhn2019wasserstein}. The presented decay rate $\mathcal{O}(N^{-1/n})$ of $\varepsilon$ as $N$ increases is optimal~\citep{fournier2015rate}. 
\end{proof}
Now that we gave a confidence for the unknown radius $\varepsilon$ satisfying $\mathbb{P}^0 \in \mathfrak{B}_{\varepsilon}({\mathbb{P}}_{N})$, we analyze the underlying optimization problems. 
Most of the theory is well-established for logistic loss function, and in the following we show that similar results follow for the adversarial loss function. For convenience, we state~\ref{advdro} again by using the adversarial loss function as defined in Observation~\ref{obs:prelim}:
\begin{align}\label{advdro_alternative}\tag{DR-ARO}
    \begin{array}{cl}
        \displaystyle \underset{\bm{\beta}}{\mathrm{minimize}} & \displaystyle \underset{\mathbb{Q} \in \mathfrak{B}_{\varepsilon}({\mathbb{P}}_N)}{\sup} \quad  \mathbb{E}_{\mathbb{Q}} [\ell^{\alpha}_{\bm{\beta}}(\bm{x}, y)] \\
        \mathrm{subject\;to} & \bm{\beta} \in \mathbb{R}^n.
    \end{array}
\end{align}
\begin{theorem}\label{theorem_over}
    If the assumptions of Theorem~\ref{thm_includes} are satisfied and $\varepsilon$ is chosen as in the statement of Theorem~\ref{thm_includes}, then
    \begin{align*}
        [\mathbb{P}^0]^N \left( \mathbb{E}_{\mathbb{P}^0} [\ell^{\alpha}_{\bm{\beta}^\star}(\bm{x}, y)] \leq \underset{\mathbb{Q} \in \mathfrak{B}_{\varepsilon}({\mathbb{P}}_N)}{\sup} \mathbb{E}_{\mathbb{Q}} [\ell^{\alpha}_{\bm{\beta}^\star}(\bm{x}, y)] \right) \geq 1- \eta
    \end{align*}
    holds for all $\eta \in (0,1)$ and all optimizers $\bm{\beta}^\star$ of~\ref{advdro_alternative}.
\end{theorem}
\begin{proof}
    The statement follows from Theorem 19 of~\cite{kuhn2019wasserstein} given that $\ell^{\alpha}_{\bm{\beta}}$ is a finite-valued continuous loss function. 
\end{proof}
Theorem~\ref{theorem_over} states that the optimal value of~\ref{advdro_alternative} overestimates the true loss with arbitrarily high confidence $1-\eta$. Despite the desired overestimation of the true loss, we show that~\ref{advdro_alternative} is still asymptotically consistent if we restrict the set of admissible $\bm{\beta}$ to a bounded set\footnote{Note that, this is without loss of generality given that we can normalize the decision boundary of linear classifiers.}. 
\begin{theorem}
    If we restrict the hypotheses $\bm{\beta}$ to a bounded set $\mathcal{H}\subseteq \mathbb{R}^{n}$, and parameterize $\varepsilon$ as $\varepsilon_N$ to show its dependency to the sample size, then, under the assumptions of Theorem~\ref{thm_includes}, we have
    \begin{align*}
        \underset{\mathbb{Q} \in \mathfrak{B}_{\varepsilon_N}({\mathbb{P}}_N)}{\sup} \mathbb{E}_{\mathbb{Q}}[\ell^{\alpha}_{\bm{\beta}^\star}(\bm{x},y)] \underset{N \rightarrow \infty}{\longrightarrow} \mathbb{E}_{\mathbb{P}^0}[\ell^{\alpha}_{\bm\beta{^\star}}(\bm{x},y)] \quad \mathbb{P}^0\text{-almost surely},
   \end{align*}
   whenever $\varepsilon_N$ is set as specified in Theorem~\ref{thm_includes} along with its finite-sample confidence $\eta_N$, and they satisfy $\sum_{N \in \mathbb{N}} \eta_N < \infty$ and $\lim_{N \rightarrow \infty} \varepsilon_N = 0$.
\end{theorem}
\begin{proof}
    If we show that there exists $\bm{\xi}^0 \in \Xi$ and $C > 0$ such that $\ell^{\alpha}_{\bm{\beta}}(\bm{x}, y) \leq C(1 + d(\bm{\xi}, \bm{\xi}^0))$ holds for all $\bm{\beta} \in \mathcal{H}$ and $\bm{\xi} \in \Xi$ (that is, the adversarial loss satisfies a growth condition), the statement will follow immediately from Theorem 20 of~\citep{kuhn2019wasserstein}. 

    To see that the growth condition is satisfied, we first substitute the definition of $\ell^{\alpha}_{\bm{\beta}}$ and $d$ explicitly, and note that we would like to show there exists $\bm{\xi}^0 \in \Xi$ and $C > 0$ such that
    \begin{align*}
        \log(1 + \exp(-y\cdot\bm{\beta}^\top \bm{x} + \alpha \cdot \lVert \bm{\beta}\rVert_{p^\star})) \leq C(1 + \lVert \bm{x} - \bm{x}^0 \rVert_{q} + \kappa \cdot \indicate{y \neq y^0})
    \end{align*}
    holds for all $\bm{\beta} \in \mathcal{H}$ and $\bm{\xi} \in \Xi$. We take $\bm{\xi}^0 = (\mathbf{0}, y^0)$ and show that the right-hand side of the inequality can be lower bounded as:
    \begin{align*}
    C(1 + \lVert \bm{x} - \bm{x}^0 \rVert_{q} + \kappa \cdot \indicate{y \neq y^0}) & = 
        C(1 + \lVert \bm{x} \rVert_{q} + \kappa \cdot \indicate{y \neq y^0})\\
        & \geq C(1 + \lVert \bm{x} \rVert_{q}).
    \end{align*}
    Moreover, the left-hand side of the inequality can be upper bounded for any $\bm{\beta} \in \mathcal{H} \subseteq [-M, M]^{n}$ (for some $M > 0$) and $\bm{\xi} = (\bm{x}, y) \in \Xi$ as:
    \begin{align*}
    \log(1 + \exp(-y\cdot\bm{\beta}^\top \bm{x} + \alpha \cdot \lVert \bm{\beta}\rVert_{p^\star})) & \leq \log(1+ \exp(\lvert \bm{\beta}^\top \bm{x}\rvert + \alpha \cdot \lVert \bm{\beta} \rVert_{p^\star}))\\ 
    & \leq \log(2 \cdot \exp(\lvert \bm{\beta}^\top \bm{x}\rvert + \alpha \cdot \lVert \bm{\beta} \rVert_{p^\star}))\\
    & = \log(2) + \lvert \bm{\beta}^\top \bm{x}\rvert + \alpha \cdot \lVert \bm{\beta} \rVert_{p^\star} \\
    & \leq \log(2) + \sup_{\bm{\beta} \in [-M, M]^n}\{\lvert \bm{\beta}^\top \bm{x}\rvert\} + \alpha \cdot \sup_{\bm{\beta} \in [-M, M]^n}\{ \lVert \bm{\beta} \rVert_{p^\star} \} \\
    & = \log(2) + M \cdot \lVert \bm{x} \rVert_{1} + M \cdot \alpha \\
    & \leq\log(2) +  M \cdot n^{(q-1)/q} \cdot \lVert \bm{x} \rVert_{1} + M \cdot \alpha 
    \end{align*} 
    where the final inequality uses Hölder's inequality to bound the $1$-norm with the $q$-norm. Thus, it suffices to show that we have
    \begin{align*}
        \log(2) +  M \cdot n^{(q - 1)/q} \cdot \lVert \bm{x} \rVert_{1} + M \cdot \alpha \leq C(1 + \lVert \bm{x} \rVert_{q})  \quad \forall \bm{\xi} \in \Xi,
    \end{align*}
    which is satisfied for any $C \geq \max \{ \log(2) + M \cdot \alpha, \ M\cdot n^{(q-1)/q}  \}$. This completes the proof by showing the growth condition is satisfied.
\end{proof}

So far, we reviewed tight characterizations for $\varepsilon$ so that the ball $\mathfrak{B}_{\varepsilon}({{\mathbb{P}}}_{{N}})$ includes the true distribution $\mathbb{P}^0$ with arbitrarily high confidence, proved that the DRO problem~\ref{advdro_alternative} overestimates the true loss, while converging to the true problem asymptotically as the confidence $1-\eta$ increases and the radius $\varepsilon$ decreases simultaneously. Finally, we discuss that for optimal solutions $\bm{\beta}^\star$ to~\ref{advdro_alternative}, there are worst case distributions $\mathbb{Q}^\star \in \mathfrak{B}_{\varepsilon}({{\mathbb{P}}}_{{N}})$ of nature's problem that are supported on at most $N+1$ outcomes.
\begin{theorem}
    If we restrict the hypotheses $\bm{\beta}$ to a bounded set $\mathcal{H}\subseteq \mathbb{R}^{n}$, then there are distributions $\mathbb{Q}^\star \in \mathfrak{B}_{\varepsilon}({{\mathbb{P}}}_{{N}})$ that are supported on at most $N+1$ outcomes and satisfy:
    \begin{align*}
        \mathbb{E}_{\mathbb{Q}^\star} [\ell^{\alpha}_{\bm{\beta}}(\bm{x}, y)] \; = \; \underset{\mathbb{Q} \in \mathfrak{B}_{\varepsilon}({\mathbb{P}}_N)}{\sup} \mathbb{E}_{\mathbb{Q}} [\ell^{\alpha}_{\bm{\beta}}(\bm{x}, y)].
    \end{align*}
\end{theorem}
\begin{proof}
    The proof follows from~\citep{YKW21:linear_optimization_wasserstein}.
\end{proof}
See the proof of~\citet[Theorem 8]{selvi2022wasserstein} and the discussion that follows for insights and further analysis on these results presented.

\subsection{Proof of Theorem~\ref{thm:stats2}}\label{app:stats2}
Firstly, since $\widehat{\mathbb{P}}_{\widehat{N}}$ is constructed from i.i.d.\@ samples of $\widehat{\mathbb{P}}$, we can overestimate the distance $\widehat{\varepsilon}_1= \mathrm{W}(\widehat{\mathbb{P}}_{\widehat{N}}, \widehat{\mathbb{P}})$ analogously by applying Theorem~\ref{thm_includes}, \textit{mutatis mutandis}. This leads us to the following result where the joint (independent) $N$-fold product distribution of $\mathbb{P}^0$ and the $\widehat{N}$-fold product distribution of $\widehat{\mathbb{P}}$ is denoted below by $[\mathbb{P}^0 \times \widehat{\mathbb{P}}]^{N \times \widehat{N}}$.
\begin{theorem}\label{thm:final_act}
    Assume that there exist $a > 1$ and $A > 0$ such that $\mathbb{E}_{\mathbb{P}^0}[\exp(\lVert \bm{\xi} \rVert^a)] \leq A$, and there exist $\widehat{a} > 1$ and $\widehat{A} > 0$ such that $\mathbb{E}_{\widehat{\mathbb{P}}}[\exp(\lVert \bm{\xi} \rVert^{\widehat{a}})] \leq \widehat{A}$ for a norm $\lVert \cdot \rVert$ on $\mathbb{R}^n$. Then, there are constants $c_1, c_2 > 0$ that only depends on $\mathbb{P}^0$ through $a$, $A$, and $n$, and constants $\widehat{c}_1, \widehat{c}_2 > 0$ that only depends on $\widehat{\mathbb{P}}$ through $\widehat{a}$, $\widehat{A}$, and $n$ such that $[\mathbb{P}^0 \times \widehat{\mathbb{P}}]^{N \times \widehat{N}}( \mathbb{P}^0 \in \mathfrak{B}_{\varepsilon}(\mathbb{P}_N) \cap  \mathfrak{B}_{\widehat{\varepsilon}}(\widehat{\mathbb{P}}_{\widehat{N}})) \geq 1 - \eta$ holds for any confidence level $\eta \in (0,1)$ as long as the Wasserstein ball radii satisfy the following characterization
    \begin{align*}
        & \varepsilon \geq 
        \begin{cases}
            \left(\dfrac{\log(c_1 / \eta_1)}{c_2 \cdot N} \right)^{1 / \max\{n, 2\}} & \text{if } N \geq \dfrac{\log(c_1 / \eta_1)}{c_2} \\
            \left(\dfrac{\log(c_1 / \eta_1)}{c_2 \cdot N} \right)^{1 / a} &\text{otherwise}
        \end{cases} \\
        & \widehat{\varepsilon} \geq \mathrm{W}(\mathbb{P}^0, \widehat{\mathbb{P}}) + 
        \begin{cases}
            \left(\dfrac{\log(\widehat{c}_1 / \eta_2)}{\widehat{c}_2 \cdot \widehat{N}} \right)^{1 / \max\{n, 2\}} & \text{if } \widehat{N} \geq \dfrac{\log(\widehat{c}_1 / \eta_2)}{\widehat{c}_2} \\
            \left(\dfrac{\log(\widehat{c}_1 / \eta_2)}{\widehat{c}_2 \cdot \widehat{N}} \right)^{1 / \widehat{a}} &\text{otherwise}
        \end{cases}
    \end{align*}
    for some $\eta_1, \eta_2 > 0$ satisfying $\eta_1 + \eta_2 = \eta$.
\end{theorem}
\begin{proof}
    It immediately follows from Theorem~\ref{thm_includes} that $[\mathbb{P}^0]^{N}( \mathbb{P}^0 \in \mathfrak{B}_{\varepsilon}({\mathbb{P}}_{N})) \geq 1 - \eta_1$ holds. If we take $\widehat{\varepsilon}_1 > 0$ as
    \begin{align*}
        \widehat{\varepsilon}_1 \geq 
        \begin{cases}
            \left(\dfrac{\log(\widehat{c}_1 / \eta_2)}{\widehat{c}_2 \cdot \widehat{N}} \right)^{1 / \max\{n, 2\}} & \text{if } \widehat{N} \geq \dfrac{\log(\widehat{c}_1 / \eta_2)}{\widehat{c}_2} \\
            \left(\dfrac{\log(\widehat{c}_1 / \eta_2)}{\widehat{c}_2 \cdot \widehat{N}} \right)^{1 / \widehat{a}} &\text{otherwise},
        \end{cases}
    \end{align*}
    then, we similarly have $[\widehat{\mathbb{P}}]^{\widehat{N}}( \widehat{\mathbb{P}} \in \mathfrak{B}_{\widehat{\varepsilon}_1}({\widehat{\mathbb{P}}}_{\widehat{N}})) \geq 1 - \eta_2$. Since the following implication follows from the triangle inequality:
    \begin{align*}
        \widehat{\mathbb{P}} \in \mathfrak{B}_{\widehat{\varepsilon}_1}({\widehat{\mathbb{P}}}_{\widehat{N}}) \implies \mathbb{P}^0 \in \mathfrak{B}_{\widehat{\varepsilon}_1 + \mathrm{W}(\mathbb{P}^0, \widehat{\mathbb{P}})}({\widehat{\mathbb{P}}}_{\widehat{N}}),
    \end{align*}
    we have that $[\widehat{\mathbb{P}}]^{\widehat{N}}( \mathbb{P}^0 \in \mathfrak{B}_{\varepsilon}({\widehat{\mathbb{P}}}_{\widehat{N}})) \geq 1 - \eta_2$. These results, along with the facts that $\widehat{\mathbb{P}}_{\widehat{N}}$ and $\mathbb{P}_N$ are independently sampled from their true distributions, imply:
    \begin{align*}
        & [\mathbb{P}^0 \times \widehat{\mathbb{P}}]^{N \times \widehat{N}}( \mathbb{P}^0 \not\in \mathfrak{B}_{\varepsilon}(\mathbb{P}_N) \lor \mathbb{P}^0 \not\in\mathfrak{B}_{\widehat{\varepsilon}}(\widehat{\mathbb{P}}_{\widehat{N}})) \\
       \leq &  [\mathbb{P}^0 \times \widehat{\mathbb{P}}]^{N \times \widehat{N}}( \mathbb{P}^0 \not\in \mathfrak{B}_{\varepsilon}(\mathbb{P}_N)) +  [\mathbb{P}^0 \times \widehat{\mathbb{P}}]^{N \times \widehat{N}}(\mathbb{P}^0 \not\in\mathfrak{B}_{\widehat{\varepsilon}}(\widehat{\mathbb{P}}_{\widehat{N}}))\\
       = & [\mathbb{P}^0]^{N}( \mathbb{P}^0 \not\in \mathfrak{B}_{\varepsilon}(\mathbb{P}_N)) +  [\widehat{\mathbb{P}}]^{\widehat{N}}(\mathbb{P}^0 \not\in\mathfrak{B}_{\widehat{\varepsilon}}(\widehat{\mathbb{P}}_{\widehat{N}})) < \eta_1 + \eta_2
    \end{align*}
    implying the desired result $[\mathbb{P}^0 \times \widehat{\mathbb{P}}]^{N \times \widehat{N}}( \mathbb{P}^0 \in \mathfrak{B}_{\varepsilon}(\mathbb{P}_N) \cap  \mathfrak{B}_{\widehat{\varepsilon}}(\widehat{\mathbb{P}}_{\widehat{N}})) \geq 1 - \eta$.
\end{proof}
The second statement immediately follows under the assumptions of Theorem~\ref{thm:final_act}:~\ref{synth} overestimates the true loss analogously as Theorem~\ref{theorem_over} with an identical proof.

\section{EXPONENTIAL CONIC REFORMULATION OF~\ref{advdro}}
For any $i \in [N]$, the constraints of~\ref{advdro} are
\begin{align*}
\begin{cases}
    \log(1 + \exp(  - y^i\cdot \bm{\beta}^\top \bm{x}^i +  \alpha \cdot \lVert \bm{\beta} \rVert_{p^\star} )) \leq s_i \\
    \log(1 + \exp( y^i\cdot \bm{\beta}^\top \bm{x}^i +  \alpha \cdot \lVert \bm{\beta} \rVert_{p^\star} )) - \lambda \cdot \kappa \leq s_i,
\end{cases}
\end{align*}
which, by using an auxiliary variable $u$, can be written as
\begin{align*}
\begin{cases}
    \log(1 + \exp(  - y^i\cdot \bm{\beta}^\top \bm{x}^i +  u)) \leq s_i \\
    \log(1 + \exp( y^i\cdot \bm{\beta}^\top \bm{x}^i +  u)) - \lambda \cdot \kappa \leq s_i \\
    \alpha \cdot \lVert \bm{\beta} \rVert_{p^\star} \leq u.
\end{cases}
\end{align*}
Following the conic modeling guidelines of~\cite{mosek2016modeling}, for new variables $v^{+}_i, w^{+}_i \in \mathbb{R}$, the first constraint can be written as
\begin{align*}
    \begin{cases}
    v^{+}_i + w^{+}_i \leq 1, \ (v^{+}_i, 1, [-u + y^i \cdot \bm{\beta}^\top \bm{x}^i) - s_i] \in \mathcal{K}_{\exp}, \ (w^{+}_i, 1, -s_i) \in \mathcal{K}_{\exp},
    \end{cases}
\end{align*}
by using the definition of the exponential cone $\mathcal{K}_{\exp}$. Similarly, for new variables $v^{-}_i, w^{-}_i \in \mathbb{R}$, the second constraint can be written as 
\begin{align*}
    \begin{cases}
    v^{-}_i + w^{-}_i \leq 1, \ (v^{-}_i, 1, [-u - y^i \cdot \bm{\beta}^\top \bm{x}^i] - s_i - \lambda \cdot \kappa) \in \mathcal{K}_{\exp}, \ (w^{-}_i, 1, -s_i - \lambda \cdot \kappa) \in \mathcal{K}_{\exp}.
    \end{cases}
\end{align*}
Applying this for all $i \in [N]$ concludes that the following is the conic formulation of~\ref{advdro}:
\begin{align*}
    \begin{array}{cll}
        \underset{
        \substack{\bm{\beta},\ \lambda, \ \bm{s}, \ u \\ \bm{v}^{+}, \bm{w}^{+}, \bm{v}^{-}, \bm{w}^{-}}
        }{\mathrm{minimize}} & \displaystyle \lambda \cdot \varepsilon + \dfrac{1}{N} \sum_{i\in[N]} s_i &  \\
        \mathrm{subject\;to} & v^{+}_i + w^{+}_i \leq 1 & \forall i \in [N] \\[.6em]
        & (v^{+}_i, 1, [-u + y^i \cdot \bm{\beta}^\top \bm{x}^i] - s_i) \in \mathcal{K}_{\exp}, \ (w^{+}_i, 1, -s_i) \in \mathcal{K}_{\exp} & \forall i \in [N] \\[.6em]
        & v^{-}_i + w^{-}_i \leq 1& \forall i \in [N] \\[.6em]
        & (v^{-}_i, 1, [-u - y^i \cdot \bm{\beta}^\top \bm{x}^i] - s_i - \lambda \cdot \kappa) \in \mathcal{K}_{\exp}, \ (w^{-}_i, 1, -s_i- \lambda \cdot \kappa) \in \mathcal{K}_{\exp}& \forall i \in [N]  \\[.6em]
        & \alpha \cdot \lVert \bm{\beta} \rVert_{{p}^\star} \leq u \\[.6em]
        & \lVert \bm{\beta} \rVert_{q^\star} \leq \lambda & \\[.6em]
        & \bm{\beta} \in \mathbb{R}^n, \ \lambda \geq 0, \ \bm{s} \in \mathbb{R}^N, \ u \in \mathbb{R}, \ \bm{v}^{+}, \bm{w}^{+}, \bm{v}^{-}, \bm{w}^{-} \in \mathbb{R}^N.
    \end{array}
\end{align*}

\section{FURTHER DETAILS ON NUMERICAL EXPERIMENTS}
\subsection{UCI Experiments} \label{app_uai_data_sect}
\paragraph{Preprocessing UCI datasets.} We experiment on 10 UCI datasets \citep{kelly2023uci} (\textit{cf.}~Table \ref{tab:UCI_appendix}). We use Python 3 for preprocessing these datasets. Classification problems with more than two classes are converted to binary classification problems (most frequent class/others). For all datasets, numerical features are standardized, the ordinal categorical features are left as they are, and the nominal categorical features are processed via one-hot encoding. As mentioned in the main paper, we obtain auxiliary (synthetic) datasets via SDV, which is also implemented in Python 3.

\begin{table}[h]
\caption{Size of the UCI datasets.}
\label{tab:UCI_appendix}
\begin{center}
\begin{small}
\begin{tabular}{lcccc}
\toprule
DataSet &  $N$ & $\widehat{N}$ & $N_{\text{te}}$ & $n$ \\
\midrule
     {{absent}}  &  {111}  &  {333} & {296} & {74} \\
     {{annealing}}  &  {134}  &  {404} & {360} & {41} \\
     {{audiology}}  &  {33}  &  {102} & {91} & {102} \\
    {{breast-cancer}}  &  {102}  &  {307} & {274} & {90} \\
     {{contraceptive}}  &  {220}  &  {663} & {590} & {23} \\
     {{dermatology}}  &  {53}  &  {161} & {144} & {99} \\
     {{ecoli}}  &  {50}  &  {151} & {135} & {9} \\
     {{spambase}}  &  {690}  &  {2,070} & {1,841} & {58} \\
     {{spect}}  &  {24}  &  {72} & {64} & {23} \\
     {{prim-tumor}}  &  {50}  &  {153} & {136} & {32}  \\
\bottomrule
\end{tabular}
\end{small}
\end{center}
\vskip -0.1in
\end{table}

\paragraph{Detailed misclassification results on the UCI datasets.} Table \ref{appendix_large_table} contains detailed results on the out-of-sample error rates of each method on 10 UCI datasets for classification. All parameters are $5$-fold cross-validated: Wasserstein radii from the grid $\{ 10^{-6}, 10^{-5}, 10^{-4}, 10^{-3}, 10^{-2}, 10^{-1}, 0,1,2,5,10 \}$, $\kappa$ from the grid $\{1, \sqrt{n}, n\}$ the weight parameter of \texttt{ARO+Aux} from grid $\{10^{-6}, 10^{-5}, 10^{-4}, 10^{-3}, 10^{-2}, 10^{-1}, 0, 1\}$. We fix the norm defining the feature-label metric to the $\ell_1$-norm, and test $\ell_2$-attacks, but other choices with analogous results are also implemented.

\begin{table*}[!ht] 
\caption{Mean ($\pm$ std) out-of-sample errors of UCI datasets, each with 10 simulations. Results for adversarial ($\ell_2$-)attack strengths $\alpha = 0.05$ and $\alpha = 0.2$ are shared.}
\label{appendix_large_table}
\centering
\resizebox{\columnwidth}{!}{%
\begin{tabular}{lcccccc}
\toprule
Data   & $\alpha$ &  \texttt{ERM}  &  \texttt{ARO} &  \texttt{ARO+Aux} & \texttt{DRO+ARO} & \texttt{DRO+ARO+Aux}\\
\midrule 
     \multirow{2}*{{absent}}  & $0.05$ & 44.02\% ($\pm$ 2.89) & 38.82\% ($\pm$ 2.86) & 35.95\% ($\pm$ 3.78) &  34.22\% ($\pm$ 2.70) & \textbf{32.64\%} ($\pm$ 2.54)\\
     &   $0.20$ &  73.65\% ($\pm$ 4.14) & 51.49\% ($\pm$ 3.39) & 49.56\% ($\pm$ 3.80) &  45.61\% ($\pm$ 2.32) & \textbf{44.90\%} ($\pm$ 2.30)\\
     \multirow{2}*{{annealing}}  & $0.05$ &  18.08\% ($\pm$ 1.89) & 16.61\% ($\pm$ 2.16) &  14.97\% ($\pm$ 1.39) & 13.50\% ($\pm$ 2.98) & \textbf{12.78\%} ($\pm$ 2.78) \\
     &   $ 0.20$ & 37.31\% ($\pm$ 3.92) & 23.08\% ($\pm$ 2.82) & 21.30\% ($\pm$ 1.93) & 20.70\% ($\pm$ 1.32) & \textbf{19.53\%} ($\pm$ 1.42) \\
     \multirow{2}*{{audiology}}  &   $ 0.05$ &  21.43\% ($\pm$ 3.64) & 21.54\% ($\pm$ 3.92)& 17.03\% ($\pm$ 2.90) & 11.76\% ($\pm$ 3.28) & \textbf{\phantom{0}9.01\%} ($\pm$ 3.54) \\
     &   $ 0.20$ & 37.91\% ($\pm$ 6.78) & 29.34\% ($\pm$ 5.89)& 20.44\% ($\pm$ 2.75) & 20.00\% ($\pm$ 3.01) & \textbf{17.91\%} ($\pm$ 3.28) \\
    \multirow{2}*{{breast-cancer}}   & $0.05$ &  \phantom{0}4.74\% ($\pm$ 1.26) & \phantom{0}4.93\% ($\pm$ 1.75)& \phantom{0}3.87\% ($\pm$ 1.17)& \phantom{0}3.06\% ($\pm$ 0.79)& \textbf{\phantom{0}2.52\%} ($\pm$ 0.50) \\
     &   $ 0.20$ & \phantom{0}9.93\% ($\pm$ 1.73) & \phantom{0}8.14\% ($\pm$ 2.01)& \phantom{0}6.09\% ($\pm$ 1.79) & \phantom{0}5.04\% ($\pm$ 1.11)& \textbf{\phantom{0}4.67\%} ($\pm$ 0.99)\\
     \multirow{2}*{{contraceptive}}  & $0.05$ &  44.14\% ($\pm$ 2.80) & 42.86\% ($\pm$ 2.59)& 40.98\% ($\pm$ 0.95) & 40.00\% ($\pm$ 1.33) & \textbf{39.65\%} ($\pm$ 1.15) \\
     &   $ 0.20$ & 66.19\% ($\pm$ 5.97) & 43.49\% ($\pm$ 2.24)  & \textbf{42.71\%} ($\pm$ 1.47) &  \textbf{42.71\%} ($\pm$ 1.47) & \textbf{42.71\%} ($\pm$ 1.47) \\
     \multirow{2}*{{dermatology}}  & $ 0.05$ &  15.97\% ($\pm$ 2.64) & 16.46\% ($\pm$ 1.67) & 13.47\% ($\pm$ 1.97)& 12.78\% ($\pm$ 1.61) & \textbf{10.84\%} ($\pm$ 1.24)  \\
     &   $ 0.20$ & 30.07\% ($\pm$ 4.24) & 28.54\% ($\pm$ 3.25) & 21.53\% ($\pm$ 2.17) & 22.64\% ($\pm$ 2.15) & \textbf{20.21\%} ($\pm$ 1.58) \\
     \multirow{2}*{{ecoli}}  & $ 0.05$ & 16.30\% ($\pm$ 4.42) & 14.67\% ($\pm$ 5.13) & 13.26\% ($\pm$ 3.07) &  11.11\% ($\pm$ 5.52) & \textbf{\phantom{0}9.78\%} ($\pm$ 2.61)\\
     &   $ 0.20$ & 51.41\% ($\pm$ 3.37) & 42.67\% ($\pm$ 2.91) & 41.85\% ($\pm$ 2.95) & 39.70\% ($\pm$ 2.68) & \textbf{38.89\%} ($\pm$ 2.57)\\
     \multirow{2}*{{spambase}} & $0.05$ & 11.35\% ($\pm$ 0.77) & 10.23\% ($\pm$ 0.54) & 10.16\% ($\pm$ 0.56) & 9.83\% ($\pm$ 0.37) & \textbf{9.81\%} ($\pm$ 0.38)\\
     &   $0.20$ & 27.32\% ($\pm$ 2.11) & 15.83\% ($\pm$ 0.77) & 15.70\% ($\pm$ 0.76) & 15.67\% ($\pm$ 0.72) & \textbf{15.50\%} ($\pm$ 0.68)\\
     \multirow{2}*{{spect}}  & $0.05$ &  33.75\% ($\pm$ 5.17) & 29.69\% ($\pm$ 5.46) &  25.78\% ($\pm$ 3.06) & 25.47\% ($\pm$ 3.38) & \textbf{21.56\%} ($\pm$ 2.74) \\
     &  $0.20$ & 54.22\% ($\pm$ 9.88) & 37.5\% ($\pm$ 3.53) & 35.16\% ($\pm$ 2.47) & 33.75\% ($\pm$ 2.68) & \textbf{30.16\%} ($\pm$ 3.61) \\
     \multirow{2}*{{prim-tumor}}  & $0.05$ & 21.84\% ($\pm$ 4.55) & 20.81\% ($\pm$ 3.97) & 17.35\% ($\pm$ 3.59) & 16.18\% ($\pm$ 3.83) & \textbf{14.78\%} ($\pm$ 2.89) \\
     &   $0.20$ & 34.19\% ($\pm$ 6.17) & 25.37\% ($\pm$ 4.58)  & 21.62\% ($\pm$ 3.45) & 21.84\% ($\pm$ 3.34)  & \textbf{19.63\%} ($\pm$ 2.71) \\
\bottomrule
\end{tabular}
}
\end{table*}

Finally, we demonstrate that our theory, especially \texttt{DRO+ARO+Aux}, contributes to the DRO literature even without adversarial attacks. In this case of $\alpha = 0$, \texttt{ERM} and \texttt{ARO} would be equivalent, and \texttt{DRO+ARO} would reduce to the traditional DR LR model~\citep{NIPS2015}. \texttt{ARO+Aux} would be interpreted as revising the empirical distribution of ERM to a mixture (mixture weight cross-validated) of the empirical and auxiliary distributions. \texttt{DRO+ARO+Aux}, on the other hand, can be interpreted as DRO over a carefully reduced ambiguity set (intersection of the empirical and auxiliary Wasserstein balls). The results are in Table~\ref{table:UCI_no_attack}. Analogous results follow as before (that is, \texttt{DRO+ARO+Aux} is the `winning' approach, \texttt{DRO+ARO} and \texttt{ARO+Aux} alternate for the `second' approach), with the exception of the dataset \textit{contraceptive}, where \texttt{ARO+Aux} outperforms others.

\paragraph{Cross-validated Wasserstein radii.} In Corollary~\ref{corr:relax}, we discussed the feasibility of ignoring the auxiliary data when its data-generating distribution is distant from the true data-generating distribution. To examine whether large values of $\widehat{\varepsilon}$ are selected via cross-validation in our experiments, we investigated their histograms and observed that this indeed occurs frequently. For example, as shown in Table~\ref{table:UCI_no_attack}, when there are no adversarial attacks, \texttt{DRO+ARO} and \texttt{DRO+ARO+Aux} yield identical errors on the `annealing' dataset. This is because the two models become equivalent: \texttt{DRO+ARO+Aux} selects a large $\widehat{\varepsilon}$ (either $5$ or $10$), which ensures that the intersection of the Wasserstein balls reduces to the empirical distribution. With $\alpha = 0.05$, \texttt{DRO+ARO+Aux} selects a large $\widehat{\varepsilon}$ in 7 out of 10 simulations, while in the remaining 3 it selects a smaller value, resulting in a nontrivial intersection and, in turn, improved performance over \texttt{DRO+ARO}. We conclude that our method selects a nontrivial $\widehat{\varepsilon}$ only when there is evidence that the auxiliary data is \textit{useful}, that is, when its data-generating distribution is sufficiently close to the true one. For instance, on the `absent' dataset, we always have $\widehat{\varepsilon} \in \{10^{-2}, 10^{-1}\}$. We also revised our numerical experiments by modifying the Gaussian copula synthesizer to enforce uniform marginals (which results in significant information loss), and in this setting, our models consistently selected large $\widehat{\varepsilon}$ values.

\begin{table*}[!t]
\caption{Mean out-of-sample errors of UCI experiments without adversarial attacks.}
\label{table:UCI_no_attack}
\centering
\begin{tabular}{lcccccc}
\toprule
Data   &  \texttt{ERM}  &  \texttt{ARO} & \texttt{ARO+Aux} & \texttt{DRO+ARO} & \texttt{DRO+ARO+Aux} \\
\midrule 
        {absent} & 36.28\%  & 36.28\%  & 31.86\%  &  28.31\%  & \textbf{27.74\%}\\
     {annealing} & 10.61\%  & 10.61\%  &  \phantom{0}7.64\%  & \textbf{\phantom{0}7.14\%}  & \textbf{\phantom{0}7.14\%}  \\
  {audiology} & 14.94\%  & 14.94\% & 12.97\%  & 10.11\%  & \textbf{\phantom{0}7.69\%}  \\
      {breast-cancer} & \phantom{0}6.64\%  & \phantom{0}6.64\% & \phantom{0}5.22\% & \phantom{0}2.55\% & \textbf{\phantom{0}2.15\%}  \\
     {contraceptive} & 35.00\%  & 35.00\% & \textbf{33.75\%} & 34.56\% & 33.85\%  \\
     {dermatology} & 16.04\%  & 16.04\% & 11.60\% & \phantom{0}9.93\% & \textbf{\phantom{0}8.06\%} \\
      {ecoli} & \phantom{0}6.74\% & \phantom{0}6.74\% & \phantom{0}4.96\% & \phantom{0}5.19\% & \textbf{\phantom{0}4.37\%} \\
      spambase & \phantom{0}8.95\% & \phantom{0}8.95\% & \phantom{0}8.52\% & \phantom{0}8.34\% & \textbf{\phantom{0}8.16\%} \\
      spect & 30.74\% & 30.74\% & 24.69\% & 22.35\% & \textbf{18.75\%} \\ 
      {prim-tumor} & 22.79\% & 22.79\% & 17.28\% & 15.07\% & \textbf{13.97\%}\\
\bottomrule
\end{tabular}
\end{table*}

\paragraph{Different training/auxiliary dataset ratio.}
Recall that in the UCI experiments, we sampled $15\%$ of the original dataset as the training set, and used $45\%$ to generate a synthetic auxiliary dataset. This setup was chosen to simulate scenarios with limited training data, where overfitting is a particular concern. If we reverse these proportions and use $45\%$ of the original dataset for training and $15\%$ for generating auxiliary data instead, we find that the best-performing method remains \texttt{DRO+ARO+Aux}, and the relative ranking among the models remains qualitatively similar. One notable difference is that the improvement of the auxiliary-data-oblivious \texttt{DRO+ARO} method over the non-DR \texttt{ARO+Aux} model becomes less significant, and is even reversed in the case of $\alpha = 0.05$ on the `ecoli' dataset: \texttt{ERM} ($14.59\%$), \texttt{ARO} ($11.41\%$), \texttt{ARO+Aux} ($9.04\%$), \texttt{DRO+ARO} ($9.41\%$), \texttt{DRO+ARO+Aux} ($8.89\%$). As expected, since more data is drawn from the true data-generating distribution, all methods exhibit improved out-of-sample performance compared to the original setup. However, in this case, \texttt{DRO+ARO} no longer outperforms \texttt{ARO+Aux}.

\subsection{MNIST/EMNIST Experiments}\label{sec:new_referee}
The setting in the MNIST/EMNIST experiments is similar to that in the UCI experiments. However, for auxiliary data, we use the EMNIST dataset which we accessed via the MLDatasets package of Julia.

Moreover, in \S\ref{section:related} we reviewed the literature showing that when statistical error is not a concern, that is, when optimizing over the empirical distribution cannot cause overfitting (\textit{e.g.}, in high-data regimes), then adversarial training is equivalent to a type-$\infty$ Wasserstein DRO problem with radius $\varepsilon = \alpha$. Hence, a natural question is whether increasing the value of $\varepsilon$ further also provides distributional robustness. To this end, in Table~\ref{tab:mnist_2}, we revise Table~\ref{tab:mnist} and add an additional benchmark \texttt{DRO++}. Here, we take $\varepsilon = \alpha + \varepsilon'$, and cross-validate $\varepsilon'$ from the same grid that we cross-validated $\varepsilon$ for methods \texttt{DRO+ARO} and \texttt{DRO+ARO+Aux}. We observe that \texttt{DRO++} does not improve over \texttt{DRO+ARO}, which is expected given that type-$\infty$ Wasserstein DRO does not provide better generalizations, unlike type-$1$ Wasserstein DRO (\textit{cf.} the discussion at the end of \S\ref{section:prelim}). Yet, this method improves over \texttt{ARO} in all cases, and even over \texttt{ARO+Aux} in the no-attack or $\ell_\infty$-attack settings.

\begin{table}
\caption{Additional MNIST/EMNIST Benchmark.}
    \centering
\begin{tabular}{lccccccc}
\toprule
Attack   &  \texttt{ERM}  &  \texttt{ARO} & \texttt{ARO+Aux} & \texttt{DRO++} & \texttt{DRO+ARO} & \texttt{DRO+ARO+Aux} \\
\midrule 
    No attack ($\alpha = 0$) & \phantom{10}1.55\%  & \phantom{10}1.55\%  & \phantom{10}1.19\%  & \phantom{10}0.72\% &   \phantom{10}0.64\%  & \phantom{10}\textbf{0.53\%}\\
    $\ell_1$ ($\alpha = 68/255$) & \phantom{10}2.17\% & \phantom{10}1.84\% & \phantom{10}1.33\% & \phantom{10}1.40\% & \phantom{10}0.66\% & \phantom{10}\textbf{0.57\%} \\
    $\ell_2$ ($\alpha = 128/255$) & \phantom{1}99.93\% & \phantom{10}3.36\% & \phantom{10}2.54\% & \phantom{10}2.72\% & \phantom{10}2.40\% & \phantom{10}\textbf{2.12\%}  \\
    $\ell_\infty$ ($\alpha = 8/255$) & 100.00\% & \phantom{10}2.60\% & \phantom{10}2.38\% & \phantom{10}2.31\%  & \phantom{10}2.20\% & \phantom{10}\textbf{1.95\%}  \\
\bottomrule
\end{tabular}
\label{tab:mnist_2}
\end{table}

\subsection{Artificial Experiments}\label{appendix_artificial}
\paragraph{Data generation.} We sample a `true' $\bm{\beta}$ from a unit $\ell_2$-ball, and generate data as summarized in Algorithm~\ref{alg:artificial}. Such a dataset generation gives $N$ instances from the same true data-generating distribution. In order to obtain $\hat{N}$ auxiliary dataset instances, we perturb the probabilities $p^i$ with standard random normal noise which is equivalent to sampling i.i.d.\@ from a \textit{perturbed} distribution. Testing is always done on true data, that is, the test set is sampled according to Algorithm~\ref{alg:artificial}.
\begin{algorithm}[!h]
\caption{Data from a ground truth logistic classifier}
\textbf{Input:}{ set of feature vectors $\bm{x}^i, \ i \in [N]$; vector $\bm{\beta}$}
    \begin{algorithmic}
        \FOR{$i \in \{ 1, \ldots, N \}$}
            \STATE Find the 
            probability  
            $p^i = \left[ 1 + \exp (- \bm{\beta}^\top \bm{x}) \right]^{-1}.$
            \STATE Sample $u=\mathcal{U}(0,1)$
            \IF {$p^i \geq u$}
                \STATE $y^i = +1$
                \ELSE 
                \STATE $y^i = -1$
            \ENDIF
        \ENDFOR
    \end{algorithmic}
     \textbf{Output}: $(\bm{x}^i, y^i), \ i \in [N]$.
     \label{alg:artificial}
\end{algorithm}

\paragraph{Strength of the attack and importance of auxiliary data.} In the main paper we discussed how the strength of an attack determines whether using auxiliary data in ARO (\texttt{ARO+Aux}) or considering distributional ambiguity (\texttt{DRO+ARO}) is more effective, and observed that unifying them to obtain \texttt{DRO+ARO+Aux} yields the best results in all attack regimes. Now we focus on the methods that rely on auxiliary data, namely \texttt{ARO+Aux} and \texttt{DRO+ARO+Aux} and explore the importance of auxiliary data ${\widehat{\mathbb{P}}}_{\widehat{N}}$ in comparison to its empirical counterpart $\mb P_N$. Table \ref{table_CVd_values} shows the average values of $w$ for problem~\eqref{synth_literature} obtained via cross-validation. We see that the greater the attack strength is the more we should use the auxiliary data in \texttt{ARO+Aux}. The same relationship holds for the average of $\varepsilon / \widehat{\varepsilon}$ obtained via cross-validation in~\ref{synth}, which means that the relative size of the Wasserstein ball built around the empirical distribution gets larger compared to the same ball around the auxiliary data, that is, ambiguity around the auxiliary data is smaller than the ambiguity around the empirical data. We highlight as a possible future research direction exploring when a larger attack \textit{per se} implies the intersection 
will move towards the auxiliary data distribution.

\begin{table}
\caption{Mean $w$ in problem~\eqref{synth_literature} and $\varepsilon/\widehat{\varepsilon}$ in problem~\ref{synth} across 25 simulations of cross-validating $\omega$, $\varepsilon$, and $\widehat\varepsilon$.} \label{table_CVd_values}
\centering
\begin{tabular}{lcc}
\toprule
Attack & \texttt{ARO+Aux} (cross-validated $w$)  & \texttt{DRO+ARO+Aux} (cross-validated $\varepsilon/\widehat{\varepsilon}$)  \\
\midrule
$\alpha = 0\phantom{.00}$ & $0.002$ & $0.0120$ \\
$\alpha = 0.1\phantom{0}$ & $0.046$ & $0.172$ \\
            $\alpha = 0.25$ & $0.086$ & $0.232$\\
            $\alpha = 0.5\phantom{0}$ & $0.290$ & $0.241$\\
\bottomrule
\end{tabular}
\end{table}

\paragraph{More results on scalability.} We further simulate $25$ cases with an $\ell_2$-attack strength of $\alpha = 0.2$, $N = 200$ instances in the training dataset, $\widehat{N} = 200$ instances in the auxiliary dataset, and we vary the number of features $n$. We report the median ($50\% \pm 15\%$ quantiles shaded) runtimes of each method in Figure \ref{fig:appendix_runtimes}. 
The fastest methods are \texttt{ERM} and \texttt{ARO} among which the faster one depends on $n$ (as the adversarial loss includes a regularizer of $\bm{\beta}$), followed by \texttt{ARO+Aux}, \texttt{DRO+ARO}, and \texttt{DRO+ARO+Aux}, respectively. \texttt{DRO+ARO+Aux} is the slowest, which is expected given that \texttt{DRO+ARO} is its special for large $\widehat{\varepsilon}$. The runtime however scales graciously.
\begin{figure*}[!ht]
    \centering
    \includegraphics[width=0.495\textwidth]{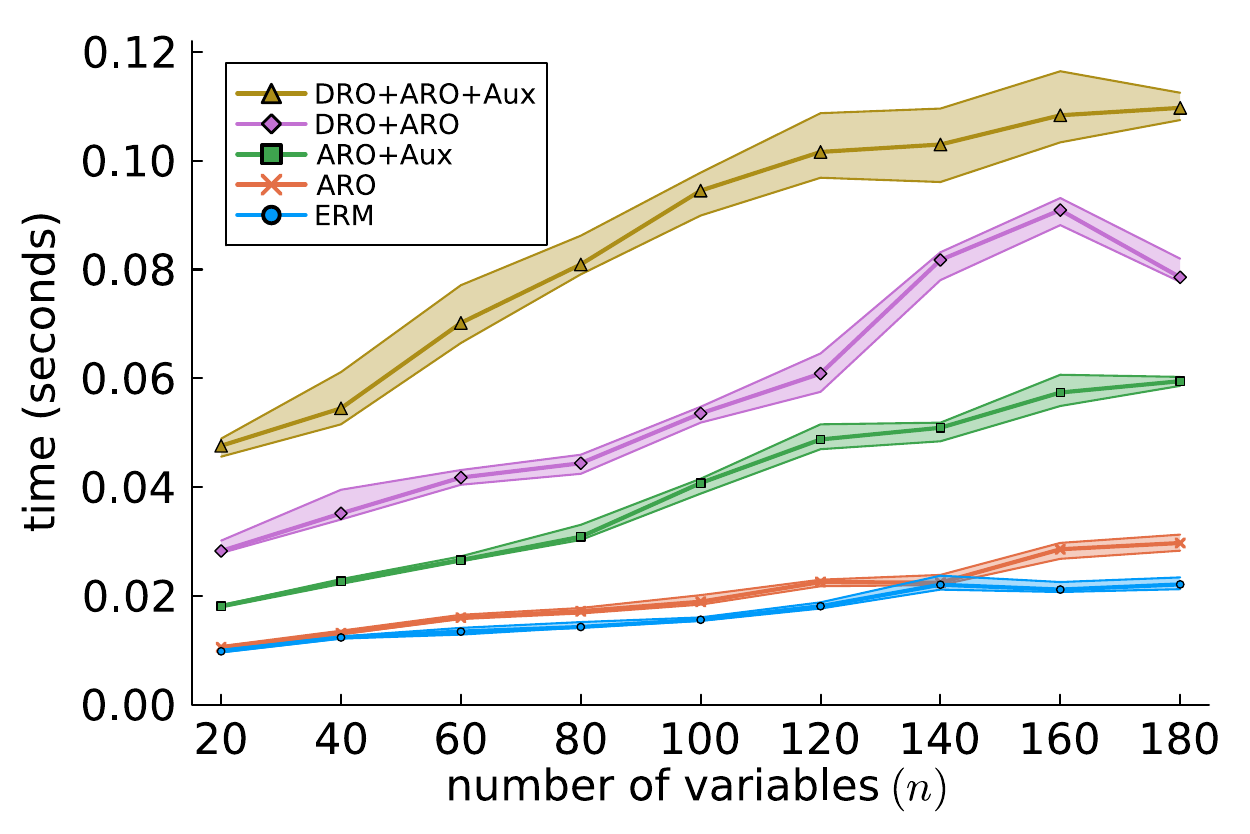} 
    \caption{\textit{Runtimes under a varying number of features in the artificially generated empirical and auxiliary datasets.}}
    \label{fig:appendix_runtimes}
\end{figure*}

Finally, we focus further on \texttt{DRO+ARO+Aux} which solves problem~\ref{synth} with $\mathcal{O}(n \cdot N \cdot \widehat{N})$ variables and exponential cone constraints. For $n = 1,000$ and $N = \widehat{N} = 10,000$, we observe that the runtimes vary between 134 to 232 seconds across 25 simulations.

\end{document}